\documentclass[a4paper,12pt,reqno]{amsart}

\usepackage{amsmath,amssymb,amsthm}
\usepackage{latexsym}
\usepackage{color}
\usepackage{graphicx}
\usepackage{mathrsfs}
\usepackage{enumerate}
\usepackage[abbrev]{amsrefs}
\usepackage[T1]{fontenc}
\usepackage{mathtools}
\mathtoolsset{showonlyrefs=true}
\usepackage[colorlinks,
            linkcolor=blue,      
           anchorcolor=blue,  
           citecolor=red,       
           ]{hyperref}
\usepackage{hyperref}

\allowdisplaybreaks[4]

\theoremstyle{plain}
\newtheorem{thm}{Theorem}[section]
\newtheorem{lemm}[thm]{Lemma}

\newtheorem{cor}[thm]{Corollary}

\theoremstyle{definition}
\newtheorem{df}[thm]{Definition}
\newtheorem{rem}[thm]{Remark}

\makeatletter

\@addtoreset{equation}{section}
\makeatother

\newcommand{\dB}{\dot{B}}

\newcommand{\supp}{\operatorname{supp}}

\renewcommand{\leq}{\leqslant}
\renewcommand{\geq}{\geqslant}

\newcommand{\pnabla}{\widetilde{\nabla}}

\newcommand{\phih}{\phi^{\rm h}}
\newcommand{\phiv}{\phi^{\rm v}}
\newcommand{\Deltah}{\Delta_{\rm h}}
\newcommand{\nablah}{\nabla_{\rm h}}
\newcommand{\Ph}{\mathbb{P}_{\rm h}}
\newcommand{\xh}{x_{\rm h}}

\newcommand{\xih}{\xi_{\rm h}}
\newcommand{\Deltahh}{\Delta^{\rm h}}
\newcommand{\Deltav}{\Delta^{\rm v}}
\newcommand{\R}{\mathbb{R}}
\newcommand{\Grad}{\nabla}

\newcommand{\f}{\frac}

\newcommand{\n}[1]{{\left\|#1\right\|}}

\newcommand{\lp}[1]{\left[#1\right]}
\newcommand{\Mp}[1]{\left\{#1\right\}}
\renewcommand{\sp}[1]{\left(#1\right)}
\newcommand{\alphah}{\alpha_{\rm h}}
\newcommand{\vh}{v_{\rm h}}
\newcommand{\ep}{\varepsilon}
\newcommand{\p}{\partial}

\newcommand{\Z}{\mathbb{Z}}

\begin{document}
\title[$3$D anisotropic Boussinesq flow with stratification]
{Linear and nonlinear stability \\ for the $3$D stratified Boussinesq equations \\ with the horizontal viscosity and diffusivity}

\author[M.Fujii]{Mikihiro Fujii}
\author[Y.Li]{Yang Li}
\address[M.Fujii]{Graduate School of Science, Nagoya City University, Nagoya, 467-8501, Japan}
\email[M.Fujii]{fujii.mikihiro@nsc.nagoya-cu.ac.jp}
\address[Y.Li]{School of Mathematical Sciences and Center of Pure Mathematics, Anhui University, Hefei, 230601, People's Republic of China}
\email[Y.Li]{lynjum@163.com}
\begin{abstract}
In this manuscript, we consider the $3$D Boussinesq equations for stably stratified fluids with the horizontal viscosity and thermal diffusivity and investigate the large time behavior of the solutions.
Making use of the anisotropic Littlewood--Paley theory, we obtain their precise $L^1$-$L^p$ decay estimates, which provide us information on both the anisotropic and dispersive structure of the system. 
More precisely, we reveal that the dispersion from the skew symmetric terms of stratification makes the decay rates of some portions of the solutions faster and furthermore the third component of the velocity field exhibit the enhanced dissipative effect, which provides the additional fast decay rate.  
\end{abstract}
\maketitle

\tableofcontents

\section{Introduction}
Let us consider the initial value problem for the Navier--Stokes equations taking into account temperature variations under the Boussinesq approximation. 
In many applications like the rotating stably stratified fluids in Ekman layers, it is reasonable to model the fluid motions by anisotropic equations. 
By neglecting the effect of vertical dissipation, the fluid equations concerned here only admit the horizontal dissipation and read as 
\begin{align}\label{eq:B1}
    \begin{cases}
        \partial_t v - \nu \Deltah v + (v \cdot \nabla)v + \nabla \Pi = \Theta e_3, \qquad & t>0, x \in \mathbb{R}^3,\\
        \partial_t \Theta - \kappa \Deltah \Theta + (v \cdot \nabla) \Theta = 0, & t >0,x \in \mathbb{R}^3,\\
        \nabla \cdot v = 0, & t \geq 0,  x \in \mathbb{R}^3,\\
        \Theta(0,x) = \Theta_0(x),\quad v(0,x) = v_0(x), & x \in \mathbb{R}^3,
    \end{cases}
\end{align}
where $v=v(t,x) : [0,\infty) \times \mathbb{R}^3 \to \mathbb{R}^3$, $\Pi=\Pi(t,x) : (0,\infty) \times \mathbb{R}^3 \to \mathbb{R}$, and $\Theta=\Theta(t,x) : [0,\infty) \times \mathbb{R}^3 \to \mathbb{R}$ represent the unknown velocity field, pressure, and temperature of the fluid, respectively,
whereas given functions $v_0=v_0(x) :  \mathbb{R}^3 \to \mathbb{R}^3$ and $\Theta_0=\Theta_0(x) :  \mathbb{R}^3 \to \mathbb{R}$ are initial data. 
We denote by $e_3=(0,0,1)^{T}$ the vertical unit vector. 
Two positive constants $\nu$ and $\kappa$ represent the kinematic viscosity coefficient and heat conductivity coefficient, respectively. To simplify the presentation, we assume that $\nu=\kappa=1$. The operator $\Deltah=\partial_{x_1}^2 + \partial_{x_2}^2$ is the Laplacian only for the horizontal variables. 
It is well-known that \eqref{eq:B1} possesses the following explicit stationary solutions 
\begin{align}
    \widetilde{v}(x)
    =0,\qquad
    \widetilde{\Theta}(x)
    =
    x_3,\qquad
    \widetilde{\Pi}(x)
    :=
    \frac{1}{2}x_3^2.
\end{align}
We consider the perturbation 
\begin{align}
    \theta:= {\Theta - \widetilde{\Theta}},\qquad
    p := \Pi - \widetilde{\Pi}
\end{align}
of $\Theta$ and $\Pi$
from $\widetilde{\Theta}$ and $\widetilde{\Pi}$, respectively.
Then, the velocity $v$ and the thermal disturbance $\theta$ should solve the following equations
\begin{align}\label{eq:B2}
    \begin{cases}
        \partial_t v -  \Deltah v -  \theta e_3 
        + (v\cdot \nabla) v + \nabla p = 0, \qquad & t > 0, x \in \mathbb{R}^3,\\
        \partial_t \theta -  \Deltah \theta +  v_3 
        + (v\cdot \nabla) \theta = 0, & t > 0, x \in \mathbb{R}^3,\\
        \nabla \cdot  v = 0, & t \geq 0, x \in \mathbb{R}^3,\\
        v(0,x)=v_0(x),\quad \theta(0,x)=\theta_0(x), & x \in \mathbb{R}^3,
    \end{cases}
\end{align} 
where $\theta_0:=\Theta_0-\widetilde{\Theta}$.

The aim of this paper is to derive the precise $L^1$-$L^p$ decay estimates of the solutions for $2 \leq p \leq \infty$, which provide us information on the anisotropic and dispersive structure of the system. More precisely, we prove that the skew symmetric terms $-\theta e_3$ and $v_3$ provide the dispersive phenomenon that affect the decay rate of some portions of the solutions more faster than the $2$D heat kernel. 
Furthermore, the third component of the velocity field $v_3$ exhibits the enhanced dissipative effect, which was confirmed for the anisotropic Navier--Stokes equations \eqref{eq:ANS} below by \cites{F21,XZ22}.

Before focusing on our main results precisely, we recall some known results related to our study.
For the case of $\Theta \equiv 0$, the system \eqref{eq:B1} is reduced to the anisotropic Navier--Stokes equations:
\begin{align}\label{eq:ANS}
    \begin{cases}
        \partial_t v -  \Deltah v 
        + (v\cdot \nabla) v + \nabla p = 0, \qquad & t > 0, x \in \mathbb{R}^3,\\
        \nabla \cdot  v = 0, & t \geq 0, x \in \mathbb{R}^3,\\
        v(0,x)=v_0(x), & x \in \mathbb{R}^3.
    \end{cases}
\end{align}  
It was Chemin, Desjardins, Gallagher, and Grenier \cite{CDGG01} who first constructed the local and global solutions to \eqref{eq:ANS} in anisotropic scaling sub-critical Sobolev spaces, and later their results were improved to the scaling critical anisotropic Besov spaces framework; see \cites{CZ07,Pai05,ZF1,Li-Pai-Zha-20}. 
For the asymptotic behavior of the global solutions,  
Ji, Wu, and Yang \cite{JWY21} proved that the $H^4(\mathbb{R}^3)$ solutions behave as $2$D heat kernel in the sense of $L^2$-decay rate.
Xu and Zhang \cite{XZ22} and the first author \cite{F21} revealed that the horizontal component of the velocity field decays as $2$D heat kernel, while the vertical component behaves as $3$D heat kernel.
More precisely, \cite{F21} proved that the small solutions $v=(v_1,v_2,v_3) = (\vh,v_3)$\footnote{Throughout this paper, we use the notation $a_{\rm h}=(a_1,a_2)$ for a $3$D vector $a=(a_1,a_2,a_3) \in \mathbb{R}^3$ and call $a_{\rm h}$ as the horizontal component of $a$.} in $H^8(\mathbb{R}^3) \cap L^1(\mathbb{R}^2_{\xh};(W^{1,1}\cap W^{1,\infty})(\mathbb{R}_{x_3}))$ decays as 
\begin{align} 
    \n{\nabla^{\alpha} \vh(t)}_{L^p(\R^3)}
    & = O\sp{t^{-(1-\frac{1}{p})-\frac{|\alphah|}{2}}}, \nonumber 
    \\       
    \n{\nablah^{\alphah} v_3(t)}_{L^p(\R^3)} &
    = O\sp{t^{-\frac{3}{2}(1-\frac{1}{p})-\frac{|\alphah|}{2}}} \label{ANS-p} 
\end{align}
for $1\leq p \leq \infty$ and $\alpha = (\alphah,\alpha_3) \in (\mathbb{N} \cup \{0 \})^2 \times (\mathbb{N} \cup \{0 \})$ with $|\alpha|\leq 1$ and also showed that the above decay rates are sharp by establishing the precise asymptotic profile. See \cite{L22} for the related study on the anisotropic MHD system and \cite{FL-24} for the corresponding results to \cite{F21} on the half space case.

Next, we focus on the decay estimates for the $3$D anisotropic Boussinesq equations \eqref{eq:B2}. 
Wu and Zhang \cite{Wu-Zha-21} proved the stability and decay estimates of solutions to $3$D anisotropic Boussinesq equations with horizontal dissipation and vertical thermal diffusion in $\R^2 \times \mathbb{T}$. More precisely, under the smallness and certain symmetry conditions of initial data, they showed that global solution converges to the limiting system exponentially. 
Under the smallness assumption 
\begin{align}
\n{(v_0,\theta_0)}_{H^4(\R^3)}
+{}
\n{ |\nabla_{\rm h}|^{-\sigma}(v_0, \theta_0 )   }_{L^2(\R^3)}
+
\n{ \p_{x_3}|\nabla_{\rm h}|^{-\sigma}(v_0, \theta_0 )   }_{L^2(\R^3)}
\ll 1
\end{align} 
for $3/4\leq \sigma<1$, Ji, Yan, and Wu \cite{Ji-Yan-Wu-22} proved that the global solution $(v,\theta)$ to \eqref{eq:B2} satisfies
\begin{align}\label{decay:JWY}
    \n{ \nabla^{\alpha}(v,\theta) (t) }_{L^2(\R^3)}
    =
    O \sp{t^{-\f{\sigma}{2}-\frac{|\alphah|}{2}}}   
\end{align} 
as $t \to \infty$
for $\alpha \in (\mathbb{N} \cup \{ 0 \})^3$ with $|\alpha|\leq 1$.
This result implies that the solutions to \eqref{eq:B2} behave like the $2$D heat kernel in the sense of $L^2$-decay rate.
We also refer to \cites{Ben-Pan-Wu-22,Cha-06,Den-Wu-Zha-21,Don-Wu-Xu-Zhu-21,Hou-Li-05,Kan-Lee-Ngu-24,Lai-Wu-Zho-21,Lar-Lun-Titi-13,Li-Titi-16,Tao-Wu-Zha0-Zhen-20,Zil-21} for more results on the global existence of a unique solution, stability and the decay estimates for $2$D Boussinesq equations with partial dissipation and thermal diffusion.

The goal of this paper is to improve the decay estimates \eqref{decay:JWY} by \cite{Ji-Yan-Wu-22} and reveal in detail how the dispersion and anisotropy of the system \eqref{eq:B2} affect the large time behavior of its solutions. 
Let us mention the novelty of the present article more precisely.
Due to the skew symmetric terms $-\theta e_3$ and $+ v_3$ in the first and second equation of \eqref{eq:B2}, respectively, the linearized solution to \eqref{eq:B2} presents the dispersive effect; indeed the linear solution formula \eqref{lin-sol-1} contains the dispersive semigroup $\{e^{t\Deltah}e^{\pm i t \frac{|\nablah|}{|\nabla|}}\}_{t\geq 0}$.
Using the precise analysis by the anisotropic Littlewood--Paley theory, we take this dispersion into account of our linear and nonlinear analysis and show that the $L^p$-decay estimates of the horizontal curl free part $\vh - \Ph \vh$\footnote{We denote by $\Ph$ the horizontal Helmholtz projection; see Sections \ref{main-res} and \ref{lin-ana} for the definition.} and $\theta$ are faster than the $2$D heat kernel since their decay consist of both the $2$D heat decay $O(t^{-(1-\frac{1}{p})})$ and dispersive decay $O(t^{-\frac{3}{4}(1-\frac{2}{p})})$.
Furthermore, the vertical component $v_3$ of the velocity field has not only $2$D heat decay and the dispersive decay but also the additional decay $O(t^{-\frac{1}{4}})$, which comes from the enhanced dissipation mechanism that is investigated by \cites{F21,FL-24,XZ22} for \eqref{eq:ANS}.

The rest of this paper is arranged as follows. 
In Section \ref{main-res}, we present the main results mentioned above by separating linear estimates and nonlinear estimates. 
In Section \ref{sec:pre}, we prepare some useful tools on the anisotropic Littlewood--Paley analysis. 
In Section \ref{lin-ana}, we give the decay estimates for the linearized equations, which plays the key role in the whole analysis. 
Section \ref{sec:non} is devoted to the decomposition of nonlinear term, the decay estimates of Duhamel terms and the proof of Theorem \ref{main-thm}.

\section{Main results}\label{main-res}
In this section, we provide the precise statements of our two main theorems.
The first main theorem mentions the $L^1$-$L^p$ decay estimates for the solutions to the linearized system of \eqref{eq:B-lin-0} below.
In the second theorem, we focus on the decay results for the small nonlinear solutions to \eqref{eq:B2}.

In order to state our main results, we introduce some notations. 
Throughout this paper, we denote by $C$ the constant, which may differ in each line. In particular, $C=C(*,...,*)$ means that $C$ depends only on the qualities in parentheses.
For two non-negative numbers $A$, $B$, the relation $A \sim B$ means that there exists a positive constant $C$ such that $C^{-1}A\leqslant B \leqslant C A$ holds.
For a given $3$D vector field $a=(a_1,a_2,a_3) \in \mathbb{R}^3$, we call $a_{\rm h}:=(a_1,a_2)$ as the horizontal component of $a$.
For $1 \leq p_1,p_2 \leq \infty$ and $s \in \mathbb{N}$, we use the following abbreviation:
\begin{align}
    \n{f}_{L^{p_1}_{\xh}L^{p_2}_{x_3}(\R^3)}
    :={}&
    \n{f}_{L^{p_1}(\mathbb{R}^2_{\xh};L^{p_2}(\mathbb{R}_{x_3}))},\\
    \n{f}_{L^{p_1}_{\xh}W^{s,p_2}_{x_3}(\R^3)}
    :={}&
    \n{f}_{L^{p_1}(\mathbb{R}^2_{\xh};W^{s,p_2}(\mathbb{R}_{x_3}))}.
\end{align}
For $s_1,s_2\in \mathbb{N}$, we set 
\begin{align}
    X^{s_1,s_2}(\R^3):= 
    H^{s_1}(\R^3) \cap L^1(\mathbb{R}^2_{\xh}; W^{s_2,1}  (\R_{x_3}))
\end{align} 
with the associated norms 
\begin{align}
    \n{f}_{  X^{s_1,s_2}(\R^3)   } :={}&
    \n{f}_{  H^{s_1}(\R^3)    }
    +
    \n{f}_{L^1_{\xh}W^{s_2,1}_{x_3}(\R^3)}\\
    ={}&
    \sum_{|\alpha|\leq s_1}  
    \n{\nabla^{\alpha}f}_{L^2(\mathbb{R}^3)}
    + \sum_{j=0}^{s_2}  \n{\partial_{x_3}^jf}_{L^1(\R^3)}.
\end{align} 
\subsection{Linear stability}
Let us consider the solutions to the following linear system:
\begin{align}\label{eq:B-lin-0}
    \begin{cases}
        \partial_t v^{\rm lin} -  \Deltah v^{\rm lin} -  \theta^{\rm lin} e_3 
        +  \nabla p^{\rm lin} = 0, \qquad & t > 0, x \in \mathbb{R}^3,\\
        \partial_t \theta^{\rm lin} -  \Deltah \theta^{\rm lin} +  v^{\rm lin}_3 
         = 0, & t > 0, x \in \mathbb{R}^3,\\
        \nabla \cdot  v^{\rm lin} = 0, & t \geq 0, x \in \mathbb{R}^3,\\
        v^{\rm lin}(0,x)=v_0(x),\quad \theta^{\rm lin}(0,x)=\theta_0(x), & x \in \mathbb{R}^3.
    \end{cases}
\end{align} 
Our first goal is to obtain the $L^p$-decay estimates of solutions to \eqref{eq:B-lin-0}. 
\begin{thm}\label{thm:lin}
Let $0<\ep <1/4$.
Then, there exist an absolute positive constant $C$ and a positive constant $C_{\ep}$ depending only on $\ep$ such that the linear solution $(v^{\rm lin},\theta^{\rm lin})$ to \eqref{eq:B-lin-0}, with the initial data $(v_0,\theta_0) \in X^{3,4}(\R^3)$ with $\nabla \cdot v_0 = 0$, satisfies
\begin{align}
    &
    \n{ \Grad^{\alpha} \Ph v^{\rm lin}_{\rm h} (t)  }_{L^p(\R^3)}  
    \leq 
    C
    (1+t)^{ -(1-\frac{1}{p}) - \f{|\alphah|}{2} }
     \n{(v_0,\theta_0)}_{ 
     {  
     X^{2+|\alphah|,2+\alpha_3}(\R^3) }
    }  ,  
    \\
    & 
    \n{ \Grad^{\alpha} ( v^{\rm lin}_{\rm h}-\Ph \vh^{\rm lin}) (t)  }_{L^p(\R^3)}  
    \leq 
    C
    (1+t)^{ -(1-\frac{1}{p}) - \f{|\alphah|}{2} -\f{3}{4} (1-\f{2}{p} )  }
    { \n{(v_0,\theta_0)}_{ 
    X^{2+|\alphah|,3+\alpha_3}(\R^3)
    }  },  
    \\
    & 
    \n{  \nablah^{\alphah}  v_{3}^{\rm lin}(t)}_{L^p(\R^3)}
    \leq 
    C_{\varepsilon}
    (1+t)^{
    -(1-\frac{1}{p}) -\frac{|\alphah|}{2}-\frac{1}{4} - (\frac{3}{4}-\varepsilon)(1-\frac{2}{p}) }
    \n{(v_0,\theta_0)}_{ X^{3,3}(\R^3)  },\\
    &
    \n{ \Grad^{\alpha}   \theta^{\rm lin}(t)}_{L^p(\R^3)}
    \leq 
    C
    (1+t)^{ -(1-\frac{1}{p}) - \f{|\alphah|}{2} -\f{3}{4} (1-\f{2}{p} )  }
    {
    \n{(v_0,\theta_0)}_{ X^{2+|\alphah|,3+\alpha_3}(\R^3)   }  }
\end{align} 
for all $\alpha=(\alphah,\alpha_3) \in (\mathbb{N} \cup \{ 0\} )^2 \times (\mathbb{N} \cup \{ 0\} )$ with $|\alpha| \leq 1$, $2\leq p \leq\infty$, and $t>0$. 
Here, $\Ph = \{ \delta_{j,k} + \partial_{x_j}\partial_{x_k}(-\Deltah)^{-1}\}_{1 \leq j,k \leq 2}$ is the horizontal Helmholtz projection.
\end{thm}

\begin{rem}\label{rem:thm:lin}
We mention some remarks on Theorem \ref{thm:lin}.
\begin{enumerate}
    \item 
    The proof of Theorem \ref{thm:lin} relies on the linear solution formula \eqref{lin-sol-1}.
    Thus, the key of the linear analysis is to establish the dispersive $L^p$ decay estimate of the semigroup $\{e^{t\Deltah}e^{\pm i t \frac{|\nablah|}{|\nabla|}}\}_{t\geq 0}$:
    \begin{align}
        \n{e^{t\Deltah}e^{\pm i t \frac{|\nablah|}{|\nabla|}}f}_{L^p(\R^3)}
        \leq
        C
        t^{-(1-\frac{1}{p})}
        t^{-\frac{3}{4}(1-\frac{2}{p})}
        \n{f}_{L^1_{\xh}W^{3,1}_{x_3}(\R^3)}. 
    \end{align}
    See Corollary \ref{cor:disp} below for the detail.
    Here, the additional decay rate $O(t^{-\frac{3}{4}(1-\frac{2}{p})})$ comes from the dispersive estimate of the oscillatory integral related to $\{e^{\pm i t \frac{|\nablah|}{|\nabla|}}\}_{t \in \mathbb{R}}$.
    We should note that the anisotropic Littlewood--Paley analysis and the additional vertical derivative $W^{3,1}_{x_3}$ for $f$ make the above dispersive decay rate faster than the {\it isotropic} dispersive decay rate $O(t^{-\frac{1}{2}(1-\frac{2}{p})})$, that is
    \begin{align}
        \n{e^{t\Delta}e^{\pm it \frac{|\nablah|}{|\nabla|}}f}_{L^p(\R^3)}
        \leq 
        Ct^{-\frac{3}{2}(1-\frac{1}{p})}t^{-\frac{1}{2}(1-\frac{2}{p})}
        \n{f}_{L^1(\R^3)},
    \end{align}
     where $\Delta=\partial_{x_1}^2+\partial_{x_2}^2+\partial_{x_3}^2$ is the $3$D Laplacian.
    This estimate is proved by using \cite{Lee-Tak-17}*{Lemmas 4.2 and 4.3}.
    \item 
    From the divergence free condition, it holds $\partial_{x_3} v_3^{\rm lin} = - \nablah \cdot \vh^{\rm lin}$, which enables us to obtain $t^{-\frac{1}{4}}$-faster decay rate for $v_3^{\rm lin}$ in $L^p$ with all $2 \leq p \leq \infty$.
    The similar phenomenon has already been known in \cites{F21,FL-24,XZ22,L22}; see also \eqref{ANS-p}. 
    However, our enhanced decay rate is weaker than these known results for $p>2$ because of the complicated structure of linear solutions by the dispersion.
    \item 
    Theorem \ref{thm:lin} does not contain the decay information of $\partial_{x_3}v_3^{\rm lin}(t)$, but it is easily obtained. 
    Indeed, from the divergence free condition, it holds 
    \begin{align}\label{p3v3}
        \partial_{x_3}v_3^{\rm lin}(t)
        =
        -\nablah \cdot \vh^{\rm lin}(t)
        =
        -\nablah \cdot \sp{\vh^{\rm lin}(t) - \Ph \vh^{\rm lin}(t)},
    \end{align}
    which and the second estimate of Theorem \ref{thm:lin} imply 
    \begin{align}
        \n{\partial_{x_3}v_3^{\rm lin}(t)}_{L^p(\R^3)}
        ={}&
        \n{\nablah \cdot \sp{\vh^{\rm lin}(t) - \Ph \vh^{\rm lin}(t)}}_{L^p(\R^3)}
        \\
        \leq{}&
        C
        (1+t)^{  
        -(1-\frac{1}{p})- \frac{1}{2}-\f{3}{4}(1-\f{2}{p}) }
        \n{(v_0,\theta_0)}_{ X^{3,3}(\R^3) }.
    \end{align}
\end{enumerate}
\end{rem}

\subsection{Nonlinear stability}
Next, we provide the $L^1$-$L^p$ decay estimates for the nonlinear solutions $(v,\theta)$ to \eqref{eq:B2}.
To state our main result for the nonlinear solutions precisely, we recall the small global well-posedness of \eqref{eq:B2} in Sobolev spaces.  
\begin{lemm}[\cite{Ji-Yan-Wu-22}*{Proposition 1.2}]\label{lemm:GWP-Sob}
Let $m \in \mathbb{N}$ satisfy $m \geq 2$.
Then, there exists a positive constant $\delta=\delta(m)$ such that if a given initial datum $(v_0,\theta_0)\in H^m(\R^3)$ with $\Grad \cdot v_0=0$ satisfies $\n{  (v_0,\theta_0) }_{H^m(\R^3)} \leq \delta$, then \eqref{eq:B2} admits a unique global solution $(v,\theta)\in C([0,\infty);H^m(\R^3))$. Moreover, there exists a positive constant $C=C(m)$ such that
\begin{align}
\n{  (v,\theta)(t) }_{H^m(\R^3)}^{2}
+\int_0^t   
\n{  \nablah (v,\theta)(\tau) }_{H^m(\R^3)}^{2}
d\tau
\leq  C \n{  (v_0,\theta_0) }_{H^m(\R^3)}^2
\end{align}
for all $t \geq 0$.
\end{lemm}
Our second result of this paper reads as follows.
\begin{thm}\label{main-thm}
Let $0<\varepsilon< 1/4$. 
Then, 
there exist positive constants $C_{\varepsilon}$ and $\delta_{\varepsilon}$ such that
for any $(v_0,\theta_0)\in X^{8,4}(\R^3)$ with $\nabla \cdot v_0=0$ and $\n{(v_0,\theta_0)  }_{ X^{8,4}(\R^3)   }\leq \delta_{\varepsilon}$, the corresponding unique global solution $(v,\theta)\in C([0,\infty);H^8(\R^3))$ to \eqref{eq:B2} ensured by Lemma \ref{lemm:GWP-Sob} satisfies 
\begin{align}
    &
    \n{ 
    \Grad^{\alpha} 
    \Ph \vh (t) 
    }_{L^p(\R^3)} 
    \leq 
    C_{\varepsilon}
    (1+t)^{
    -(1-\frac{1}{p}) 
    -\f{|\alphah|}{2}
    }
    \n{(v_0,\theta_0)}_{X^{8,4}  (\R^3) },
    \\
    &
    \n{ 
    \Grad^{\alpha} 
    (\vh - \Ph \vh) (t) 
    }_{L^p(\R^3)} 
    \leq 
    C_{\varepsilon}
    (1+t)^{
    -(1-\frac{1}{p}) 
    -\f{|\alphah|}{2}
    -A_0(p)
    }
    \n{(v_0,\theta_0)}_{ X^{8,4} (\R^3) },
    \\
    &
    \n{ 
    \nablah^{\alphah} 
    v_{3}(t)
    }_{L^p(\R^3)}
    \leq 
    C_{\varepsilon}
    (1+t)^{
    -(1-\frac{1}{p})
    - \frac{|\alphah|}{2} - \frac{1}{4} - A_{\varepsilon}(p) 
    }
    \n{(v_0,\theta_0)}_{ X^{8,4} (\R^3) },
    \label{dec-v3}
    \\
    &
    \n{  
    \Grad^{\alpha} 
    \theta(t)}_{L^p(\R^3)}
    \leq 
    C_{\varepsilon}
    (1+t)^{
    -(1-\frac{1}{p}) 
    -\f{|\alphah|}{2}
    -A_0(p)
    }
    \n{(v_0,\theta_0)}_{ X^{8,4} (\R^3)}
\end{align} 
for all $2\leq p \leq\infty$, $t>0$, and $\alpha=(\alphah,\alpha_3) \in (\mathbb{N} \cup \{ 0\} )^2 \times (\mathbb{N} \cup \{ 0\} )$ with $|\alpha|\leq 1$.
Here, $\Ph = \{ \delta_{j,k} + \partial_{x_j}\partial_{x_k}(-\Deltah)^{-1}\}_{1 \leq j,k \leq 2}$ denotes the horizontal Helmholtz projection and we have set 
\begin{align}
    A_{\varepsilon}(p)
    :={}&
    \min
    \Mp{
    \sp{\frac{3}{4}-\varepsilon}
    \sp{1-\frac{2}{p}},
    {\frac{1}{4}}
    }, \qquad \varepsilon \geq 0.
\end{align}
\end{thm}

\begin{rem}
    Let us provide some comments on Theorem \ref{main-thm}.
    \begin{enumerate}
        \item 
        In comparison to the decay estimate \eqref{decay:JWY} by \cite{Ji-Yan-Wu-22}, not only do we obtain the $L^2$-decay estimates with the anisotropic enhanced dissipative effect:
        \begin{align}
            \n{\nabla^{\alpha}(\vh,\theta)(t)}_{L^2(\R^3)}
            =
            O(t^{-\frac{1}{2}-\frac{|\alphah|}{2}}),
            \quad
            \n{\nablah^{\alphah}v_3(t)}_{L^2(\R^3)}
            =
            O(t^{-\frac{3}{4}-\frac{|\alphah|}{2}})
        \end{align}
        but also extend to the general $L^p$ case and capture the dispersive effect from the stratification. 
        \item 
        Comparing with Theorem \ref{thm:lin}, we see that the nonlinear dispersive decay rates $O(t^{-A_0(p)})$ for $(\vh -\Ph \vh,\theta)$ and $O(t^{-A_{\varepsilon}(p)})$ for $v_3$ become slower than the linear dispersive decay estimates for large $p$.
        This is because some portions of nonlinear terms have some worse decay rates than linear decay rates; see the estimates of $\{\mathcal{D}_{\pm,j}^{\rm vel,h}[v,\theta]\}_{j=2,6}$ and $\{\mathcal{D}_{\pm,j}^{\rm temp}[v,\theta]\}_{j=2,6}$ in Lemma \ref{lemm:vel_h-temp}.
        \item 
        From the same observation as in the third statement of Remark \ref{rem:thm:lin}, we obtain the decay estimate of $\partial_{x_3}v_3$ as follows:
        \begin{align}
            \n{\partial_{x_3}v_3(t)}_{L^p(\R^3)}
            ={}&
            \n{\nablah \cdot \sp{\vh(t) - \Ph \vh(t)}}_{L^p(\R^3)}
            \\
            \leq{}&
            C_{\varepsilon}
            (1+t)^{  
            -(1-\frac{1}{p})- \frac{1}{2}-A_0(p) }
            \n{(v_0,\theta_0)}_{ X^{8,4}(\R^3) }.
        \end{align}
    \end{enumerate}
\end{rem}

\section{Anisotropic Littlewood--Paley theory}\label{sec:pre}
In this section, we prepare the anisotropic Littlewood--Paley theory, which plays a key role in our analysis. 
The facts discussed in this section are used without reference in the following sections and thereafter, so the readers are encouraged to refer back to this section as needed.
Let $\phi \in C_c^{\infty}([0,\infty))$ satisfy
\begin{align}\label{phi}
    0 \leq \phi(r) \leq1 \ {\rm on}\ [0,\infty),
    \quad
    \supp \phi \subset \lp{2^{-1},2},
    \quad
    \sum_{j \in \mathbb{Z}} \phi(2^{-j}r) = 1 \ {\rm on}\ (0,\infty).
\end{align}
Let $\phih_j(\xih):=\phi(2^{-j}|\xih|)$ and $\phiv_k(\xi_3):=\phi(2^{-k}|\xi_3|)$ for $\xih \in \mathbb{R}^2$, $\xi_3 \in \mathbb{R}$, and $j,k \in \mathbb{Z}$.
We define the anisotropic Littlewood--Paley frequency localization operators as 
\begin{align}
    \Deltahh_j f := \mathscr{F}^{-1}_{\mathbb{R}^2}\lp{\phih_j(\xih)\mathscr{F}_{\mathbb{R}^2}[f](\xih)},\quad
    \Deltav_k f := \mathscr{F}^{-1}_{\mathbb{R}}\lp{\phiv_k(\xi_3)\mathscr{F}_{\mathbb{R}}[f](\xi_3)}.
\end{align}
Here, we state the Bernstein inequalities of anisotropic type.
\begin{lemm}\label{lemm:Bern}
    The following statements hold.
    \begin{enumerate}
        \item 
        Let $1 \leq p_1 \leq p_2 \leq \infty$.
        Then, there exists a positive constant $C=C(p_1,p_2)$ such that 
        \begin{align}
            \n{\Deltahh_j\Deltav_k f}_{L^{p_2}(\R^3)}
            \leq
            C
            2^{2(\frac{1}{p_1}-\frac{1}{p_2})j}
            2^{(\frac{1}{p_1}-\frac{1}{p_2})k}
            \n{\Deltahh_j\Deltav_k f}_{L^{p_1}(\R^3)}
        \end{align}
        for all $j,k \in \Z$ and $f$ provided that the right-hand side is finite.
        \item 
        Let $1 \leq p \leq \infty$ and $m,\ell \in \mathbb{N}$.
        Then, it holds
        \begin{align}
            \max_{|\alphah|=m}
            \n{\Deltahh_j\Deltav_k \nablah^{\alphah}\partial_{x_3}^{\ell} f}_{L^{p}(\R^3)}
            \sim
            2^{mj}
            2^{\ell k}
            \n{\Deltahh_j\Deltav_k f}_{L^{p}(\R^3)}.
        \end{align}
        Moreover, for any $s,\sigma \in \R$, it holds
        \begin{align}
            \n{\Deltahh_j\Deltav_k |\nablah|^s|\partial_{x_3}|^{\sigma} f}_{L^{p}(\R^3)}
            \sim
            2^{sj}
            2^{\sigma k}
            \n{\Deltahh_j\Deltav_k f}_{L^{p}(\R^3)}.
        \end{align}
    \end{enumerate}
\end{lemm}
As we may prove Lemma \ref{lemm:Bern} by the same strategy for the isotropic case (see \cite{Bah-Che-Dan-11} for instance), we omit the proof.

Now, we define the anisotropic Besov spaces.
\begin{df}
For $1 \leq p,q \leq \infty$, and $s_1,s_2 \in \R$,
the anisotropic Besov space $\dot{\mathcal{B}}_{p,q}^{s_1,s_2}(\R^3)$ is defined by the set of all tempered distributions $f$ on $\R^3$ satisfying the following semi-norm is finite:
\begin{align}
    \n{f}_{\dot{\mathcal{B}}_{p,q}^{s_1,s_2}(\R^3)}
    :=
    \n{
    \Mp{
    2^{s_1j}2^{s_2k}
    \n{\Deltahh_j\Deltav_k f}_{L^p(\R^3)}
    }_{(j,k)\in \mathbb{Z}^2}
    }_{\ell^q(\mathbb{Z}^2)}.
\end{align}
\end{df}
We notice that our anisotropic Besov spaces are different from the well-used ones introduced in \cites{Pai05,CZ07}.
Therefore, we summarize the basic properties of the anisotropic Besov spaces.
First, we mention embedding relations.
\begin{lemm}\label{lemm:emb}
The following properties hold.
\begin{enumerate}
    \item
    For any $1 \leq p \leq \infty$, the continuous embedding $\dot{\mathcal{B}}_{p,1}^{0,0}(\R^3) \hookrightarrow L^p(\R^3) \hookrightarrow \dot{\mathcal{B}}_{p,\infty}^{0,0}(\R^3)$ holds. 
    \item 
    Let $1 \leq p \leq \infty$, $1 \leq q_1 \leq q_2 \leq \infty$ and $s_1,s_2 \in \mathbb{R}$.
    Then, the continuous embedding $\dot{\mathcal{B}}_{p,q_1}^{s_1,s_2}(\R^3) \hookrightarrow \dot{\mathcal{B}}_{p,q_2}^{s_1,s_2}(\R^3)$ holds.
    \item 
    Let $s_1$, $s_2$, and $s$ be positive real numbers satisfying $s>s_1+s_2$. Then, the continuous embedding $H^s(\mathbb{R}^3) \hookrightarrow \dot{\mathcal{B}}_{2,1}^{s_1,s_2}(\R^3)$ holds. 
\end{enumerate}
\end{lemm}
\begin{proof}
The first and second claims are proved by the same method as in the usual Besov spaces case. 
Therefore, we only focus on the proof of the third claim.
Let $s=s_3+s_4$ with $s_3>s_1$ and {$s_4>s_2$}. 
Then, we have
\begin{align}
    &
    \sum_{k \in \mathbb{Z}}
    2^{s_2k}
    \n{\Deltahh_j\Deltav_kf}_{L^2(\R^3)}
    ={}
    \sum_{k \leq 0}
    2^{s_2k}
    \n{\Deltahh_j\Deltav_kf}_{L^2(\R^3)}
    +
    \sum_{k \geq 1}
    2^{s_2k}
    \n{\Deltahh_j\Deltav_kf}_{L^2(\R^3)}\\
    &\quad
    \leq{}
    \sum_{k \leq 0}
    2^{s_2k}
    \sup_{k\leq0}
    \n{\Deltahh_j\Deltav_kf}_{L^2(\R^3)}
    +
    \sum_{k \geq 1}
    2^{(s_2-s_4)k}
    \sup_{k\geq1}
  {  2^{s_4k}  }
    \n{\Deltahh_j\Deltav_kf}_{L^2(\R^3)}\\
    &\quad
    \leq{}
    C
    \sup_{k\leq 0}
    \n{\Deltahh_j\Deltav_kf}_{L^2(\R^3)}
    +
    C
    \sup_{k \geq 1}
    \n{\Deltahh_j\Deltav_k|\partial_{x_3}|^{s_4}f}_{L^2(\R^3)}\\
    &\quad\leq{}
    C
    \n{\Deltahh_jf}_{L^2(\R^3)}
    +
    C
    \n{\Deltahh_j|\partial_{x_3}|^{s_4}f}_{L^2(\R^3)}.
\end{align}
Thus, multiplying this by $2^{s_1j}$ and taking $\ell^1(\Z)$ norm with respect to $j$, we see that 
\begin{align}
    \n{f}_{\dot{\mathcal{B}}_{2,1}^{s_1,s_2}}
    \leq{}&
    C
    \sum_{j\in \mathbb{Z}}
    2^{s_1j}
    \n{\Deltahh_jf}_{L^2(\R^3)}
    +
    C
    \sum_{j \in \mathbb{Z}}
    2^{s_1j}
    \n{\Deltahh_j|\partial_{x_3}|^{s_4}f}_{L^2(\R^3)}.
\end{align}
By the similar argument as above, we have 
\begin{align}
    &
    \begin{aligned}
    \sum_{j\in \mathbb{Z}}
    2^{s_1j}
    \n{\Deltahh_jf}_{L^2(\R^3)}
    \leq{}&
    C\n{f}_{L^2(\R^3)}
    +
    C\n{|\nablah|^{s_3}f}_{L^2(\R^3)}\\
    \leq{}&
    C\n{f}_{H^s(\R^3)},
    \end{aligned}
    \\
    &
    \begin{aligned}
    \sum_{j\in \mathbb{Z}}
    2^{s_1j}
    \n{\Deltahh_j|\partial_{x_3}|^{s_4}f}_{L^2(\R^3)}
    \leq{}&
    C\n{|\partial_{x_3}|^{s_4}f}_{L^2(\R^3)}
    +
    C\n{|\nablah|^{s_3}|\partial_{x_3}|^{s_4}f}_{L^2(\R^3)},\\
    \leq{}&
    C\n{f}_{H^s(\R^3)},
    \end{aligned}
\end{align}
which completes the proof.
\end{proof}
Next, we focus on the interpolation inequalities.
\begin{lemm}\label{lemm:inter}
The following statements hold.
\begin{enumerate}
    \item 
    Let { $s_1,s_2 \in \mathbb{R}$ with $s_1<s_2$}, $0<\vartheta<1$, and $s = \vartheta s_1 + (1-\vartheta) s_2$. Then there exists a positive constant $C=C(s_1,s_2,\vartheta)$ such that  
    \begin{align}
    \sum_{k \in \mathbb{Z}}
    2^{sk}
    \sup_{j \in \mathbb{Z}}
    \n{ \Deltahh_j \Deltav_k f }_{L^p(\R^3)}
    \leq 
    C
    \n{f}_{\dot{\mathcal{B}}_{p,\infty}^{0,s_1} { (\R^3) } }
    ^{\vartheta}
    \n{f}_{\dot{\mathcal{B}}_{p,\infty}^{0,s_2} { (\R^3) }}
    ^{1-\vartheta}
    \end{align}
    for all $1\leq p \leq \infty$ and $f \in \dot{\mathcal{B}}_{p,\infty}^{0,s_1}(\R^3) \cap \dot{\mathcal{B}}_{p,\infty}^{0,s_2}(\R^3)$.
    \item 
    Let $1\leq p,q \leq \infty$.
    Let $s,s_1,s_2,\sigma,\sigma_1,\sigma_2 \in \mathbb{R}$ and $0 \leq \vartheta \leq 1$ 
    satisfy 
    $s=\vartheta s_1+(1-\vartheta)s_2$, 
    and 
    $\sigma=\vartheta \sigma_1+(1-\vartheta)\sigma_2$.
    Then, it holds
    \begin{align}
        \n{f}_{\dot{\mathcal{B}}_{p,q}^{s,\sigma}(\R^3)}
        \leq 
        \n{f}_{\dot{\mathcal{B}}_{p,q}^{s_1,\sigma_1}(\R^3)}^{\vartheta}
        \n{f}_{\dot{\mathcal{B}}_{p,q}^{s_2,\sigma_2}(\R^3)}^{1-\vartheta}
    \end{align}
    for all $f \in \dot{\mathcal{B}}_{p,q}^{s_1,\sigma_1}(\R^3) \cap \dot{\mathcal{B}}_{p,q}^{s_2,\sigma_2}(\R^3)$.
\end{enumerate}
\end{lemm}
\begin{proof} 
The strategy of the proof is based on \cite{Bah-Che-Dan-11}*{Proposition 2.22}.
We prove the first claim.
It suffices to consider the case $\n{f}_{\dot{\mathcal{B}}_{p,{\infty} }^{0,s_1}(\R^3)} \neq 0$.
We see that
\begin{align}
    \sum_{k \in \mathbb{Z}}
    2^{sk}
    \sup_{j \in \mathbb{Z}}
    \n{ \Deltahh_j \Deltav_k f }_{L^p(\R^3)}
    ={}&
    \sum_{j\leq N}
    2^{k  (\vartheta s_1+(1-\vartheta) s_2)   }
    \sup_{j \in \mathbb{Z}}
    \n{\Deltahh_j\Deltav_kf}_{L^p(\mathbb{R}^3)} 
    \\
    &+
    \sum_{j>N} 
    2^{k  (\vartheta s_1+(1-\vartheta) s_2)   }
    \sup_{j \in \mathbb{Z}}
    \n{\Deltahh_j\Deltav_kf}_{L^p(\mathbb{R}^3)}\\
    \leq{}
    &
    \sp{\sum_{j\leq N}
    2^{k  (1-\vartheta) (s_2-s_1)   }
    }
    \sup_{j,k \in \mathbb{Z}}
  {  2^{s_1 k}  }
    \n{\Deltahh_j\Deltav_kf}_{L^p(\mathbb{R}^3)}
    \\
    &
    +
    \sp{\sum_{j>N} 
    2^{-k  \vartheta (s_2-s_1)   }
    }\sup_{j,k \in \mathbb{Z}}
   { 2^{s_2 k}  }
    \n{\Deltahh_j\Deltav_kf}_{L^p(\mathbb{R}^3)}\\
    \leq{}&
    \f{  2^{N (1-\vartheta) (s_2-s_1)   }   }
    {  2^{  (1-\vartheta) (s_2-s_1)   } -1  }
    \n{f}_{\dot{\mathcal{B}}_{p,{\infty } }^{0,s_1}(\R^3)}
    \\
    &
    +
    \f{   2^{-N  \vartheta (s_2-s_1)   }   }
    { 1- 2^{-  \vartheta (s_2-s_1)   }    }
    \n{f}_{\dot{\mathcal{B}}_{p,{ \infty} }^{0,s_2}(\R^3)}.
\end{align}
Choosing $N$ so that
\begin{align}
    2^{(N-1) (s_2-s_1) } 
    <   
    \dfrac
    {\n{f}_{\dot{\mathcal{B}}_{p,{\infty} }^{0,s_2}(\R^3)}}
    {\n{f}_{\dot{\mathcal{B}}_{p,{ \infty} }^{0,s_1}(\R^3)}}
    \leq 
    2^{N (s_2-s_1) },
\end{align}
we complete the proof of the first claim.

For the second assertion,
the H\"older inequality yields
\begin{align}
    \n{f}_{\dot{\mathcal{B}}_{p,q}^{s,\sigma}(\R^3)}
    ={}&
    \n{
    \Mp{
    \sp{
    2^{s_1j}2^{\sigma_1k}
    \n{\Deltahh_j\Deltav_k f}_{L^p}
    }^{\vartheta}
    \sp{
    2^{s_2j}2^{\sigma_2k}
    \n{\Deltahh_j\Deltav_k f}_{L^p}
    }^{1-\vartheta}
    }_{j,k}
    }_{\ell^q}\\
    \leq{}&
    \n{f}_{\dot{\mathcal{B}}_{p,q}^{s_1,\sigma_1}(\R^3)}^{\vartheta}
    \n{f}_{\dot{\mathcal{B}}_{p,q}^{s_2,\sigma_2}(\R^3)}^{1-\vartheta}.
\end{align}
Thus, we complete the proof.
\end{proof}
In the following lemma, we state the action of the horizontal heat kernel in $\R^3$.
\begin{lemm}
    There exist absolute positive constants $c$ and $C$ such that 
    \begin{align}
        \n{e^{t\Deltah}\Deltahh_jf}_{L^2(\R^3)}
        \leq
        C
        e^{-c2^{2j}t}
        \n{\Deltahh_jf}_{L^2(\R^3)}
    \end{align}
    for all $j \in \mathbb{Z}$, $t>0$, and $f$ provided that the right-hand side is finite.
    Moreover, for any $s>0$, there exists a positive constant $C=C(s)$ such that 
    \begin{align}
        \sum_{j \in \mathbb{Z}}
        2^{sj}
        \n{e^{t\Deltah}\Deltahh_jf}_{L^2(\R^3)}
        \leq
        Ct^{-\frac{s}{2}}
        \sup_{j \in \mathbb{Z}}\n{\Deltahh_jf}_{L^2(\R^3)}
    \end{align}
    for all $t>0$ and $f$ provided that the right-hand side is finite.
\end{lemm}
\begin{proof}
It follows from \cite{Bah-Che-Dan-11}*{Lemma 2.4} that 
\begin{align}
    \n{e^{t\Deltah}\Deltahh_jf(\cdot,x_3)}_{L^2(\R^2)}
    \leq
    C
    e^{-c2^{2j}t}
    \n{\Deltahh_jf(\cdot,x_3)}_{L^2(\R^2)}.
\end{align}
Taking $L^2(\mathbb{R}_{x_3})$-norm of this, we obtain the first estimate.
The second estimate is immediately obtained by the following elementary estimate:
\begin{align}
    \sum_{j \in \mathbb{Z}}
    2^{sj}
  {   e^{-c2^{2j}t}   }
    \leq C t^{-\frac{s}{2}}.
\end{align}
Thus, we complete the proof.
\end{proof}
Next, we investigate the boundedness of zeroth Fourier multipliers in anisotropic Besov spaces. 
Later, we apply them in particular to the $3$D and $2$D Riesz transforms. 
For a function $M=M(\xi)$ and $M'=M'(\xih)$ on the Fourier space, we define their Fourier multipliers $\mathscr{M}_M$ and $\mathscr{M}_{M'}$ by 
\begin{align}
    \mathscr{M}_Mf
    :=
    \mathscr{F}_{\R^3}^{-1}
    \lp{
    M(\xi)
    \mathscr{F}_{\R^3}[f](\xi)
    }, 
    \quad
    \mathscr{M}_{M'}f
    :=
    \mathscr{F}_{\R^3}^{-1}
    \lp{
    M'(\xih)
    \mathscr{F}_{\R^3}[f](\xi)
    }.
\end{align}
\begin{lemm}\label{lemm:M-bdd}
        Let $M=M(\xi)\in C^{\infty}(\R^3_{\xi}\setminus \{0\})$ and $M'=M'(\xih) \in C^{\infty}(\R^2_{\xih}\setminus \{0\})$ be homogeneous functions of $0$-th order.
        Then, there exists a positive constant $C=C(M,M')$ such that 
        \begin{align}
            \n{\mathscr{M}_{M}f}_{\dot{\mathcal{B}}_{p,q}^{s_1,s_2}(\R^3)}
            \leq{}&
            C\n{f}_{\dot{\mathcal{B}}_{p,q}^{s_1,s_2}(\R^3)},\\
            \n{\mathscr{M}_{M'}f}_{\dot{\mathcal{B}}_{p,q}^{s_1,s_2}(\R^3)}
            \leq{}&
            C\n{f}_{\dot{\mathcal{B}}_{p,q}^{s_1,s_2}(\R^3)}
        \end{align}
        for all $ { s_1,s_2 \in \R}, 1 \leq p,q \leq \infty$ and $f \in \dot{\mathcal{B}}_{p,q}^{s_1,s_2}(\R^3)$.
\end{lemm}
\begin{proof}
    Since the the $2$D case $M' = M'(\xih)\in C^{\infty}(\R^2_{\xih}\setminus \{0\})$ is easily obtained by following the argument for the usual isotropic case (see \cite{Bah-Che-Dan-11} for instance), 
    we only show the $3$D case $M = M(\xi)\in C^{\infty}(\R^3_{\xi}\setminus \{0\})$.

    In this proof, we use the $3$D isotropic Littlewood--Paley operators given by 
    \begin{align}
        \Delta_{\ell} f := \mathscr{F}_{\mathbb{R}^3}\lp{\phi(2^{-\ell}|\xi|)\mathscr{F}_{\mathbb{R}^3}[f](\xi)}, \qquad \ell \in \mathbb{Z},
    \end{align}
    where $\phi \in C_c^{\infty}([0,\infty))$ satisfying \eqref{phi}.
    Let us define a Schwartz function 
    \begin{align}
        \psi_{\ell}^M(x)
        :=
        \mathscr{F}^{-1}_{\R^3}
        \lp{
        M(\xi)\sum_{m=-3}^3 \phi(2^{-(\ell+m)}|\xi|)
        }(x)
    \end{align}
    for $x \in \R^3$ and $\ell \in \Z$. Since $\psi_{\ell}^M(x) = 2^{3\ell} \psi_{0}^M(2^{\ell}x)$, we have $\n{\psi_{\ell}^M}_{L^1(\R^3)} = \n{\psi_0^M}_{L^1(\R^3)}$
    for all $\ell \in \Z$.
    Since it holds 
    \begin{align}
        &
        \Mp{ \xi \in \mathbb{R}^3 \ ; \ 2^{j-1} \leq |\xih| \leq 2^{j+1} }
        \cap 
        \Mp{ \xi \in \mathbb{R}^3 \ ; \ 2^{k-1} \leq |\xi_3| \leq 2^{k+1} }
        \\
        &\quad
        \subset
        \Mp{ \xi \in \mathbb{R}^3 \ ; \ 2^{\max\{j,k\}-2} \leq |\xi| \leq 2^{\max\{j,k\}+2} }
    \end{align}
    and 
    \begin{align}
        \sum_{m=-3}^3 \phi(2^{-(\max\{j,k\}+m)}|\xi|) = 1
    \end{align}
     for all $\xi \in \mathbb{R}^3$ with $2^{\max\{j,k\}-2} \leq |\xi| \leq 2^{\max\{j,k\}+2} $,
    we have
    \begin{align}
        \Deltahh_j \Deltav_k \mathscr{M}_Mf
        =
        \Deltahh_j \Deltav_k \sum_{m=-3}^3 \Delta_{\max\{j,k\}+m} \mathscr{M}_Mf
        =
        \psi_{\max\{j,k\}}^M * \Deltahh_j \Deltav_kf.
    \end{align}
    Thus, we obtain
    \begin{align}
        \n{\Deltahh_j \Deltav_k \mathscr{M}_Mf }_{L^p(\R^3)}
        = {}&
        \n{\Deltahh_j \Deltav_k \psi_{\max\{j,k\}}^M*f }_{L^p(\R^3)}\\
        \leq {}&
        \n{\psi_0^M}_{L^1(\R^3)}
        \n{\Deltahh_j \Deltav_k f }_{L^p(\R^3)},
    \end{align}
    which completes the proof.
\end{proof}
Finally, we prepare a product estimate in the anisotropic Besov space $\dot{\mathcal{B}}_{2,1}^{1,\frac{1}{2}}(\R^3)$.
\begin{lemm}\label{lemm:para}
    There exists a pure positive constant $C$ such that 
    \begin{align}
        \n{fg}_{\dot{\mathcal{B}}_{2,1}^{1,\frac{1}{2}}(\R^3)}
        \leq 
        C
        \n{f}_{\dot{\mathcal{B}}_{2,1}^{1,\frac{1}{2}}(\R^3)}
        \n{g}_{\dot{\mathcal{B}}_{2,1}^{1,\frac{1}{2}}(\R^3)}
    \end{align}
    for all $f,g \in \dot{\mathcal{B}}_{2,1}^{1,\frac{1}{2}}(\R^3)$.
\end{lemm}
Before proving Lemma \ref{lemm:para}, we recall the following Chemin--Lerner norm: 
\begin{align}
    \n{f}_{\widetilde{L^r}(\mathbb{R};\dB_{p,q}^{s}(\R^2))}
    :=
    \n{
    \Mp{2^{sj}\n{\Deltahh_j f}_{L^r(\R;L^p(\R^2))}}_{j \in \mathbb{Z}}
    }_{\ell^q(\Z)} 
\end{align}
for $1 \leq p,q,r \leq \infty$ and $s \in \mathbb{R}$.
It is well-known that for $1 \leq p,r,r_1,r_2 \leq \infty$ with $1/r = 1/r_1 + 1/r_2$, it holds 
\begin{align}
    \n{fg}_{\widetilde{L^r}(\mathbb{R};\dB_{p,1}^{\frac{2}{p}}(\R^2))}
    \leq
    C
    \n{f}_{\widetilde{L^{r_1}}(\mathbb{R};\dB_{p,1}^{\frac{2}{p}}(\R^2))}
    \n{g}_{\widetilde{L^{r_2}}(\mathbb{R};\dB_{p,1}^{\frac{2}{p}}(\R^2))}.
\end{align}
We may prove the above inequality by the same strategy as in the proof of the fact that $\dB_{p,1}^{\frac{2}{p}}(\mathbb{R}^2)$ is an algebra; see \cite{Bah-Che-Dan-11}*{Section 2.6} for the details.
Now, we show Lemma \ref{lemm:para}.
\begin{proof}[Proof of Lemma \ref{lemm:para}]
The proof is essentially based on the para-product calculus. 
We apply the Bony decomposition to $fg$ for the vertical direction:
\begin{align}
    fg
    =
    T^{\rm v}_fg
    +
    R^{\rm v}(f,g)
    +
    T^{\rm v}_gf,
\end{align}
where we have set 
\begin{align}
    T^{\rm v}_fg
    :=
    \sum_{k \in \mathbb{Z}}
    \sum_{k' \leq k-3}
    \Deltav_{k'}f \Deltav_k g, \qquad
    R^{\rm v}(f,g)
    :=
    \sum_{k \in \mathbb{Z}}
    \sum_{|k' - k| \leq 2}
    \Deltav_{k'}f \Deltav_k g.
\end{align}
It follows from 
\begin{align}
    \Deltav_k 
    T^{\rm v}_fg
    =
    \Deltav_k
    \sum_{ | k'' - k | \leq 3}
    \sum_{k' \leq k'' -3}
    \Deltav_{k'}f \Deltav_{k''} g
\end{align}
that 
\begin{align}
    &\sum_{j \in \mathbb{Z}}
    2^{j}
    \n{\Deltahh_j\Deltav_k 
    T^{\rm v}_fg}_{L^2(\R^3)}
    \leq{}
    C
    \sum_{ | k'' - k | \leq 3}
    \sum_{k' \leq k'' -3}
    \sum_{j \in \mathbb{Z}}
    2^{j}
    \n{\Deltahh_j\sp{\Deltav_{k'}f \Deltav_{k''} g}}_{L^2(\R^3)}
    \\
    &\quad
    ={}
    C
    \sum_{ | k'' - k | \leq 3}
    \sum_{k' \leq k'' -3}
    \n{\Deltav_{k'}f \Deltav_{k''} g}_{\widetilde{L^2}(\mathbb{R};\dB_{2,1}^1(\R^2))}\\
    &\quad
    \leq 
    {}
    C
    \sum_{ | k'' - k | \leq 3}
    \sum_{k' \leq k'' -3}
    \n{\Deltav_{k'}f }_{\widetilde{L^{\infty}}(\mathbb{R};\dB_{2,1}^1(\R^2))}
    \n{\Deltav_{k''} g}_{\widetilde{L^2}(\mathbb{R};\dB_{2,1}^1(\R^2))} \\
    &\quad\leq 
    {}
    C
    \sum_{ | k'' - k | \leq 3}
    \sum_{k' \leq k'' -3}
    2^{\frac{1}{2}k'}
    \n{\Deltav_{k'}f }_{\widetilde{L^2}(\mathbb{R};\dB_{2,1}^1(\R^2))}
    \n{\Deltav_{k''} g}_{\widetilde{L^2}(\mathbb{R};\dB_{2,1}^1(\R^2))} \\
    &\quad\leq 
    {}
    C
    \n{f }_{\dot{\mathcal{B}}_{2,1}^{1,\frac{1}{2}}(\R^3)}
    \sum_{ | k'' - k | \leq 3}
    \n{\Deltav_{k''} g}_{\widetilde{L^2}(\mathbb{R};\dB_{2,1}^1(\R^2))}.
\end{align}
Multiplying this by $2^{\frac{1}{2}k}$ and then taking $\ell^1(\Z)$-norm, we have
\begin{align}
    \n{T^{\rm v}_fg}_{\dot{\mathcal{B}}_{2,1}^{1,\frac{1}{2}}}
    \leq 
    C
    \n{f }_{\dot{\mathcal{B}}_{2,1}^{1,\frac{1}{2}}(\R^3)}
    \n{g }_{\dot{\mathcal{B}}_{2,1}^{1,\frac{1}{2}}(\R^3)}.
\end{align}
Similarly, we have 
$\n{T^{\rm v}_gf}_{\dot{\mathcal{B}}_{2,1}^{1,\frac{1}{2}}}
\leq 
C
\n{f }_{\dot{\mathcal{B}}_{2,1}^{1,\frac{1}{2}}(\R^3)}
\n{g }_{\dot{\mathcal{B}}_{2,1}^{1,\frac{1}{2}}(\R^3)}$.
Using 
\begin{align}
    \Deltav_k 
    R^{\rm v}(f,g)
    =
    \Deltav_k
    \sum_{k' \geq k -4}
    \sum_{ | k'' - k'  | \leq 2}
    \Deltav_{k'}f \Deltav_{k''} g,
\end{align}
we have 
\begin{align}
    &
    \sum_{j \in \mathbb{Z}}
    2^j
    \n{\Deltahh_j\Deltav_k 
    R^{\rm v}(f,g)}_{L^2(\R^3)}
    \\
    &\quad\leq 
    C
    \sum_{k' \geq k -4}
    \sum_{ | k'' - k'  | \leq 2}
    \n{\Deltav_{k'}f \Deltav_{k''} g}_{\widetilde{L^2}(\mathbb{R};\dB_{2,1}^1(\R^2))}\\
    &\quad
    \leq 
    C
    \sum_{k' \geq k -4}
    \sum_{ | k'' - k'  | \leq 2}
    \n{\Deltav_{k'}f}_{\widetilde{L^{\infty}}(\mathbb{R};\dB_{2,1}^1(\R^2))}
    \n{\Deltav_{k''} g}_{\widetilde{L^2}(\mathbb{R};\dB_{2,1}^1(\R^2))}\\
    &\quad
    \leq 
    C
    \sum_{k' \geq k -4}
    2^{-\frac{1}{2}k'}
    \sum_{ | k'' - k'  | \leq 2}
    2^{\frac{1}{2}k'}
    \n{\Deltav_{k'}f}_{\widetilde{L^2}(\mathbb{R};\dB_{2,1}^1(\R^2))}
    2^{\frac{1}{2}k''}
    \n{\Deltav_{k''} g}_{\widetilde{L^2}(\mathbb{R};\dB_{2,1}^1(\R^2))},
\end{align}
which and the Hausdorff--Young inequality imply
\begin{align}
    &
    \n{
    R^{\rm v}(f,g)}_{\dot{\mathcal{B}}_{2,1}^{1,\frac{1}{2}}(\R^3)}\\
    &\quad
    \leq 
    C
    \sum_{k \in \mathbb{Z}}
    \sum_{k' \geq k -4}
    2^{-\frac{1}{2}(k'-k)}
    \sum_{ | k'' - k'  | \leq 2}
    2^{\frac{1}{2}k'}
    \n{\Deltav_{k'}f}_{\widetilde{L^2}(\mathbb{R};\dB_{2,1}^1(\R^2))}
    2^{\frac{1}{2}k''}
    \n{\Deltav_{k''} g}_{\widetilde{L^2}(\mathbb{R};\dB_{2,1}^1(\R^2))}\\
    &\quad\leq{}
    C
    \sum_{k' \in \mathbb{Z}}
    \sum_{ | k'' - k'  | \leq 2}
    2^{\frac{1}{2}k'}
    \n{\Deltav_{k'}f}_{\widetilde{L^2}(\mathbb{R};\dB_{2,1}^1(\R^2))}
    2^{\frac{1}{2}k''}
    \n{\Deltav_{k''} g}_{\widetilde{L^2}(\mathbb{R};\dB_{2,1}^1(\R^2))}\\
    &\quad
    \leq
    C
    \n{f }_{\dot{\mathcal{B}}_{2,1}^{1,\frac{1}{2}}(\R^3)}
    \n{g }_{\dot{\mathcal{B}}_{2,1}^{1,\frac{1}{2}}(\R^3)}.
\end{align}
Thus, we complete the proof.
\end{proof}

\section{Linear stability: Proof of Theorem \ref{thm:lin}}\label{lin-ana}
In this section we consider the solutions to the following linearized system:
\begin{align}\label{eq:B-lin}
    \begin{cases}
        \partial_t v^{\rm lin} -  \Deltah v^{\rm lin} -  \theta^{\rm lin} e_3 
        +  \nabla p^{\rm lin} = 0, \qquad & t > 0, x \in \mathbb{R}^3,\\
        \partial_t \theta^{\rm lin} -  \Deltah \theta^{\rm lin} +  v^{\rm lin}_3 
         = 0, & t > 0, x \in \mathbb{R}^3,\\
        \nabla \cdot  v^{\rm lin} = 0, & t \geq 0, x \in \mathbb{R}^3,\\
        v^{\rm lin}(0,x)=v_0(x),\quad \theta^{\rm lin}(0,x)=\theta_0(x), & x \in \mathbb{R}^3.
    \end{cases}
\end{align} 
\subsection{Explicit formula for linear solutions}
Applying the Helmholtz projection $\mathbb{P}=\Mp{\delta_{j,k}+\partial_{x_j}\partial_{x_k} (-\Delta)^{-1}}_{1 \leq j,k \leq 3}$ to the first equation of \eqref{eq:B-lin}, we see that $u^{\rm lin}=(v^{\rm lin},\theta^{\rm lin})$ solves  
\begin{align}
    \begin{cases}
        \partial_t u^{\rm lin} - \Deltah u^{\rm lin} + \widetilde{\mathbb{P}} J \widetilde{\mathbb{P}} u^{\rm lin} = 0, & t > 0, x \in \mathbb{R}^3,\\
        \pnabla \cdot u^{\rm lin} = 0, & t \geq 0, x \in \mathbb{R}^3,\\
        u^{\rm lin}(0,x) = u_0 := (v_0,\theta_0), & x \in \mathbb{R}^3,
    \end{cases}
\end{align}
where $\pnabla:=(\nabla,0)=(\partial_{x_1},\partial_{x_2},\partial_{x_3},0)$ and $4 \times 4$ matrices $\widetilde{\mathbb{P}}$ and $J$ are defined as 
\begin{align}
    \widetilde{\mathbb{P}}
    :=
    \sp{
    \begin{array}{c|c}
        \mbox{\Large $\mathbb{P}$} 
        &
        \begin{matrix}
            0 \\ 0 \\ 0
        \end{matrix} \\
        \hline
        \begin{matrix}
            0 & 0 & 0
        \end{matrix} 
        & 1
    \end{array}
    },\qquad
    J
    :=
    \begin{pmatrix}
        0 & 0 & 0 & 0 \\
        0 & 0 & 0 & 0 \\
        0 & 0 & 0 & -1 \\
        0 & 0 & 1 & 0 \\
    \end{pmatrix}.
\end{align}
Then, the linear solution $u^{\rm lin}$ is given by 
\begin{align}\label{u^lin}
    u^{\rm lin}(t)
    =
    e^{t\Deltah}
    e^{-t\widetilde{\mathbb{P}} J \widetilde{\mathbb{P}}}
    u_0.
\end{align}
By \cites{Lee-Tak-17,Tak-19}, the semigroup $e^{-t\widetilde{\mathbb{P}} J \widetilde{\mathbb{P}}}$ is written explicitly as follows:
\begin{align}\label{e^tJ}
    e^{-t\widetilde{\mathbb{P}} J \widetilde{\mathbb{P}}}u_0
    =
    \sum_{\sigma \in \{ 0, \pm \}}
    e^{\sigma i t \frac{|\nablah|}{|\nabla|}}P_{\sigma}u_0,
\end{align}
where $P_{\sigma}u_0=\mathscr{F}^{-1}\lp{\widehat{P_{\sigma}}(\xi)\widehat{u_0}(\xi)}$ are eigenprojections
and the corresponding matrices $\widehat{P_{\sigma}}(\xi)$ are given by
\begin{align}\label{P_s}
    \begin{split}
    \widehat{P_0}(\xi)
    :={}&
    \begin{pmatrix}
        1 - \xi_1^2|\xih|^{-2} & -\xi_1\xi_2|\xih|^{-2} & 0 & 0 \\
        -\xi_1\xi_2|\xih|^{-2} & 1 - \xi_2^2|\xih|^{-2} & 0 & 0 \\
        0 & 0 & 0 & 0 \\
        0 & 0 & 0 & 0 
    \end{pmatrix},
    \\
    \widehat{P_{\pm}}(\xi)
    :={}&
    \frac{1}{2|\xih|^2|\xi|^2}
    \begin{pmatrix}
        -\xi_1^2\xi_3^2 & -\xi_1\xi_2\xi_3^2 & \xi_1\xi_3|\xih|^2 & \pm i\xi_1\xi_3|\xih||\xi|\\
        -\xi_1\xi_2\xi_3^2 & -\xi_2^2\xi_3^2 & \xi_2\xi_3|\xih|^2 & \pm i\xi_2\xi_3|\xih||\xi|\\
        \xi_1\xi_3|\xih|^2 & \xi_2\xi_3|\xih|^2 & -|\xih|^4 & \mp i|\xih|^3|\xi|\\
        \pm i\xi_1\xi_3|\xih||\xi| & \pm i\xi_2\xi_3|\xih||\xi| & \mp i|\xih|^3|\xi| & |\xih|^2|\xi|^2
    \end{pmatrix}.
    \end{split}
\end{align}
Then, from \eqref{u^lin}, \eqref{e^tJ}, and \eqref{P_s}, $v^{\rm lin}$ and $\theta^{\rm lin}$ has the following solution formula:
\begin{align}\label{lin-sol-1}
    \begin{split}
    \mathbb{P}_{\rm h}v_{\rm h}^{\rm lin}(t)
    ={}&
    e^{t\Deltah}\Ph v_{0,{\rm h}},\\
    v_{\rm h}^{\rm lin}(t) - \mathbb{P}_{\rm h}v_{\rm h}^{\rm lin}(t)
    ={}&
    \sum_{\sigma \in \{\pm \}}
    e^{t\Deltah}
    e^{\sigma it \frac{|\nablah|}{|\nabla|}}
    Q_{\sigma}^{\rm vel,h}(v_0,\theta_0),\\
    v^{\rm lin}_3(t)
    ={}&
    \sum_{\sigma \in \{\pm \}}
    e^{t\Deltah}
    e^{\sigma it \frac{|\nablah|}{|\nabla|}}
    Q_{\sigma}^{\rm vel,v}(v_0,\theta_0),\\
    \theta^{\rm lin}(t)
    ={}&
    \sum_{\sigma \in \{\pm \}}
    e^{t\Deltah}
    e^{\sigma it \frac{|\nablah|}{|\nabla|}}
    Q_{\sigma}^{\rm temp}(v_0,\theta_0),
    \end{split}
\end{align}
where $\mathbb{P}_{\rm h}=\Mp{\delta_{j,k}+\partial_{x_j}\partial_{x_k} (-\Deltah)^{-1}}_{1 \leq j,k \leq 2}$ is the $2$D Helmholtz projection and 
\begin{align}
    &
    Q_{\pm}^{\rm vel, h}(v_0,\theta_0)
    :=
    \frac{1}{2}\Mp{-\frac{\nablah}{|\nablah|}
    \frac{\partial_{x_3}^2}{|\nabla|^2}
    \sp{\frac{\nablah}{|\nablah|} \cdot v_{0,{\rm h}}}
    +
    \frac{\nablah}{|\nabla|}
    \frac{\partial_{x_3}}{|\nabla|}
    v_{0,3}
    \pm 
    i
    \frac{\nablah}{|\nablah|}
    \frac{\partial_{x_3}}{|\nabla|}
    \theta_0},\\
    &
    Q_{\pm}^{\rm vel,v}(v_0,\theta_0)
    :=
    -
    \frac{1}{2}
    \Mp{
    \frac{\partial_{x_3}}{|\nabla|}
    \sp{
    \frac{\nablah}{|\nabla|}
    \cdot 
    v_{0,{\rm h}}
    }
    +
    \frac{|\nablah|^2}{|\nabla|^2}
    v_{0,3}
    \mp 
    i
    \frac{\nablah}{|\nabla|}\cdot\frac{\nablah}{|\nablah|}
    \theta_0
    },\\
    &
    Q_{\pm}^{\rm temp}(v_0,\theta_0)
    :=
    \mp 
    \frac{i}{2}
    \Mp{
    \frac{\partial_{x_3}}{|\nabla|}
    \sp{\frac{\nablah}{|\nablah|}
    \cdot v_{0,{\rm h}}}
    -
    \frac{\nablah}{|\nabla|} \cdot \frac{\nablah}{|\nablah|}
    v_{ 0,3 }
    }  
    +
    \frac{1}{2}
    \theta_0.
\end{align}
\subsection{Estimates of semigroups related to the linear solutions}
Based on the solution formula \eqref{lin-sol-1}, we first consider the decay estimate of the horizontal heat kernel $\{e^{t\Deltah} \}_{t>0}$ in $L^2(\R^3)$ and $L^{\infty}(\R^3)$.
\begin{lemm}\label{lemm:heat-kernel}
    There exists a positive constant $C$ such that 
    \begin{align}
        \n{e^{t\Deltah}f}_{L^2(\R^3)}
        \leq{}&
        Ct^{-\frac{1}{2}}
        \n{f}_{\dot{\mathcal{B}}_{1,\infty}^{0,0}(\R^3)}^{\frac{1}{2}}
        \n{f}_{\dot{\mathcal{B}}_{1,\infty}^{0,1}(\R^3)}^{\frac{1}{2}},\\
        \n{e^{t\Deltah}f}_{L^{\infty}(\R^3)}
        \leq{}&
        Ct^{-1}
        \n{f}_{\dot{\mathcal{B}}_{1,\infty}^{0,0}(\R^3)}^{\frac{1}{2}}
        \n{f}_{\dot{\mathcal{B}}_{1,\infty}^{0,2}(\R^3)}^{\frac{1}{2}}
    \end{align}
    for all $t>0$
    and 
    $f \in \dot{\mathcal{B}}_{1,\infty}^{0,0}(\R^3) \cap \dot{\mathcal{B}}_{1,\infty}^{0,2}(\R^3)$.
\end{lemm}
\begin{proof}
For $L^2$-estimate, we see
\begin{align}
    \n{ e^{t\Deltah}   f }_{L^{2}(\R^3)}
    &  
    \leq 
    C \sum_{j \in \mathbb{Z}}
    e^{ -c2^{2j} t  } 
     \sum_{k\in \mathbb{Z}}
     \n{ \Deltahh_j \Deltav_k   f   }_{L^{2}(\R^3)}   \\  
    & 
    \leq  
    C 
    \sum_{j \in \mathbb{Z}}
    e^{ -c2^{2j} t  }  2^{j}
    \sum_{k\in \mathbb{Z}} 
    2^{\f{1}{2} k } 
    \n{ \Deltahh_j \Deltav_k   f  }_{L^{1}(\R^3)}     
    \\ 
    & 
    \leq  
    C 
    t^{-\frac{1}{2}}
    \sum_{k\in \mathbb{Z}} 
    2^{\f{1}{2} k } 
    \sup_{j \in \mathbb{Z}}
    \n{ \Deltahh_j \Deltav_k   f  }_{L^{1}(\R^3)}     
    \\ 
    & 
    \leq 
    Ct^{-\frac{1}{2}}
    \n{f}_{\dot{\mathcal{B}}_{1,\infty}^{0,0}(\R^3)}^{\frac{1}{2}}
    \n{f}_{\dot{\mathcal{B}}_{1,\infty}^{0,1}(\R^3)}^{\frac{1}{2}}.
\end{align}
For $L^{\infty}$-estimate, we estimate as
\begin{align}
    \n{ e^{t\Deltah}   f }_{L^{\infty}(\R^3)}
    &
    \leq 
    C \sum_{j \in \mathbb{Z}}
    e^{ -c2^{2j} t  } 
     \sum_{k\in \mathbb{Z}}
     \n{ \Deltahh_j \Deltav_k  f   }_{L^{\infty}(\R^3)}   \\  
    & 
    \leq  
    C 
    \sum_{j \in \mathbb{Z}}
    e^{ -c2^{2j} t  }  2^{2j}
    \sum_{k\in \mathbb{Z}} 
    2^{ k } 
    \n{ \Deltahh_j \Deltav_k  f   }_{L^{1}(\R^3)}     
    \\   
    & 
    \leq  
    C 
    t^{-1}
    \sum_{k\in \mathbb{Z}} 
    2^{ k } 
    \sup_{j \in \mathbb{Z}}
    \n{ \Deltahh_j \Deltav_k  f   }_{L^{1}(\R^3)}     
    \\   
    &    
    \leq 
    Ct^{-1}
    \n{f}_{\dot{\mathcal{B}}_{1,\infty}^{0,0}(\R^3)}^{\frac{1}{2}}
    \n{f}_{\dot{\mathcal{B}}_{1,\infty}^{0,2}(\R^3)}^{\frac{1}{2}}.
\end{align}
Thus, we complete the proof.
\end{proof}

Next, we establish a dispersive estimate of the following linear propagator
\begin{align}
e^{ it \frac{|\nablah|}{|\nabla|}}  f(x)
:=
\f{1}{(2\pi)^3}
\int_{\R^3}
e^{i x \cdot \xi + i t \f{|\xih|}{|\xi| }}
\widehat{f}(\xi)  d\xi,
\quad t\in \R, x\in \R^3.
\end{align} 
\begin{lemm}\label{dis-est-pm} 
There exists a positive constant $C$ such that
\begin{align} 
    \n{
    \Deltahh_j\Deltav_k e^{ it\frac{|\nablah|}{|\nabla|}}f}_{L^{\infty}(\mathbb{R}^3)}
    \leq 
    \begin{cases}
    C 
    2^{2j}2^k
    \sp{1+2^{-(k-j)}|t|}^{-\frac{3}{2}}
    \n{\Deltahh_j\Deltav_k f}_{L^1(\R^3)}, & (k > j),\\
    C
    2^{2j}2^k
    \sp{1+2^{2(k-j)}|t|}^{-\frac{3}{2}}
    \n{\Deltahh_j\Deltav_k f}_{L^1(\R^3)}, & (k \leq j)
    \end{cases}
\end{align}
for all $j,k \in \mathbb{Z}$, $t \in \mathbb{R}$, and $f$ provided that the right-hand side is finite.
\end{lemm}
In order to investigate the dispersive estimate in Lemma \ref{dis-est-pm}, 
we use the stability for the stationary phase estimates of oscillatory integrals under small perturbations. For the proof, we refer to \cite{FW}. 
\begin{lemm}\label{sta-pha-sta}
Let $d\geq 1$ be an integer. 
Let $\psi\in C_c^{\infty}(\R^d)$ and let $\mathcal{D}$ be a neighborhood of $\supp \psi$. 
Let $\Phi \in C^{\infty}(\mathcal{D};\R)$ satisfy
\begin{align}
{ \rm{rank} } ( \nabla^2\Phi (\xi)    ) \geq k,
\quad 
\xi \in \supp \psi
\end{align}
for some $k \in \{1,...,d  \}$. Then, there exists two positive constants $\varepsilon=\varepsilon(d,\psi,\Phi)$ and $C=C(d,\psi,\Phi)$ such that
\begin{align}
 \left|
\int_{\R^d}
e^{i x \cdot \xi}
e^{i t ( \Phi(\xi) +\Psi(\xi) )   }
\psi(\xi) d\xi 
 \right|
 \leq 
 C 
 (1+|t|)^{ -\f{k}{2}   }
\end{align}
for all $t\in \R,x\in \R^d$ and $\Psi \in C^{\infty}(\mathcal{D};\R)$ satisfying $\n{\Psi}_{ C^{d+3}( \supp \psi )    }\leq \varepsilon$. 
\end{lemm}
Making use of Lemma \ref{sta-pha-sta}, we shall prove Lemma \ref{dis-est-pm}.
\begin{proof}[Proof of Lemma \ref{dis-est-pm}]
It holds that
    \begin{align}
    \Deltahh_j\Deltav_k
    e^{it\frac{|\nablah|}{|\nabla|}}f
    =
    I_{j,k}(t,\cdot)*
    \Deltahh_j\Deltav_kf(x),
\end{align}
where 
\begin{align}
    I_{j,k}(t,x)
    :={}
    &
    \frac{1}{(2\pi)^3}
    \int_{\mathbb{R}^3}
    e^{ix\cdot \xi}
    e^{it\frac{|\xih|}{|\xi|}}
    \psi(2^{-j}\xih,2^{-k}\xi_3)
    d\xi\\
    ={}
    &
    \frac{2^{2j}2^k}{(2\pi)^3}
    J_{k-j}(t,2^j \xh,2^k x_3).
\end{align}
Here, we have set
\begin{align}
    \psi(\xi)
    :=
    \widetilde{\phi}(|\xih|)
    \widetilde{\phi}(|\xi_3|),
    \qquad
    \widetilde{\phi}(r)
    :=\phi(2^{-1}r)+\phi(r)+\phi(2r),
\end{align}
and
\begin{align}
    J_m(t,x):=
    \int_{\mathbb{R}^3}
    e^{ix\cdot \xi}
    e^{it p_m(\xi)}
    \psi(\xi)
    d\xi,\qquad
    p_m(\xi)
    :=
    \frac{|\xih|}{\sqrt{|\xih|^2+2^{2m}\xi_3^2}}.
\end{align}
By the Hausdorff--Young inequlity, it suffices to show that there exists a positive constant $C=C( \psi )$ such that for all $t\in \R$
\begin{align}\label{decay-J}
  \n{J_m(t)}_{L^{\infty}(\mathbb{R}^3)}
    \leq 
     \begin{cases}
        C (1+ 2^{-m}|t|)^{-\f{3}{2}} 
        , \quad & m \geq 1,\\
          C (1+ 2^{2m}|t|)^{-\f{3}{2}} 
        , \quad & m \leq 0. \\
    \end{cases}
\end{align}
On the one hand, one easily sees that
\begin{align}
\supp  \psi
    \subset
    \{  
    (\xih,\xi_3)\in \R^3; 
    2^{-2} \leq |\xih| \leq 2^{2},
     2^{-2} \leq |\xi_3| \leq 2^{2}
    \}.
\end{align}
On the other hand, a straightforward calculation gives 
\begin{align}
    \det(\Grad^2 p_m(\xi)) 
    =
    -
    \f{  2^{6m} \xi_3^4       }
    {  ( |\xih|^2+ 2^{2m}\xi_3^2   )^{\f{9}{2}}           |\xih|   }, \quad
    \xi \in \R^3. 
\end{align} 
Thus,
\begin{align}
{ \rm{rank} } ( \nabla^2 p_m(\xi)    ) =3,
\quad 
\xi \in \supp \psi ;
\end{align}
whence Lemma \ref{sta-pha-sta} (with zero-perturbation $\Psi \equiv 0$) ensures that for any $m\in \mathbb{Z}$, there exists a positive constant $C_m=C( \psi,m )$ such that for all $t\in \R$ 
\begin{align}
 \n{J_m(t)}_{L^{\infty}(\mathbb{R}^3)}
    \leq
    C_m 
    (1+|t|)^{-\frac{3}{2}}. 
\end{align}
However, the above argument does not provide the exact behavior of $C_m$ as $m \rightarrow \pm \infty$. To continue, we specify the behavior of $C_m$ based on the stability Lemma \ref{sta-pha-sta}. We begin with the case where $m\ll -1$. Observe that 
\begin{align}
2^{-2m} (1-p_m(\xi)  )
= \f{1}{2} \f{\xi_3^2}{|\xih|^2}
+E_m^{(1) }(\xi)
\end{align}
with 
\begin{align}
 &  \n{  E_m^{(1) }  }_{C^6( \supp  \psi )   } 
    \leq 
    C 2^{2m},  \\
 & \det \nabla^2\left\{ \f{1}{2} \f{\xi_3^2}{|\xih|^2}   \right\}  
 =
 \f{\xi_3^4}{|\xih|^{7}} \neq 0 , \quad 
 \xi \in \supp  \psi. 
\end{align}
Using the stability Lemma \ref{sta-pha-sta}, we see that there exists an integer $m_1=m_1(\psi)\leq 0$ such that
\begin{align}
 |J_m(t,x)|& =
    \left|\int_{\mathbb{R}^3}
    e^{ix\cdot \xi}
    e^{it p_m(\xi)}
    \psi(\xi) 
    d\xi \right| \\
    &= 
    \left| e^{-it}
    \int_{\mathbb{R}^3}
    e^{ix\cdot \xi} 
    e^{i (t2^{2m}) 2^{-2m}(1-p_m(\xi))    }
    \psi(\xi)
    d\xi \right| \\
    & =
     \left| 
    \int_{\mathbb{R}^3}
    e^{ix\cdot \xi} 
    e^{i (t2^{2m}) \left(  \f{1}{2} \f{\xi_3^2}{|\xih|^2} 
+E_m^{(1) }(\xi)   \right)          }
    \psi(\xi)
    d\xi \right|  \\
    & \leq 
    C (1+ 2^{2m} |t| )^{-\f{3}{2}}
\end{align}
for all $m \leq m_1,t\in \R,x\in \R^3$. 
We apply the similar spirit to the case of $m \gg 1$ and it holds 
\begin{align}
 2^{m} p_m(\xi)
=  \f{|\xih|}{|\xi_3|} 
+E_m^{(2) }(\xi)
\end{align}
with
\begin{align}
 &  \n{  E_m^{(2) }  }_{C^6( \supp \psi)   } 
    \leq 
    C 2^{-2m},  \\
 & \det \nabla^2
 \left\{ \f{|\xih|}{|\xi_3|}  \right\}  
 =
 -\f{1}{  |\xih||\xi_3|^5  } \neq 0 , \quad 
 \xi \in \supp \psi.  
\end{align}
Invoking again the stability Lemma \ref{sta-pha-sta}, we infer that there exists an integer $m_2=m_2(\psi)\geq 0$ such that
\begin{align}
 |J_m(t,x)|& =
    \left|\int_{\mathbb{R}^3}
    e^{ix\cdot \xi}
    e^{it p_m(\xi)}
    \psi(\xi)
    d\xi \right| \\
    &= 
    \left|
    \int_{\mathbb{R}^3}
    e^{ix\cdot \xi} 
    e^{i (t 2^{-m} )  2^{m} p_m(\xi)    }
   \psi(\xi)
    d\xi \right| \\
    & =
     \left| 
    \int_{\mathbb{R}^3}
    e^{ix\cdot \xi} 
    e^{i (t2^{-m}) \left(   \f{|\xih|}{|\xi_3|} 
+E_m^{(2) }(\xi)   \right)          }
   \psi(\xi)
    d\xi \right|  \\
    & \leq 
    C (1+ 2^{-m} |t| )^{-\f{3}{2}}
\end{align}
for all $m \geq m_2$, $t\in \R$, and $x\in \R^3$. Combining the two estimates gives \eqref{decay-J} and we complete the proof.
\end{proof}

From the above lemma, we deduce the $L^{\infty}$ and $L^2$ estimate for the semigroup $\{e^{t\Deltah}e^{\pm i t \frac{|\nablah|}{|\nabla|}}\}_{t>0}$, which plays a crucial role in the proof of Theorem \ref{thm:lin}.
\begin{lemm}\label{disp-est-p}
For any $0<\vartheta <1/2$, there exists a positive constant $C=C(\vartheta)$ such that 
\begin{align}
    \n{ e^{t\Deltah}e^{ \pm it \frac{|\nablah|}{|\nabla|}}f}_{L^{\infty}(\R^3)} 
    \leq{}&
    C
    t^{-\frac{7}{4}}
    \n{f}_{\dot{\mathcal{B}}_{1,\infty}^{0,1}(\R^3)}^{\frac{1}{4}}
    \n{f}_{\dot{\mathcal{B}}_{1,\infty}^{0,3}(\R^3)}^{\frac{3}{4}}\\
    &
    +
    Ct^{-(1+2\vartheta) }  
    \n{f}_{\dot{\mathcal{B}}_{1,\infty}^{0,0}(\R^3)}^{2\vartheta}
    \n{f}_{\dot{\mathcal{B}}_{1,\infty}^{0,1}(\R^3)}^{1-2\vartheta},\\
    \n{ e^{t\Deltah}e^{ \pm  it \frac{|\nablah|}{|\nabla|}}     f}_{L^{2}(\R^3)}
    \leq{}&  
    C 
    t^{-\frac{1}{2}}
    \n{f}_{  \dot{\mathcal{B}}_{1,\infty}^{0,0}(\R^3)   }^{\frac{1}{2}}
    \n{f}_{  \dot{\mathcal{B}}_{1,\infty}^{0,1}(\R^3)   }^{\frac{1}{2}}
\end{align}
for all $t>0$ and $f \in \dot{\mathcal{B}}_{1,\infty}^{0,0}(\R^3) \cap \dot{\mathcal{B}}_{1,\infty}^{0,3}(\R^3)$.
\end{lemm}

\begin{proof}
We first consider $L^{\infty}$-estimate.
It follows from Lemma \ref{dis-est-pm} that 
we may decompose 
\begin{align}
    &
    \n{ e^{t\Deltah}e^{ \pm it \frac{|\nablah|}{|\nabla|}}     f}_{L^{\infty}(\R^3)} 
    \leq 
    \sum_{j,k\in \mathbb{Z}}
    \n{ \Deltahh_j\Deltav_ke^{t\Deltah}e^{ \pm  it \frac{|\nablah|}{|\nabla|}}     f}_{L^{\infty}(\R^3)}\\
    & \quad \leq 
    C\sum_{k \geq j+1} 2^{2j}2^{k}e^{-c2^{2j} t } 
    \sp{1+2^{-(k-j)} t}^{-\frac{3}{2}}
    \n{\Deltahh_j\Deltav_k f}_{L^1(\R^3)} \\
    & \qquad +
    C\sum_{k \leq j } 2^{2j}2^{k}e^{-c2^{2j} t } 
     \sp{1+2^{2(k-j)}t}^{-\frac{3}{2}}
    \n{\Deltahh_j\Deltav_k f}_{L^1(\R^3)}\\
    &\quad=:CJ_1(t)+CJ_2(t).
\end{align}
For the estimate of $J_1(t)$, we have 
\begin{align}
    J_1(t)
    \leq{}&
    \sum_{k \geq j+1} 2^{2j}2^{k}e^{-c2^{2j} t } 
    \sp{2^{-(k-j)} t}^{-\frac{3}{2}}
    \n{\Deltahh_j\Deltav_k f}_{L^1(\R^3)}\\
    ={}&
    \sum_{k \geq j+1} 2^{\frac{1}{2}j}2^{\frac{5}{2}k}e^{-c2^{2j} t } 
    t^{-\frac{3}{2}}
    \n{\Deltahh_j\Deltav_k f}_{L^1(\R^3)}\\
    \leq{}&
    C
    t^{-\frac{3}{2}-\frac{1}{4}}
    \sum_{k \in \mathbb{Z}}
    2^{\frac{5}{2}k}
    \sup_{j \in \mathbb{Z}}\n{\Deltahh_j\Deltav_k f}_{L^1(\R^3)}\\
    \leq{}&
    C
    t^{-\frac{3}{2}-\frac{1}{4}}
    \n{f}_{\dot{\mathcal{B}}_{1,\infty}^{0,1}(\R^3)}^{\frac{1}{4}}
    \n{f}_{\dot{\mathcal{B}}_{1,\infty}^{0,3}(\R^3)}^{\frac{3}{4}},
\end{align}
where we have invoked Lemma \ref{lemm:inter}. 
For the other part, we see that for $0<\vartheta<1/2$
\begin{align}
    J_2(t)
    \leq{}&
    \sum_{k \leq j } 2^{2j}2^{k}e^{-c2^{2j} t } 
     \sp{1+2^{2(k-j)}t}^{-\vartheta } 
    \n{\Deltahh_j\Deltav_k f}_{L^1(\R^3)}\\
    \leq{}&
    \sum_{k \leq j } 2^{2j}2^{k}e^{-c2^{2j} t } 
    \sp{2^{2(k-j)}t}^{ -\vartheta }
    \n{\Deltahh_j\Deltav_k f}_{L^1(\R^3)}\\
    ={}& t^{-\vartheta}
    \sum_{k \leq j } 2^{2j(1+\vartheta)}
    e^{-c2^{2j} t }  
    2^{ (1-2\vartheta)k     } 
    \n{\Deltahh_j\Deltav_k f}_{L^1(\R^3)}\\
    \leq{}&
    C t^{-\vartheta} t^{-(1+\vartheta) }  
    \sum_{k \in \mathbb{Z}} 
    2^{ (1-2\vartheta)k     } 
    \sup_{j \in \mathbb{Z}}
    \n{\Deltahh_j\Deltav_k f}_{L^1(\R^3)}\\
    \leq{}&
    C t^{-\vartheta} t^{-(1+\vartheta) }  
    \n{f}_{\dot{\mathcal{B}}_{1,\infty}^{0,0}(\R^3)}^{2\vartheta}
    \n{f}_{\dot{\mathcal{B}}_{1,\infty}^{0,1}(\R^3)}^{1-2\vartheta}.
\end{align}
Combining the above two estimates, we obtain the desired $L^{\infty}$-estimate.

For $L^2$-estimate, we see by the Plancherel theorem that
\begin{align}\label{L^2-L^2}
    \n{ e^{t\Deltah}e^{ \pm  it \frac{|\nablah|}{|\nabla|}}     f}_{L^{2}(\R^3)}
    &  
    \leq
    C
    \n{ e^{t\Deltah} f}_{L^{2}(\R^3)} 
    \\
    &   
    \leq
    C
    \sum_{j \in \mathbb{Z}}
    e^{-c2^{2j}t}
    2^j
    \n{\Deltahh_j f}_{L^1_{\xh} L^2_{x_3}(\R^3) }
    \\
    & 
    \leq  
    C 
    t^{-\frac{1}{2}}
    \sup_{j\in \mathbb{Z}}
    \n{ \Deltahh_j f}_{ L^1_{\xh} L^2_{x_3}(\R^3)      }  
    \\
    &
    \leq 
    C 
    t^{-\frac{1}{2}}
    \sum_{k \in \mathbb{Z}}
    2^{\frac{1}{2}k} 
    \sup_{j \in \mathbb{Z}}
    \n{  \Deltahh_j \Deltav_k f}_{  L^1(\R^3)   }
    \\
    &
    \leq 
    C 
    t^{-\frac{1}{2}}
    \n{f}_{  \dot{\mathcal{B}}_{1,\infty}^{0,0}(\R^3)   }^{\frac{1}{2}}
    \n{f}_{  \dot{\mathcal{B}}_{1,\infty}^{0,1}(\R^3)   }^{\frac{1}{2}}.
\end{align}
Thus, we complete the proof.
\end{proof}
From Lemma \ref{dis-est-pm} and the interpolation inequalities, 
we immediately obtain the following corollary.
\begin{cor}\label{cor:disp}
    There exists a positive constant $C$ such that
    \begin{align}
        \n{e^{t\Deltah}e^{\pm i t \frac{|\nablah|}{|\nabla|}}f}_{L^p(\R^3)}
        \leq
        C
        t^{-(1-\frac{1}{p})}
        t^{-\frac{3}{4}(1-\frac{2}{p})}
        \n{f}_{\dot{\mathcal{B}}_{1,\infty}^{0,0}(\R^3) \cap \dot{\mathcal{B}}_{1,\infty}^{0,3}(\R^3)}
    \end{align}
    for all $t >0$, $2 \leq p \leq \infty$, and $f \in \dot{\mathcal{B}}_{1,\infty}^{0,0}(\R^3) \cap \dot{\mathcal{B}}_{1,\infty}^{0,3}(\R^3)$.
\end{cor}
\subsection{Proof of Theorem \ref{thm:lin}}
Now, we are ready to prove Theorem \ref{thm:lin}.
\begin{proof}[Proof of Theorem \ref{thm:lin}]
We notice that since it holds
\begin{align}
    \n{(\Ph \vh^{\rm lin},\vh^{\rm lin} - \Ph \vh^{\rm lin},v_3^{\rm lin},\theta^{\rm lin})(t)}_{L^p(\R^3)}
    &
    \leq{}
    C
    \n{(v^{\rm lin},\theta^{\rm lin})(t)}_{H^2(\R^3)}
    \\
    &={}
    C
    \n{(v_0,\theta_0)}_{H^2(\R^2)}.
\end{align}
Thus, we only focus on the case of $t>1$.

First, we focus on the estimate of $\n{\Ph\vh ^{\rm lin}(t)}_{L^p(\R^3)}$, which equals to $\n{e^{t\Deltah} \Ph v_{0,\rm h}}_{L^p(\R^3)}$. 
{It follows from Lemma \ref{lemm:heat-kernel} that}
\begin{align}
    \n{ e^{t\Deltah}   \Ph v_{0,\rm h} }_{L^{2}(\R^3)}
    \leq 
    C 
    t^{-\frac{1}{2}}
    \n{v_{0,\rm h}}_{  \dot{\mathcal{B}}_{1,\infty}^{0,0}(\R^3)   }^{\frac{1}{2}}
    \n{v_{0,\rm h}}_{  \dot{\mathcal{B}}_{1,\infty}^{0,1}(\R^3)   }^{\frac{1}{2}}   
    \leq 
      C 
    t^{-\frac{1}{2}}
    \n{v_0}_{ L^1_{\xh}W^{1,1}_{x_3}(\R^3)  }, \\
    \n{ e^{t\Deltah}   \Ph v_{0,\rm h} }_{L^{\infty}(\R^3)}
    \leq 
    C 
    t^{-1}   
    \n{ v_{0,\rm h}}_{  \dot{\mathcal{B}}_{1,\infty}^{0,0}(\R^3)   }^{\frac{1}{2}}
    \n{ v_{0,\rm h}}_{  \dot{\mathcal{B}}_{1,\infty}^{0,2}(\R^3)   }^{\frac{1}{2}}  
    \leq 
      C 
    t^{-1}
    \n{v_0}_{ L^1_{\xh}W^{2,1}_{x_3}(\R^3)  } .
\end{align}
Interpolating the above estimates gives rise to 
\begin{align}
    \n{\Ph \vh^{\rm lin}(t)}_{L^p(\R^3)}
    \leq 
    C
    t^{ -(1-\frac{1}{p})} 
    \n{(v_0,\theta_0)}_{L^1_{\xh}W^{2,1}_{x_3}(\R^3)  } 
\end{align}
for all $2\leq p \leq \infty$ and $t>0$.

For the estimate of  { $\vh^{\rm lin}(t)-\Ph \vh^{\rm lin} (t)$} and { $\theta^{\rm lin}(t)$},
we have by Corollary \ref{cor:disp} that
\begin{align}
    &
    \n{\vh^{\rm lin}(t)-\Ph \vh^{\rm lin} (t)}_{L^p(\R^3)}
    +
    \n{\theta^{\rm lin} (t)}_{L^p(\R^3)}
    \\
    &\quad
    \leq{}
    C
    t^{-(1-\frac{1}{p})}
    t^{-\frac{3}{4}(1-\frac{2}{p})}
    \sum_{\sigma \in \{\pm\}}
    \n{
    \sp{
    Q_{\sigma}^{\rm vel,h}(v_0,\theta_0),
    Q_{\sigma}^{\rm temp}(v_0,\theta_0)
    }
    }_{\dot{\mathcal{B}}_{1,\infty}^{0,0}(\R^3) \cap \dot{\mathcal{B}}_{1,\infty}^{0,3}(\R^3)}
    \\
    &\quad
    \leq{}
    C
    t^{-(1-\frac{1}{p})}
    t^{-\frac{3}{4}(1-\frac{2}{p})}
    \n{(v_0,\theta_0)}_{L^1_{\xh}W^{3,1}_{x_3}(\R^3)}.
\end{align}

We need to modify the above argument when we consider the decay estimate of $\n{v^{\rm lin}_3(t)}_{L^p(\R^3)}$ by taking account of the enhanced dissipation. 
We first notice that  
\begin{align}\label{vel-v}
    \partial_{x_3}
    Q_{\pm}^{\rm vel,v}(v_0,\theta_0)
    =
    |\nablah|
    \widetilde{Q_{\pm}^{\rm vel,v}}(v_0,\theta_0),
\end{align}
where
\begin{align}
    \widetilde{Q_{\pm}^{\rm vel,v}}(v_0,\theta_0)
    :=
    -
    \frac{1}{2}
    \Mp{
    \frac{\partial_{x_3}^2}{|\nabla|^2}
    \sp{
    \frac{\nablah}{|\nablah|}
    \cdot 
    v_{0,{\rm h}}
    }
    -
    \frac{\partial_{x_3}}{|\nabla|}
    \frac{\nablah}{|\nabla|}
    \cdot 
    \frac{\nablah}{|\nablah|}
    v_{0,3}
    \pm 
    i
    \frac{\partial_{x_3}}{|\nabla|}
    \theta_0
    }.
\end{align}
For $L^{\infty}$-estimate, we deduce from Lemma \ref{dis-est-pm} that 
\begin{align}
    &
    \n{v_3^{\rm lin}(t)}_{L^{\infty}(\R^3)}
    \leq
    \sum_{\sigma \in \{\pm\}}
    \n{ e^{\frac{t}{2}\Deltah}e^{ it \frac{|\nablah|}{|\nabla|}}     
    \sp{e^{\frac{t}{2}\Deltah}Q_{\sigma}^{\rm vel,v}(v_0,\theta_0)}}_{L^{\infty}(\R^3)} 
    \\
    &
    \leq 
    \sum_{\sigma \in \{\pm \}}
    \sum_{j,k\in \mathbb{Z}}
    \n{ e^{\frac{t}{2}\Deltah}e^{ it \frac{|\nablah|}{|\nabla|}}   
    \Deltahh_j \Deltav_k 
    \sp{e^{\frac{t}{2}\Deltah}Q_{\sigma}^{\rm vel,v}(v_0,\theta_0)}}_{L^{\infty}(\R^3)}\\
    & \leq 
    C
    \sum_{\sigma \in \{\pm \}}
    \sum_{k \geq j+1} 2^{2j}e^{-c2^{2j} t } 
    \sp{1+2^{-(k-j)} t}^{-\frac{3}{2}}
    \n{\Deltahh_j \Deltav_k e^{\frac{t}{2}\Deltah}\partial_{x_3} Q_{\sigma}^{\rm vel,v}(v_0,\theta_0)}_{L^1(\R^3)} \\
    & \quad +
    C
    \sum_{\sigma \in \{\pm \}}
    \sum_{k \leq j } 2^{2j}2^{k}e^{-c2^{2j} t } 
    \sp{1+2^{2(k-j)}t}^{-\frac{3}{2}}
    \n{\Deltahh_j\Deltav_k e^{\frac{t}{2}\Deltah} Q_{\sigma}^{\rm vel,v}(v_0,\theta_0)}_{L^1(\R^3)}\\
    &=:CJ_1(t)+CJ_2(t).
\end{align}
For the estimate of $J_1(t)$, we have 
\begin{align}
    J_1(t)
    \leq{}&
    \sum_{k \geq j+1} 2^{2j}e^{-c2^{2j} t } 
    \sp{2^{-(k-j)} t}^{-\frac{3}{2}}
    \n{\Deltahh_j\Deltav_k 
    e^{\frac{t}{2}\Deltah}
    \partial_{x_3}
    Q_{\sigma}^{\rm vel,v}(v_0,\theta_0)}_{L^1(\R^3)}\\
    ={}&
    \sum_{k \geq j+1} 2^{\frac{1}{2}j}2^{\frac{5}{2}k}e^{-c2^{2j} t } 
    t^{-\frac{3}{2}}
    \n{\Deltahh_j\Deltav_k 
    e^{\frac{t}{2}\Deltah}|\nablah| \widetilde{Q_{\sigma}^{\rm vel,v}}(v_0,\theta_0)}_{L^1(\R^3)}\\
    \leq{}&
    C
    t^{-\frac{3}{2}-\frac{1}{4}-\frac{1}{2}}
    \sum_{k \in \mathbb{Z}}
    2^{\frac{3}{2}k}
    \sup_{j \in \mathbb{Z}}\n{\Deltahh_j\Deltav_k \widetilde{Q_{\sigma}^{\rm vel,v}}(v_0,\theta_0)}_{L^1(\R^3)}\\
    \leq{}&
    C
    t^{-\frac{9}{4}}
    \n{(v_0,\theta_0)}_{\dot{\mathcal{B}}_{1,\infty}^{0,1}(\R^3)}^{\frac{1}{2}}
    \n{(v_0,\theta_0)}_{\dot{\mathcal{B}}_{1,\infty}^{0,2}(\R^3)}^{\frac{1}{2}}  \\
      \leq{}& 
      C 
    t^{-\frac{9}{4}}
     \n{(v_0,\theta_0)}_{  L^1_{\xh}W^{2,1}_{x_3}(\R^3)   } .  
\end{align}
For the other part, we see that for $0<\vartheta<1/2$
\begin{align}
    &
    J_2(t)
    \leq{}
    \sum_{k \leq j } 2^{2j}2^{k}e^{-c2^{2j} t } 
     \sp{1+2^{2(k-j)}t}^{-\vartheta } 
    \n{\Deltahh_j\Deltav_k e^{\frac{t}{2}\Deltah} Q_{\sigma}^{\rm vel,v}(v_0,\theta_0)}_{L^1(\R^3)}\\
    &\leq{}
    \sum_{k \leq j } 2^{2j}2^{k}e^{-c2^{2j} t } 
    \sp{2^{2(k-j)}t}^{ -\vartheta }
    \n{\Deltahh_j\Deltav_k e^{\frac{t}{2}\Deltah} Q_{\sigma}^{\rm vel,v}(v_0,\theta_0)}_{L^1(\R^3)}\\
    &={} t^{-\vartheta}
    \sum_{k \leq j } 2^{2j(1+\vartheta)}
    e^{-c2^{2j} t }  
    2^{ (1-2\vartheta)k     } 
    \n{\Deltahh_j\Deltav_k e^{\frac{t}{2}\Deltah} Q_{\sigma}^{\rm vel,v}(v_0,\theta_0)}_{L^1(\R^3)}\\
    &\leq{}
    C t^{-\vartheta} t^{-(1+\vartheta) }  
    \sum_{k \in \mathbb{Z}} 
    2^{ (1-2\vartheta)k     } 
    \sup_{j \in \mathbb{Z}}
    \n{\Deltahh_j\Deltav_k e^{\frac{t}{2}\Deltah} Q_{\sigma}^{\rm vel,v}(v_0,\theta_0)}_{L^1(\R^3)}\\
    &\leq{}
    C t^{-\vartheta} t^{-(1+\vartheta) }
    \n{ Q_{\sigma}^{\rm vel,v}(v_0,\theta_0)}_{\dot{\mathcal{B}}_{1,\infty}^{0,0}(\R^3)}^{2\vartheta}
    \n{ e^{\frac{t}{2}\Deltah} \partial_{x_3} Q_{\sigma}^{\rm vel,v}(v_0,\theta_0)}_{\dot{\mathcal{B}}_{1,\infty}^{0,0}(\R^3)}^{1-2\vartheta}\\
    &\leq{}
    C t^{-\vartheta} t^{-(1+\vartheta) }
    \n{(v_0,\theta_0)}_{L^1(\R^3)}^{2\vartheta}
    \n{e^{\frac{t}{2}\Deltah} |\nablah| \widetilde{Q_{\sigma}^{\rm vel,v}}(v_0,\theta_0)}_{\dot{\mathcal{B}}_{1,\infty}^{0,0}(\R^3)}^{1-2\vartheta}\\
    &\leq{}
    C t^{-\vartheta} t^{-(1+\vartheta) } t^{-\frac{1}{2}(1-2\vartheta)} 
    \n{(v_0,\theta_0)}_{L^1(\R^3)}.
\end{align}
By choosing $\vartheta=1/2-\ep$ with $0<\ep <1/4$, it holds that
\begin{align}\label{v3-infty}
    \n{v_3^{\rm lin}(t)}_{L^{\infty}(\R^3)}
    \leq{}&
    C_{\ep}
    \left(  t^{-\frac{9}{4}} +t^{-2+\ep} \right)
    \n{(v_0,\theta_0)}_{  L^1_{\xh}W^{3,1}_{x_3}(\R^3)   }\\
    \leq{}&
    C_{\ep}
    t^{-2+\ep} 
    \n{(v_0,\theta_0)}_{  L^1_{\xh}W^{3,1}_{x_3}(\R^3)   }.
\end{align}
For the $L^2$-estimate, we see that
\begin{align}  
    \n{v_3^{\rm lin}(t)}_{L^2(\R^3)}  &
    \leq{}
    \sum_{\sigma \in \{\pm \}}
    \n{
    e^{\frac{t}{2}\Deltah}
    \sp{
    e^{\frac{t}{2}\Deltah}
    Q_{\sigma}^{\rm vel,v}(v_0,\theta_0)
    }
    }_{L^2 (\R^3) }\\
    &\leq{}
    C
    \sum_{\sigma \in \{\pm \}}
    \sum_{j \in \mathbb{Z}}
    e^{-c2^{2j}t}
    \sum_{k \in \mathbb{Z}}
    \n{
    \Deltahh_j
    \Deltav_k
    e^{\frac{t}{2}\Deltah}
    Q_{\sigma}^{\rm vel,v}(v_0,\theta_0)
    }_{L^2 (\R^3) }\\
    &\leq{}
    C
    \sum_{\sigma \in \{\pm \}}
    \sum_{j \in \mathbb{Z}}
    e^{-c2^{2j}t}
    2^{j}
    \sum_{k \in \mathbb{Z}}
    2^{\frac{1}{2}k}
    \n{
    \Deltahh_j
    \Deltav_k
    e^{\frac{t}{2}\Deltah}
    Q_{\sigma}^{\rm vel,v}(v_0,\theta_0)
    }_{L^1 (\R^3) }\\
    &\leq{}
    C
    t^{-\frac{1}{2}}
    \sum_{\sigma \in \{\pm \}}
    \sup_{j \in \mathbb{Z}}
    \sum_{k \in \mathbb{Z}}
    2^{\frac{1}{2}k}
    \n{
    \Deltahh_j
    \Deltav_k
    e^{\frac{t}{2}\Deltah}
    Q_{\sigma}^{\rm vel,v}(v_0,\theta_0)
    }_{L^1 (\R^3)  }\\
    &\leq{}
    C
    t^{-\frac{1}{2}}
    \sum_{\sigma \in \{\pm \}}
    \n{
    e^{\frac{t}{2}\Deltah}
    Q_{\sigma}^{\rm vel,v}(v_0,\theta_0)
    }_{\dot{\mathcal{B}}_{1,\infty}^{0,0} (\R^3) }^{\frac{1}{2}}  \label{lin-v3-L^2}      \\
    & \qquad  \qquad \qquad \qquad 
    \times  
    \n{
    e^{\frac{t}{2}\Deltah}
    \partial_{x_3}Q_{\sigma}^{\rm vel,v}(v_0,\theta_0)
    }_{\dot{\mathcal{B}}_{1,\infty}^{0,0} (\R^3)  }^{\frac{1}{2}}\\
    &={}
    C
    t^{-\frac{1}{2}}
    \sum_{\sigma \in \{\pm \}}
    \n{
    e^{\frac{t}{2}\Deltah}
    Q_{\sigma}^{\rm vel,v}(v_0,\theta_0)
    }_{\dot{\mathcal{B}}_{1,\infty}^{0,0} (\R^3) }^{\frac{1}{2}} \\
    & \qquad  \qquad \qquad \qquad 
    \times
    \n{ 
    e^{\frac{t}{2}\Deltah}
    |\nablah|\widetilde{Q_{\sigma}^{\rm vel,v}}(v_0,\theta_0)
    }_{ \dot{\mathcal{B}}_{1,\infty}^{0,0} (\R^3) }^{\frac{1}{2}}\\
    &\leq{}
    C
    t^{-\frac{1}{2}}
    t^{-\frac{1}{4}}
    \n{
    (v_0,\theta_0)
    }_{L^1 (\R^3) }
    =
    C
    t^{-\frac{3}{4}}
    \n{
    (v_0,\theta_0)
    }_{L^1 (\R^3) }.
\end{align}
Interpolating between \eqref{v3-infty} and \eqref{lin-v3-L^2}, we obtain the decay estimate of $\n{v_3(t)}_{L^p(\R^3)}$ for all $2\leq p \leq \infty$.  
We notice that the decay estimates for the first order derivatives of solutions follow in a similar manner. 
Thus, we complete the proof of Theorem \ref{thm:lin}. 
\end{proof}

\section{Nonlinear stability: Proof of Theorem \ref{main-thm}}\label{sec:non}

\subsection{Decomposition of nonlinear terms}
In this subsection, we decompose the nonlinear Duhamel integral of the solution $(v,\theta)$ to \eqref{eq:B2}.
From \eqref{lin-sol-1} and the Duhamel principle, we see that
\begin{align}
    &\Ph \vh (t)   
    ={}
  { \Ph }
  \vh^{\rm lin} (t) 
    - 
    \int_0^t
    e^{(t-\tau)\Deltah} 
    \Ph \nabla \cdot (\vh \otimes v)(\tau) d\tau,\\
    &
    \begin{aligned}
    (\vh-\Ph \vh)(t)
    ={}&
    (\vh^{\rm lin} - \Ph \vh^{\rm lin})(t)
    \\ 
    &-
    \sum_{\sigma \in \{ \pm \} }
    \int_0^t
    e^{(t-\tau)\Deltah}
    e^{\sigma i(t-\tau) \frac{|\nablah|}{|\nabla|}}
    Q_{\sigma}^{\rm vel,h}
    ((v \cdot \nabla)v,v \cdot \nabla \theta)(\tau)d\tau,
    \end{aligned}\\
    &
    v_3(t)
    ={}
    v_3^{\rm lin}(t)
    -
    \sum_{\sigma \in \{ \pm \} }
    \int_0^t
    e^{(t-\tau)\Deltah}
    e^{\sigma i(t-\tau) \frac{|\nablah|}{|\nabla|}}
    Q_{\sigma}^{\rm vel,v}
    ((v \cdot \nabla)v,v \cdot \nabla \theta)(\tau)d\tau,\\
    &
    \theta(t)
    ={}
    \theta^{\rm lin}(t)
    -
    \sum_{\sigma \in \{ \pm \} }
    \int_0^t
    e^{(t-\tau)\Deltah}
    e^{\sigma i(t-\tau) \frac{|\nablah|}{|\nabla|}}
    Q_{\sigma}^{\rm temp}
    ((v \cdot \nabla)v,v \cdot \nabla \theta)(\tau)d\tau.
\end{align}
Here, using  
\begin{align}
    &
    (v\cdot \Grad) v= \Grad \cdot (v \otimes v)
    =
    \begin{pmatrix}
       \nablah \cdot (\vh \otimes \vh)+ \p_{x_3}(\vh v_3)   \\
        \nablah \cdot (v_3 \vh)+ \p_{x_3}(v_3 v_3)
    \end{pmatrix},\\
    &
    v\cdot \Grad \theta = \Grad \cdot (v\theta)
    = \nablah \cdot (\vh\theta) + \partial_{x_3}(v_3\theta),
\end{align}
we may decompose the nonlinear terms as 
\begin{align}
    \int_0^t
    e^{(t-\tau)\Deltah} 
    \Ph \nabla \cdot (\vh \otimes v)(\tau) d\tau
    ={}&
    \int_0^t
    e^{(t-\tau)\Deltah} 
    \Ph \nablah \cdot (\vh \otimes \vh)(\tau) d\tau
    \\
    &+
    \int_0^t
    e^{(t-\tau)\Deltah} 
    \Ph \partial_{x_3} (\vh v_3)(\tau) d\tau\\
    =:{}&
    \sum_{j=1}^2
    \mathcal{D}_j[v](t),
\end{align}
\begin{align}
    &\int_0^t
    e^{(t-\tau)\Deltah}
    e^{\pm i (t-\tau)\frac{|\nablah|}{|\nabla|}}
    Q_{\pm}^{\rm vel,h}
    ((v \cdot \nabla)v,v \cdot \nabla \theta)(\tau)d\tau
    \\
    &\quad=
    -
    \frac{1}{2}
    \int_0^t
    e^{(t-\tau)\Deltah}
    e^{\pm i (t-\tau)\frac{|\nablah|}{|\nabla|}}
    \frac{\nablah}{|\nablah|}
    \frac{\partial_{x_3}^2}{|\nabla|^2}
    \sp{\frac{\nablah}{|\nablah|} \cdot \sp{\nablah \cdot (\vh \otimes \vh)(\tau)}}
    d\tau\\
    &\qquad
    -
    \frac{1}{2}
    \int_0^t
    e^{(t-\tau)\Deltah}
    e^{\pm i (t-\tau)\frac{|\nablah|}{|\nabla|}}
    \frac{\nablah}{|\nablah|}
    \frac{\partial_{x_3}^2}{|\nabla|^2}
    \sp{\frac{\nablah}{|\nablah|} \cdot \partial_{x_3}(\vh v_3)(\tau)}
    d\tau\\
    &\qquad
    +
    \frac{1}{2}
    \int_0^t
    e^{(t-\tau)\Deltah}
    e^{\pm i (t-\tau)\frac{|\nablah|}{|\nabla|}}
    \frac{\nablah}{|\nabla|}
    \frac{\partial_{x_3}}{|\nabla|}
    \nablah \cdot (v_3\vh)(\tau)
    d\tau\\
    &\qquad
    +
    \frac{1}{2}
    \int_0^t
    e^{(t-\tau)\Deltah}
    e^{\pm i (t-\tau)\frac{|\nablah|}{|\nabla|}}
    \frac{\nablah}{|\nabla|}
    \frac{\partial_{x_3}}{|\nabla|}
    \partial_{x_3}(v_3^2)(\tau)
    d\tau\\
    &\qquad 
    \pm 
    \frac{i}{2}
    \int_0^t
    e^{(t-\tau)\Deltah}
    e^{\pm i (t-\tau)\frac{|\nablah|}{|\nabla|}}
    \frac{\nablah}{|\nablah|}
    \frac{\partial_{x_3}}{|\nabla|}
    \nablah\cdot(\vh\theta)(\tau)
    d\tau\\
    &\qquad 
    \pm
    \frac{i}{2}
    \int_0^t
    e^{(t-\tau)\Deltah}
    e^{\pm i (t-\tau)\frac{|\nablah|}{|\nabla|}}
    \frac{\nablah}{|\nablah|}
    \frac{\partial_{x_3}}{|\nabla|}
    \partial_{x_3}(v_3\theta)(\tau)
    d\tau\\
    &\quad
    =:\sum_{j=1}^6
    \mathcal{D}_{\pm,j}^{\rm vel,h}[v,\theta](t),\\
    &\int_0^t
    e^{(t-\tau)\Deltah}
    e^{\pm i (t-\tau)\frac{|\nablah|}{|\nabla|}}
    Q_{\pm}^{\rm vel,v}
    ((v \cdot \nabla)v,v \cdot \nabla \theta)(\tau)d\tau
    \\
    &\quad=
    -
    \frac{1}{2}
    \int_0^t
    e^{(t-\tau)\Deltah}
    e^{\pm i (t-\tau)\frac{|\nablah|}{|\nabla|}}
    \frac{\partial_{x_3}}{|\nabla|}
    \sp{\frac{\nablah}{|\nabla|} \cdot \sp{\nablah \cdot (\vh \otimes \vh)(\tau)}}
    d\tau\\
    &\qquad
    -
    \frac{1}{2}
    \int_0^t
    e^{(t-\tau)\Deltah}
    e^{\pm i (t-\tau)\frac{|\nablah|}{|\nabla|}}
    \frac{\partial_{x_3}}{|\nabla|}
    \sp{\frac{\nablah}{|\nabla|} \cdot \partial_{x_3} (\vh v_3)(\tau)}
    d\tau\\
    &\qquad
    -
    \frac{1}{2}
    \int_0^t
    e^{(t-\tau)\Deltah}
    e^{\pm i (t-\tau)\frac{|\nablah|}{|\nabla|}}
    \frac{|\nablah|^2}{|\nabla|^2}
    \nablah \cdot (v_3\vh)(\tau)
    d\tau\\
    &\qquad
    -
    \frac{1}{2}
    \int_0^t
    e^{(t-\tau)\Deltah}
    e^{\pm i (t-\tau)\frac{|\nablah|}{|\nabla|}}
    \frac{|\nablah|^2}{|\nabla|^2}
    \partial_{x_3}(v_3^2)(\tau)
    d\tau\\
    &\qquad 
    \pm 
    \frac{i}{2}
    \int_0^t
    e^{(t-\tau)\Deltah}
    e^{\pm i (t-\tau)
    \frac{|\nablah|}{|\nabla|}}
    \frac{\nablah}{|\nabla|}
    \cdot 
    \frac{\nablah}{|\nablah|}
    \nablah\cdot(\vh\theta)(\tau)
    d\tau\\
    &\qquad 
    \pm
    \frac{i}{2}
    \int_0^t
    e^{(t-\tau)\Deltah}
    e^{\pm i (t-\tau)
    \frac{|\nablah|}{|\nabla|}}
    \frac{\nablah}{|\nabla|}
    \cdot 
    \frac{\nablah}{|\nablah|}
    \partial_{x_3}(v_3\theta)(\tau)
    d\tau\\
    &\quad
    =:\sum_{j=1}^6
    \mathcal{D}_{\pm,j}^{\rm vel,v}[v,\theta](t),\\
    &\int_0^t
    e^{(t-\tau)\Deltah}
    e^{\pm i (t-\tau)\frac{|\nablah|}{|\nabla|}}
    Q_{\pm}^{\rm temp}
    ((v \cdot \nabla)v,v \cdot \nabla \theta)(\tau)d\tau
    \\
    &\quad=
    \mp
    \frac{i}{2}
    \int_0^t
    e^{(t-\tau)\Deltah}
    e^{\pm i (t-\tau)\frac{|\nablah|}{|\nabla|}}
    \frac{\partial_{x_3}}{|\nabla|}\sp{
    \frac{\nablah}{|\nablah|} \cdot \sp{\nablah \cdot (\vh \otimes \vh)(\tau)}}
    d\tau\\
    &\qquad
    \mp
    \frac{i}{2}
    \int_0^t
    e^{(t-\tau)\Deltah}
    e^{\pm i (t-\tau)\frac{|\nablah|}{|\nabla|}}
    \frac{\partial_{x_3}}{|\nabla|}\sp{
    \frac{\nablah}{|\nablah|} \cdot \partial_{x_3} (\vh v_3)(\tau)}
    d\tau\\
    &\qquad
    \pm
    \frac{i}{2}
    \int_0^t
    e^{(t-\tau)\Deltah}
    e^{\pm i (t-\tau)\frac{|\nablah|}{|\nabla|}}
    \frac{\nablah}{|\nabla|}
    \cdot 
    \frac{\nablah}{|\nablah|}
    \nablah \cdot (v_3\vh)(\tau)
    d\tau\\
    &\qquad
    \pm
    \frac{i}{2}
    \int_0^t
    e^{(t-\tau)\Deltah}
    e^{\pm i (t-\tau)\frac{|\nablah|}{|\nabla|}}
    \frac{\nablah}{|\nabla|}
    \cdot 
    \frac{\nablah}{|\nablah|}
    \partial_{x_3}(v_3^2)(\tau)
    d\tau\\
    &\qquad 
    +
    \frac{1}{2}
    \int_0^t
    e^{(t-\tau)\Deltah}
    e^{\pm i (t-\tau)
    \frac{|\nablah|}{|\nabla|}}
    \nablah\cdot(\vh\theta)(\tau)
    d\tau\\
    &\qquad 
    +
    \frac{1}{2}
    \int_0^t
    e^{(t-\tau)\Deltah}
    e^{\pm i (t-\tau)
    \frac{|\nablah|}{|\nabla|}}
    \partial_{x_3}(v_3\theta)(\tau)
    d\tau\\
    &\quad
    =:
    \sum_{j=1}^6
    \mathcal{D}_{\pm,j}^{\rm temp}[v,\theta](t).
\end{align}

\subsection{Decay estimates of Duhamel terms}\label{subs:Duh}
For any $0<\ep<1/4$ and $0 \leq T_1 < T_2 \leq \infty$,
let us define the space $Y_{\varepsilon}(T_1,T_2)$ by  
\begin{align}
    Y_{\varepsilon}(T_1,T_2)
    :=
    \Mp{
    (v,\theta) \in C([T_1,T_2);H^{8}(\mathbb{R}^3))
    \ ;\ 
    \nabla \cdot v = 0,\ 
    \n{(v,\theta)}_{Y_{\varepsilon}(T_1,T_2)}
    < \infty
    },
\end{align}
where the norm $\n{(v,\theta)}_{Y_{\varepsilon}(T_1,T_2)}$ is defined by
\begin{align}
    \n{(v,\theta)}_{Y_{\varepsilon}(T_1,T_2)}
    :={}&
    \sup_{T_1 \leq t < T_2}
    \n{(v,\theta)(t)}_{H^{8}(\R^3)}\\
    &
    +
    \sup_{T_1 \leq t < T_2}
    (1+t)
    \n{(\vh,\theta)(t)}_{\dot{\mathcal{B}}_{2,1}^{1,\frac{1}{2}}(\R^3)}
    \\
    &
    +
    \sup_{T_1 \leq t < T_2}
    (1+t)^{\frac{5}{4}}
    \n{v_3(t)}_{\dot{\mathcal{B}}_{2,1}^{1,\frac{1}{2}}(\R^3)}
    \\
    &
    +
    \sum_{k=0}^4
    \sum_{ |\alphah| \leq 1}
    \sup_{T_1 \leq t < T_2}
    (1+t)^{\frac{1}{2}+\frac{|\alphah|}{2}}
    \n{\nablah^{\alphah}\partial_{x_3}^k(\vh,\theta)(t)}_{L^2(\R^3)}\\
    &+
    \sum_{k=0}^4
    \sum_{ |\alphah| \leq 1}
    \sup_{T_1 \leq t < T_2}
    (1+t)^{\frac{3}{4}+\frac{|\alphah|}{2}}
    \n{\nablah^{\alphah}\partial_{x_3}^kv_3(t)}_{L^2(\R^3)}\\
    &+
    \sum_{|\alpha|\leq 1}
    \sup_{2\leq p \leq \infty}
    \sup_{T_1 \leq t < T_2}
    (1+t)^{(1-\frac{1}{p})+\frac{|\alphah|}{2}}
    \n{\nabla^{\alpha} \vh (t)}_{L^p(\R^3)}\\
    &+
    \sum_{|\alphah|\leq 1}
    \sup_{2\leq p \leq \infty}
    \sup_{T_1 \leq t < T_2}
    (1+t)^{(1-\frac{1}{p})+\frac{|\alphah|}{2}+\frac{1}{4}+A_{\varepsilon}(p)}
    \n{ \nablah^{\alphah}v_3(t)}_{L^p(\R^3)}\\
    &+
    \sum_{|\alpha|\leq 1}
    \sup_{2\leq p \leq \infty}
    \sup_{T_1 \leq t < T_2}
    (1+t)^{(1-\frac{1}{p})+\frac{|\alphah|}{2}+A_0(p)}
    \n{\nabla^{\alpha} \theta (t)}_{L^p(\R^3)},
\end{align}
where  
\begin{align}
    A_{\varepsilon}(p)
    :={}&
    \min
    \Mp{
    \sp{\frac{3}{4}-\varepsilon}
    \sp{1-\frac{2}{p}},
    {\frac{1}{4}}
    }.
\end{align}
\begin{rem}
    Let us provide some remarks on the definition of $Y_{\varepsilon}(T_1,T_2)$-norm.
    \begin{enumerate}
        \item 
        We see that the $Y_{\varepsilon}(T_1,T_2)$-norm contains the following auxiliary terms:
        \begin{align}\label{aux-1}
            \sup_{T_1 \leq t < T_2}
            (1+t)
            \n{(\vh,\theta)(t)}_{\dot{\mathcal{B}}_{2,1}^{1,\frac{1}{2}}(\R^3)},
            \quad
            \sup_{T_1 \leq t < T_2}
            (1+t)^{\frac{5}{4}}
            \n{v_3(t)}_{\dot{\mathcal{B}}_{2,1}^{1,\frac{1}{2}}(\R^3)},
        \end{align}
        and
        \begin{align}\label{aux-2}
            \begin{split}
            &
            \sum_{k=0}^4
            \sum_{ |\alphah| \leq 1}
            \sup_{T_1 \leq t < T_2}
            (1+t)^{\frac{1}{2}+\frac{|\alphah|}{2}}
            \n{\nablah^{\alphah}\partial_{x_3}^k(\vh,\theta)(t)}_{L^2(\R^3)},
            \\
            &
            \sum_{k=0}^4
            \sum_{ |\alphah| \leq 1}
            \sup_{T_1 \leq t < T_2}
            (1+t)^{\frac{3}{4}+\frac{|\alphah|}{2}}
            \n{\nablah^{\alphah}\partial_{x_3}^kv_3(t)}_{L^2(\R^3)}.
            \end{split}
        \end{align}
        To estimate some nonlinear Duhamel integrals such as $\nablah \mathcal{D}_{\pm,1}^{\rm vel,h}$ in $L^{\infty}$ norm, 
        we need the decay information of $(v,\theta)$ in anisotropic Besov norm \eqref{aux-1}.
        In order to bound the Duhamel nonlinear terms by dispersive estimates, 
        we meet the following type calculation:
        \begin{align}
            &
            \n{
            e^{(t-\tau)\Deltah}e^{\pm i(t-\tau) \frac{|\nablah|}{|\nabla|}}\nabla^{\alpha}\nablah(v_kv_{\ell})(\tau)
            }_{L^p(\R^3)}
            \\
            &\quad
            \leq 
            C(t-\tau)^{-(1-\frac{1}{p})-\frac{|\alphah|}{2}-\frac{3}{4}(1-\frac{2}{p})}
            \n{\nablah(v_kv_{\ell})(\tau)}_{L^1_{\xh}W^{3+\alpha_3,1}{(\R^3) }   }  \\
            &\quad
            \leq 
            C(t-\tau)^{-(1-\frac{1}{p})-\frac{|\alphah|}{2}-\frac{3}{4}(1-\frac{2}{p})}\\
            &\qquad
            \times
            \sum_{m=1}^{3+\alpha_3}\n{\nablah \partial_{x_3}^mv(\tau)}_{L^2(\R^3)}
            \sum_{m=1}^{3+\alpha_3}\n{\partial_{x_3}^mv(\tau)}_{L^2(\R^3)}.
        \end{align}
        Then, we need \eqref{aux-2} for calculating the above right hand side.
        \item 
        Although there is no information on $\partial_{x_3}^5v_3$ in $Y_{\varepsilon}(T_1,T_2)$-norm, we may treat it by using the divergence free condition $\partial_{x_3}v_3=-\nablah \cdot \vh$ as follows:
        \begin{align}
        \n{\partial_{x_3}^5v_3(t)}_{L^2(\R^3)}
        =
        \n{\partial_{x_3}^4(\nablah \cdot \vh)(t)}_{L^2(\R^3)}
        \leq
        C
        \n{(v,\theta)}_{Y_{\varepsilon}(T_1,T_2)}
        (1+t)^{-1}
    \end{align} 
    for all {$T_1\leq t<T_2$}.
    \end{enumerate}
\end{rem}
Using this norm, we deduce several estimates for the Duhamel terms defined in the above subsection. We start with the Duhamel estimates in the anisotropic Besov space $\dot{\mathcal{B}}_{2,1}^{1,\frac{1}{2}}(\R^3)$.  
\begin{lemm}\label{lemm:B} 
Let $(v,\theta) \in Y_{\varepsilon}(0,T)$ with some $0<\varepsilon<1/4$ and $0<T\leq \infty$. Then, there exists a positive constant $C$ such that      
it holds    
\begin{align}
    &
    \begin{aligned}
    &
    \sum_{j=1,5}
    \sp{\n{\mathcal{D}_{\pm,j}^{\rm vel,h}[v,\theta](t)}_{\dot{\mathcal{B}}_{2,1}^{1,\frac{1}{2}}(\R^3)}
    +
    \n{\mathcal{D}_{\pm,j}^{\rm temp}[v,\theta](t)}_{\dot{\mathcal{B}}_{2,1}^{1,\frac{1}{2}} (\R^3)}
  {  +
    \n{\mathcal{D}_{\pm,j}^{\rm vel,v}[v,\theta](t)}_{\dot{\mathcal{B}}_{2,1}^{1,\frac{1}{2}}(\R^3)}  } 
    }\\
    &\quad
    +
    \n{\mathcal{D}_{1}[v](t)}_{\dot{\mathcal{B}}_{2,1}^{1,\frac{1}{2}}(\R^3)}
    \leq{}
   C 
    \n{(v,\theta)}_{Y_{\varepsilon}(0,T)}^2
    t^{-\frac{3}{2}}
    \log (2+t),
    \end{aligned}\\
    &
    \begin{aligned}
    &
    \sum_{j=2,6}
    \sp{\n{\mathcal{D}_{\pm,j}^{\rm vel,h}[v,\theta](t)}_{\dot{\mathcal{B}}_{2,1}^{1,\frac{1}{2}} (\R^3)}+\n{\mathcal{D}_{\pm,j}^{\rm temp}[v,\theta](t)}_{\dot{\mathcal{B}}_{2,1}^{1,\frac{1}{2}}(\R^3) }}\\
    &\quad
    +
    \n{\mathcal{D}_{2}[v](t)}_{\dot{\mathcal{B}}_{2,1}^{1,\frac{1}{2} } (\R^3)}
    \leq
  C
    \n{(v,\theta)}_{Y_{\varepsilon}(0,T)}^2
    t^{-1},
    \end{aligned}\\
    &
    \begin{aligned}
    &
    \sum_{j=3,4}
    \sp{\n{\mathcal{D}_{\pm,j}^{\rm vel,h}[v,\theta](t)}_{\dot{\mathcal{B}}_{2,1}^{1,\frac{1}{2}}(\R^3)}
    +
    \n{\mathcal{D}_{\pm,j}^{\rm temp}[v,\theta](t)}_{\dot{\mathcal{B}}_{2,1}^{1,\frac{1}{2}} (\R^3) }}
    \leq 
  C
    \n{(v,\theta)}_{Y_{\varepsilon}(0,T)}^2
    t^{-\frac{3}{2}},
    \end{aligned}\\
    &
    \begin{aligned}
    &
 {   \sum_{j=2,3,4,6} }
    \n{\mathcal{D}_{\pm,j}^{\rm vel,v}[v,\theta](t)}_{\dot{\mathcal{B}}_{2,1}^{1,\frac{1}{2}}(\R^3)}
    \leq 
      C 
    \n{(v,\theta)}_{Y_{\varepsilon}(0,T)}^2
    t^{   {  -\frac{3}{2} }  } 
    \end{aligned}\end{align}
for $2 \leq  t < T$.
\end{lemm}
\begin{proof} 

\noindent
{\it Step 1. Estimates of $\{\mathcal{D}_{\pm,j}^{\rm vel,h}[v,\theta]\}_{j=1,5}$, $\{\mathcal{D}_{\pm,j}^{\rm temp}[v,\theta]\}_{j=1,5}$, $\{\mathcal{D}_{\pm,j}^{\rm vel,v}[v,\theta]\}_{j=1,5}$, and $\mathcal{D}_{1}[v]$.}

We only focus on the estimates of $\mathcal{D}_{\pm,1}^{\rm vel,h}[v,\theta]$ as the others are treated similarly.
We see from the Plancherel theorem that 
\begin{align}
    &
    \n{\mathcal{D}_{\pm,1}^{\rm vel,h}[v,\theta](t)}_{\dot{\mathcal{B}}_{2,1}^{1,\frac{1}{2}}(\R^3)}
    \leq{}
    C
    \int_0^t
    \n{e^{(t-\tau)\Deltah}(\vh \otimes \vh)(\tau)}_{\dot{\mathcal{B}}_{2,1}^{2,\frac{1}{2}}(\R^3)}d\tau\\
    &\quad\leq{}
    C
    \int_0^{\frac{t}{2}}
    (t-\tau)^{-\frac{3}{2}}
    \sup_{j\in \mathbb{Z}}
    \sum_{k\in \mathbb{Z}}
    2^{\frac{1}{2}k}
    \n{    {  \Deltahh_j \Deltav_k }  (\vh \otimes \vh)(\tau)}_{L^2_{x_3}L^1_{\xh}(\R^3)}d\tau\\
    &\qquad
    +
    C
    \int_{\frac{t}{2}}^{t}
    (t-\tau)^{-\frac{1}{2}}
    \sup_{j\in \mathbb{Z}}
    \sum_{k\in \mathbb{Z}}
    2^j 
    2^{\frac{1}{2}k}
    \n{   {  \Deltahh_j \Deltav_k } 
    (\vh \otimes \vh)(\tau)}_{L^2(\R^3)}d\tau\\
    &\quad\leq{}
    C
    \int_0^{\frac{t}{2}}
    (t-\tau)^{-\frac{3}{2}}
    \sum_{m=0}^1
    \n{\partial_{x_3}^m    
    (\vh \otimes \vh)(\tau)}_{L^2_{x_3}L^1_{\xh}(\R^3)}d\tau\\
    &\qquad
    +
    C
    \int_{\frac{t}{2}}^{t}
    (t-\tau)^{-\frac{1}{2}}
    \n{(\vh \otimes \vh)(\tau)}_{\dot{\mathcal{B}}_{2,1}^{1,\frac{1}{2}}(\R^3)}d\tau\\
    &\quad\leq{}
    C
    \int_0^{\frac{t}{2}}
    (t-\tau)^{-\frac{3}{2}}
    \sum_{m=0}^1
    \n{\partial_{x_3}^m\vh(\tau)}_{L^2(\R^3)}\n{\vh(\tau)}_{L^2(\R^3)}^{\frac{1}{2}}\n{\partial_{x_3}\vh(\tau)}_{L^2(\R^3)}^{\frac{1}{2}}d\tau\\
    &\qquad
    +
    C
    \int_{\frac{t}{2}}^{t}
    (t-\tau)^{-\frac{1}{2}}
    \n{\vh(\tau)}_{\dot{\mathcal{B}}_{2,1}^{1,\frac{1}{2}}(\R^3)}^2d\tau\\
    &\quad\leq{}
    C
    \n{(v,\theta)}_{Y_{\varepsilon}(0,T)}^2
    \sp{
    \int_0^{\frac{t}{2}}
    (t-\tau)^{-\frac{3}{2}}
    (1+\tau)^{-1}
    d\tau
    +
    \int_{\frac{t}{2}}^{t}
    (t-\tau)^{-\frac{1}{2}}
    (1+\tau)^{-2}d\tau
    }\\
    &\quad\leq{}
    C
    \n{(v,\theta)}_{Y_{\varepsilon}(0,T)}^2
    t^{-\frac{3}{2}}\log t.
\end{align}
This completes the first step.

\noindent
{\it Step 2. Estimates of $\{\mathcal{D}_{\pm,j}^{\rm vel,h}[v,\theta]\}_{j=2,6}$, $\{\mathcal{D}_{\pm,j}^{\rm temp}[v,\theta]\}_{j=2,6}$, and $\mathcal{D}_{2}[v]$.}

We only focus on the estimates of $\mathcal{D}_{\pm,2}^{\rm vel,h}[v,\theta]$ as the others are treated similarly. It follows that 
\begin{align}
    &
    \n{\mathcal{D}_{\pm,2}^{\rm vel,h}[v,\theta](t)}_{\dot{\mathcal{B}}_{2,1}^{1,\frac{1}{2}}(\R^3)}
    \leq{}
    C
    \int_0^t
    \n{e^{(t-\tau)\Deltah} \p_{x_3} (v_3\vh)(\tau)}_{\dot{\mathcal{B}}_{2,1}^{1,\frac{1}{2}}(\R^3)}d\tau\\
    &\quad\leq{}
    C
    \int_0^{\frac{t}{2}}
    (t-\tau)^{-1}
    \sup_{j\in \mathbb{Z}}
    \sum_{k\in \mathbb{Z}}
    2^{\frac{1}{2}k}
    \n{ \p_{x_3} (v_3\vh)(\tau) }_{L^2_{x_3}L^1_{\xh}(\R^3)}d\tau\\
    &\qquad 
    +
    C
    \int_{\frac{t}{2}}^{t}
    \n{ \p_{x_3} (v_3\vh)(\tau)}_{\dot{\mathcal{B}}_{2,1}^{1,\frac{1}{2}}(\R^3)}
    d\tau\\
    &\quad\leq{}
    C
    \int_0^{\frac{t}{2}}
    (t-\tau)^{-1}
    \sum_{m=0}^1
    \n{\partial_{x_3}^{m+1}(v_3 \vh)   (\tau)}_{L^2_{x_3}L^1_{\xh}(\R^3)}d\tau\\
    &\qquad
    +
    C
    \int_{\frac{t}{2}}^{t}
    \n{ \p_{x_3} (v_3\vh)(\tau)}_{\dot{\mathcal{B}}_{2,1}^{1,\frac{1}{2}}(\R^3)}d\tau\\
     &\quad\leq{}
    C 
      \n{(v,\theta)}_{Y_{\varepsilon}(0,T)}^2
    \int_0^{\frac{t}{2}}
    (t-\tau)^{-1}
   (1+\tau)^{-\f{13}{12}}  d\tau\\
    &\qquad
    +
    C
    \int_{\frac{t}{2}}^{t}
    \n{ (v_3\vh)(\tau)}_{\dot{\mathcal{B}}_{2,1}^{1,\frac{3}{2}}(\R^3)}d\tau \\
     &\quad\leq{}
    C \n{(v,\theta)}_{Y_{\varepsilon}(0,T)}^2
    (1+\tau)^{-1}
    + 
     C
    \int_{\frac{t}{2}}^{t}
    \n{ (v_3\vh)(\tau)}_{\dot{\mathcal{B}}_{2,1}^{1,\frac{3}{2}}(\R^3)}d\tau.
\end{align}
Here, we see from Lemmas \ref{lemm:para} and \ref{lemm:emb} and the interpolation inequality that
\begin{align}
 & \int_{\frac{t}{2}}^{t}
    \n{ (v_3\vh)(\tau)}_{\dot{\mathcal{B}}_{2,1}^{1,\frac{3}{2}}(\R^3)}
    d\tau \\
    &
    \quad 
    \leq   
    C
     \int_{\frac{t}{2}}^{t}
     \n{ (v_3\vh)(\tau)}_{\dot{\mathcal{B}}_{2,1}^{1,\frac{1}{2}}(\R^3)}^{\f{9}{11}  }
     \n{ (v_3\vh)(\tau)}_{\dot{\mathcal{B}}_{2,1}^{1,6}(\R^3)}
   ^{\f{2}{11}  }
     d\tau \\
     &
     \quad  
     \leq 
     C
     \int_{\frac{t}{2}}^{t}
     \left(
  \n{ v_3(\tau)}_{\dot{\mathcal{B}}_{2,1}^{1,\frac{1}{2}}(\R^3)} 
  \n{ \vh(\tau)}_{\dot{\mathcal{B}}_{2,1}^{1,\frac{1}{2}}(\R^3)} 
  \right)^{\f{9}{11}  } 
  \n{ (v_3\vh)(\tau)}_{H^8(\R^3)}^{\f{2}{11}  } d\tau \\
  &
     \quad  
     \leq 
     C
     \int_{\frac{t}{2}}^{t}
     \left(
  \n{ v_3(\tau)}_{\dot{\mathcal{B}}_{2,1}^{1,\frac{1}{2}}(\R^3)} 
  \n{ \vh(\tau)}_{\dot{\mathcal{B}}_{2,1}^{1,\frac{1}{2}}(\R^3)} 
  \right)^{\f{9}{11}  } \\
  & \qquad \qquad  
  \times 
  \left(
 \n{v_3(\tau)}_{H^8(\R^3)} \n{\vh(\tau)}_{L^{\infty}(\R^3)}
 +\n{\vh(\tau)}_{H^8(\R^3)} 
  \n{v_3 (\tau)}_{L^{\infty}(\R^3)}
  \right)
  ^{\f{2}{11}  } d\tau \\
  &
   \quad 
    \leq   
    C
     \n{(v,\theta)}_{Y_{\varepsilon}(0,T)}^2
      \int_{\frac{t}{2}}^{t}
      (1+\tau)^{-\f{81}{44}} 
      (1+\tau)^{-\f{2}{11}} d\tau \\
      &
   \quad =
    C
     \n{(v,\theta)}_{Y_{\varepsilon}(0,T)}^2
     t^{-\f{45}{44}}.
\end{align}
This finishes the second step.

\noindent
{\it Step 3. Estimates of $\{\mathcal{D}_{\pm,j}^{\rm vel,h}[v,\theta]\}_{j=3,4}$, $\{\mathcal{D}_{\pm,j}^{\rm temp}[v,\theta]\}_{j=3,4}$, and $\{\mathcal{D}_{\pm,j}^{\rm vel,v}[v,\theta]\}_{j=2,3,4,6}$.}

We only focus on the estimates of $\mathcal{D}_{\pm,3}^{\rm vel,h}[v,\theta]$ as the others are treated similarly. 
It follows that 
\begin{align}
    &
    \n{\mathcal{D}_{\pm,3}^{\rm vel,h}[v,\theta](t)}_{\dot{\mathcal{B}}_{2,1}^{1,\frac{1}{2}}(\R^3)}
    \leq{}
    C
    \int_0^t
    \n{e^{(t-\tau)\Deltah}(v_3 \vh) (\tau)}_{\dot{\mathcal{B}}_{2,1}^{2,\frac{1}{2}}(\R^3)}d\tau\\
    &\quad\leq{}
    C
    \int_0^{\frac{t}{2}}
    (t-\tau)^{-\frac{3}{2}}
    \sup_{j\in \mathbb{Z}}
    \sum_{k\in \mathbb{Z}}
    2^{\frac{1}{2}k}
    \n{     \Deltahh_j \Deltav_k  (v_3\vh)(\tau)}_{L^2_{x_3}L^1_{\xh}(\R^3)}d\tau\\
    &\qquad
    +
    C
    \int_{\frac{t}{2}}^{t}
    (t-\tau)^{-\frac{1}{2}}
    \sup_{j\in \mathbb{Z}}
    \sum_{k\in \mathbb{Z}}
    2^j 
    2^{\frac{1}{2}k}
    \n{     \Deltahh_j \Deltav_k  
    (v_3 \vh)(\tau)}_{L^2(\R^3)}d\tau\\ 
    &\quad\leq{}
    C
    \int_0^{\frac{t}{2}}
    (t-\tau)^{-\frac{3}{2}}
    \sum_{m=0}^1
    \n{\partial_{x_3}^m     
    (v_3 \vh)(\tau)}_{L^2_{x_3}L^1_{\xh}(\R^3)}d\tau\\
    &\qquad
    +
    C
    \int_{\frac{t}{2}}^{t}
    (t-\tau)^{-\frac{1}{2}}
    \n{(v_3 \vh)(\tau)}_{\dot{\mathcal{B}}_{2,1}^{1,\frac{1}{2}}(\R^3)}d\tau \\
    &\quad\leq{}
    C
     \n{(v,\theta)}_{Y_{\varepsilon}(0,T)}^2
    \int_0^{\frac{t}{2}}
    (t-\tau)^{-\frac{3}{2}}
    (1+\tau)^{-\f{13}{12}} 
    d\tau \\
    & \qquad 
    +
    C
    \int_{\frac{t}{2}}^{t}
    (t-\tau)^{-\frac{1}{2}}
    \n{\vh(\tau)}_{\dot{\mathcal{B}}_{2,1}^{1,\frac{1}{2}}(\R^3)}
    \n{v_3(\tau)}_{\dot{\mathcal{B}}_{2,1}^{1,\frac{1}{2}}(\R^3)}
    d\tau\\
    &\quad\leq{}
    C
    \n{(v,\theta)}_{Y_{\varepsilon}(0,T)}^2
    \sp{
    \int_0^{\frac{t}{2}}
    (t-\tau)^{-\frac{3}{2}}
    (1+\tau)^{-\f{13}{12}}
    d\tau
    +
    \int_{\frac{t}{2}}^{t}
    (t-\tau)^{-\frac{1}{2}}
    (1+\tau)^{-\f{9}{4}}
    d\tau
    }\\
    &\quad\leq{}
    C
    \n{(v,\theta)}_{Y_{\varepsilon}(0,T)}^2
    t^{-\frac{3}{2}}.
\end{align}
This finishes the third step. 
\end{proof}

We proceed with the Duhamel estimates in $L^2$.
\begin{lemm}\label{lemm:L^2} 
Let $(v,\theta) \in Y_{\varepsilon}(0,T)$ with some $0<\varepsilon<1/4$ and $2<T\leq \infty$. Then, there exists a positive constant $C_{\ep}$ such that for any $2 \leq p \leq \infty$,
it holds    
\begin{align}
    &
    \begin{aligned}
    &
    \sum_{j=1,5}
    {\sum_{|\alphah| \leq 1}}
    \sum_{k=0}^4
    \sp{\n{\nablah^{\alphah}\partial_{x_3}^k\mathcal{D}_{\pm,j}^{\rm vel,h}[v,\theta](t)}_{L^2(\R^3)}+\n{\nablah^{\alphah}\partial_{x_3}^k\mathcal{D}_{\pm,j}^{\rm temp}[v,\theta](t)}_{L^2(\R^3)}}\\
    &\quad
    +
    {\sum_{|\alphah| \leq 1}}
    \sum_{k=0}^4
    \n{\nablah^{\alphah}\partial_{x_3}^k\mathcal{D}_{1}[v,\theta](t)}_{L^2(\R^3)}
    \leq{}
    C  
    \n{(v,\theta)}_{Y_{\varepsilon}(0,T)}^2
    t^{-\frac{1}{2}-\frac{1+|\alphah|}{2}}
    \log (2+t),
    \end{aligned}\\
    &
    \begin{aligned}
        &
        \sum_{j=2,6}
        {\sum_{|\alphah| \leq 1}}
        \sum_{k=0}^4
        \sp{\n{\nablah^{\alphah}\partial_{x_3}^k\mathcal{D}_{\pm,j}^{\rm vel,h}[v,\theta](t)}_{L^2(\R^3)}+\n{\nablah^{\alphah}\partial_{x_3}^k\mathcal{D}_{\pm,j}^{\rm temp}[v,\theta](t)}_{L^2(\R^3)}}\\
        &\quad
        +
        {\sum_{|\alphah| \leq 1}}
        \sum_{k=0}^4
        \n{\nablah^{\alphah}\partial_{x_3}^k\mathcal{D}_{2}[v,\theta](t)}_{L^2(\R^3)}
        \leq
         C 
        \n{(v,\theta)}_{Y_{\varepsilon}(0,T)}^2
        t^{-\frac{1}{2}-\frac{|\alphah|}{2}},
        \end{aligned}\\
        &
        \begin{aligned}
        &
        \sum_{j=3,4}
        {\sum_{|\alphah| \leq 1}}
        \sum_{k=0}^4
        \sp{\n{\nablah^{\alphah}\partial_{x_3}^k\mathcal{D}_{\pm,j}^{\rm vel,h}[v,\theta](t)}_{L^2(\R^3)}+\n{\nablah^{\alphah}\partial_{x_3}^k\mathcal{D}_{\pm,j}^{\rm temp}[v,\theta](t)}_{L^2(\R^3)}}\\
        &\quad
        \leq 
        C 
        \n{(v,\theta)}_{Y_{\varepsilon}(0,T)}^2
        t^{-\frac{1}{2}-\frac{1+|\alphah|}{2}},
        \end{aligned}\\
        &
        \begin{aligned}
        &
        {  \sum_{j=1,5} }
        {\sum_{|\alpha| \leq 1}}
        \sum_{k=0}^4
        \n{\nabla^{\alpha} \mathcal{D}_{\pm,j}^{\rm vel,v}[v,\theta](t)}_{L^2(\R^3)}
        \leq 
         C  
        \n{(v,\theta)}_{Y_{\varepsilon}(0,T)}^2
        t^{-\frac{1}{2}-\frac{1+|\alpha|}{2}}
         {  \log t  } ,
        \end{aligned}  \\
         &
        \begin{aligned}
        &
        {  \sum_{j=2,3,4,6}  }  
        {\sum_{|\alpha| \leq 1}}
        \sum_{k=0}^4
        \n{\nabla^{\alpha} \mathcal{D}_{\pm,j}^{\rm vel,v}[v,\theta](t)}_{L^2(\R^3)}
        \leq 
         C  
        \n{(v,\theta)}_{Y_{\varepsilon}(0,T)}^2
        t^{-\frac{1}{2}-\frac{1+|\alpha|}{2}}
    \end{aligned}
    \end{align}
for $2 \leq  t < T$.
\end{lemm}
\begin{proof} 

\noindent
{\it Step 1. Estimates of $\{\mathcal{D}_{\pm,j}^{\rm vel,h}[v,\theta]\}_{j=1,5}$, $\{\mathcal{D}_{\pm,j}^{\rm temp}[v,\theta]\}_{j=1,5}$, and $\mathcal{D}_{1}[v]$.}

We only focus on the estimates of $\mathcal{D}_{\pm,1}^{\rm vel,h}[v,\theta]$ as the others are treated similarly.
We see from the Plancherel theorem that 
\begin{align}
    &
    \n{\nablah^{\alphah}\partial_{x_3}^k\mathcal{D}_{\pm,1}^{\rm vel,h}[v,\theta](t)}_{L^2(\R^3)}
    \\
    &\quad\leq{}
    C
    \int_0^t
    \n{e^{(t-\tau)\Deltah}|\nablah|\nablah^{\alphah}\partial_{x_3}^k(\vh \otimes \vh)(\tau)}_{L^2(\R^3)}d\tau\\
    &\quad\leq{}
    C
    \int_0^{\frac{t}{2}}
    (t-\tau)^{-\frac{1}{2}-\frac{1+|\alphah|}{2}}
    \n{\partial_{x_3}^k(\vh \otimes \vh)(\tau)}_{L^2_{x_3}L^1_{\xh}(\R^3)}d\tau\\
    &\qquad
    +
    C
    \int_{\frac{t}{2}}^{t}
    (t-\tau)^{-\frac{1}{2}}
    \n{\nablah^{\alphah}\partial_{x_3}^k(\vh \otimes \vh)(\tau)}_{L^2(\R^3)}d\tau.
\end{align}
Here, we see that 
\begin{align}
    &\int_0^{\frac{t}{2}}
    (t-\tau)^{-\frac{1}{2}-\frac{1+|\alphah|}{2}}
    \n{\partial_{x_3}^k(\vh \otimes \vh)(\tau)}_{L^2_{x_3}L^1_{\xh}(\R^3)}d\tau
    \\
    &\quad
    \leq 
    C
    \int_0^{\frac{t}{2}}
    (t-\tau)^{-\frac{1}{2}-\frac{1+|\alphah|}{2}}
    \sum_{m=0}^k
    \n{\partial_{x_3}^m\vh(\tau)}_{L^2(\R^3)}
    \n{\vh(\tau)}_{L^{\infty}_{x_3}L^2_{\xh}(\R^3)}d\tau
    \\
    &\quad
    \leq 
    C
    \int_0^{\frac{t}{2}}
    (t-\tau)^{-\frac{1}{2}-\frac{1+|\alphah|}{2}}
    \sum_{m=0}^k
    \n{\partial_{x_3}^m\vh(\tau)}_{L^2(\R^3)}
    \n{\vh(\tau)}_{L^2(\R^3)}^{\frac{1}{2}}
    \n{\partial_{x_3}\vh(\tau)}_{L^2(\R^3)}^{\frac{1}{2}}d\tau
    \\
    &\quad
    \leq 
     C 
    \n{(v,\theta)}_{Y_{\varepsilon}(0,T)}^2
    \int_0^{\frac{t}{2}}
    (t-\tau)^{-\frac{1}{2}-\frac{1+|\alphah|}{2}}
    (1+\tau)^{-1}
    d\tau
    \\
    &\quad
    \leq 
  C   
    \n{(v,\theta)}_{Y_{\varepsilon}(0,T)}^2
    t^{-\frac{1}{2}-\frac{1+|\alphah|}{2}}
    \log (2+t) 
\end{align}
and 
\begin{align}
    &
    \int_{\frac{t}{2}}^t
    (t-\tau)^{-\frac{1}{2}}
    \n{\nablah^{\alphah}\partial_{x_3}^k(\vh \otimes \vh)(\tau)}_{L^2(\R^3)}d\tau \\
    &\quad
    \leq 
    C
    \int_{\frac{t}{2}}^t
    (t-\tau)^{-\frac{1}{2}}
    \sum_{m=0}^k
    \n{\nablah^{\alphah}\partial_{x_3}^m \vh(\tau)}_{L^2(\R^3)}\n{ \vh(\tau)}_{L^{\infty}(\R^3)}d\tau\\
    &\qquad
    +
    C
    \int_{\frac{t}{2}}^t
    (t-\tau)^{-\frac{1}{2}}
    \sum_{m=0}^k
    \n{\partial_{x_3}^m \vh(\tau)}_{L^2(\R^3)}\n{ \nablah^{\alphah} \vh(\tau)}_{L^{\infty}(\R^3)}d\tau \label{L^2est:D_{j=1,5}}\\
    &\quad
    \leq 
     C  
    \n{(v,\theta)}_{Y_{\varepsilon}(0,T)}^2
    \int_{\frac{t}{2}}^t
    (t-\tau)^{-\frac{1}{2}}
    (1+\tau)^{-\frac{3}{2}-\frac{|\alphah|}{2}}d\tau\\
    &\quad
    \leq 
     C 
    \n{(v,\theta)}_{Y_{\varepsilon}(0,T)}^2
    t^{-\frac{1}{2}-\frac{1+|\alphah|}{2}}.
\end{align}
Hence, we obtain 
\begin{align}
    \n{\nablah^{\alphah}\partial_{x_3}^k\mathcal{D}_{\pm,1}^{\rm vel,h}[v,\theta](t)}_{L^2(\R^3)}
    \leq{}
     C 
    \n{(v,\theta)}_{Y_{\varepsilon}(0,T)}^2
    t^{-\frac{1}{2}-\frac{1+|\alphah|}{2}}
    \log (2+t)
\end{align}
for $2 \leq  t < T$.

\noindent
{\it Step 2. Estimates of $\{\mathcal{D}_{\pm,j}^{\rm vel,h}[v,\theta]\}_{j=2,6}$, $\{\mathcal{D}_{\pm,j}^{\rm temp}[v,\theta]\}_{j=2,6}$, and $\mathcal{D}_{2}[v]$.}

We only focus on the estimates of $\mathcal{D}_{\pm,2}^{\rm vel,h}[v,\theta]$ as the others are treated similarly.
First, we have 
\begin{align}
    \sum_{m=0}^{5}
    \n{\partial_{x_3}^mv(\tau)}_{L^2(\R^3)}
    ={}&
    \sum_{m=0}^{4}
    \n{\partial_{x_3}^mv(\tau)}_{L^2(\R^3)}
    +
    \n{\partial_{x_3}^5v(\tau)}_{L^2(\R^3)}\\
    \leq{}&
     C
    \n{(v,\theta)}_{Y_{\varepsilon}(0,T)}
    (1+\tau)^{-\frac{1}{2}}
    +
    \n{\partial_{x_3}^4v(\tau)}_{L^2(\R^3)}^{\frac{2}{3}}
    \n{\partial_{x_3}^7v(\tau)}_{L^2(\R^3)}^{\frac{1}{3}}\\
    \leq{}&
      C   
    \n{(v,\theta)}_{Y_{\varepsilon}(0,T)}
    (1+\tau)^{-\frac{1}{3}}
\end{align}
and 
\begin{align}
    \sum_{m=0}^{5}
    \n{\partial_{x_3}^mv_3(\tau)}_{L^2(\R^3)}
    ={}&
    \n{v_3(\tau)}_{L^2(\R^3)}
    +
    \sum_{m=1}^{5}
    \n{\partial_{x_3}^mv_3(\tau)}_{L^2(\R^3)}\\
    \leq{}&
     C
    \n{(v,\theta)}_{Y_{\varepsilon}(0,T)}(1+\tau)^{-\frac{3}{4}}
    +
      C
    \n{(v,\theta)}_{Y_{\varepsilon}(0,T)}(1+\tau)^{-1}
    \\
    \leq{}&
    C 
    \n{(v,\theta)}_{Y_{\varepsilon}(0,T)}(1+\tau)^{-\frac{3}{4}},
\end{align}
which implies 
\begin{align}
    \n{\partial_{x_3}^{ { k+1}  } (v_{\ell}v_3)(\tau)}_{L^2_{x_3}L^1_{\xh}(\R^3)}
    \leq{}&
    C
    \sum_{m=0}^{k+1}
    \n{\partial_{x_3}^mv(\tau)}_{L^2(\R^3)}
    \n{v_3(\tau)}_{L^{\infty}_{x_3}L^2_{\xh}(\R^3)}\\
    &
    +
    C
    \n{v(\tau)}_{L^{\infty}_{x_3}L^2_{\xh}(\R^3)}
    \sum_{m=0}^{k+1}
    \n{\partial_{x_3}^mv_3(\tau)}_{L^2(\R^3)}\\
    \leq{}&
    C
    \sum_{m=0}^{k+1}
    \n{\partial_{x_3}^mv(\tau)}_{L^2(\R^3)}
    \n{v_3(\tau)}_{L^2(\R^3)}^{\frac{1}{2}}
    \n{\partial_{x_3}v_3(\tau)}_{L^2(\R^3)}^{\frac{1}{2}}
    \\
    &
    +
    C
    \n{v(\tau)}_{L^2(\R^3)}^{\frac{1}{2}}
    \n{\partial_{x_3}v(\tau)}_{L^2(\R^3)}^{\frac{1}{2}}
    \sum_{m=0}^{k+1}
    \n{\partial_{x_3}^mv_3(\tau)}_{L^2(\R^3)}\\
    \leq{}&
    C
    \n{(v,\theta)}_{Y_{\varepsilon}(0,T)}^2
    (1+\tau)^{-\frac{13}{12}}
\end{align}
and 
\begin{align}
    &
    \n{\partial_{x_3}^{k+1}(v_{\ell}v_3)(\tau)}_{L^2(\R^3)}
    \\
    &\quad\leq{}
    C
    \sum_{\ell=0}^{5}
    \n{\partial_{x_3}^{\ell}v(\tau)}_{L^2(\R^3)}
    \n{v_3(\tau)}_{L^{\infty}(\R^3)}
    +
    C
    \n{v(\tau)}_{L^{\infty}(\R^3)}
    \sum_{\ell=0}^{5}
    \n{\partial_{x_3}^{\ell}v_3(\tau)}_{L^2(\R^3)}\\
    &\quad
    \leq{}
     C
    \n{(v,\theta)}_{Y_{\varepsilon}(0,T)}^2(1+\tau)^{-\frac{11}{6}}
    +
    C
    \n{(v,\theta)}_{Y_{\varepsilon}(0,T)}^2(1+\tau)^{-\frac{7}{4}}\\
    &\quad\leq{}
      C
    \n{(v,\theta)}_{Y_{\varepsilon}(0,T)}^2(1+\tau)^{-\frac{7}{4}}.
\end{align} 
Thus, we have 
\begin{align}
    &
    \n{\nablah^{\alphah}\partial_{x_3}^k\mathcal{D}_{\pm,2}^{\rm vel,h}[v,\theta](t)}_{L^2(\R^3)}
    \\
    &\quad\leq{}
    C
    \sum_{\ell=1}^3
    \int_0^t
    \n{e^{(t-\tau)\Deltah}\nablah^{\alphah}\partial_{x_3}^{k+1}(v_{\ell}v_3)(\tau)}_{L^2(\R^3)}d\tau\\
    &\quad\leq{}
    C
    \sum_{\ell=1}^3
    \int_0^{\frac{t}{2}}
    (t-\tau)^{-\frac{1}{2}-\frac{|\alphah|}{2}}
    \n{
    \partial_{x_3}^{k+1} 
    (v_{\ell}v_3)(\tau)}_{L^2_{x_3}L^1_{\xh}(\R^3)}d\tau\\
    &\qquad
    +
    C
    \sum_{\ell=1}^3
    \int_{\frac{t}{2}}^{t}
    (t-\tau)^{-\frac{|\alphah|}{2}}
    \n{\partial_{x_3}^{k+1}(v_{\ell}v_3)(\tau)}_{L^2(\R^3)}d\tau\\
    &\quad\leq{}
    C
    \n{(v,\theta)}_{Y_{\varepsilon}(0,T)}^2
    \int_0^{\frac{t}{2}}
    (t-\tau)^{-\frac{1}{2}-\frac{|\alphah|}{2}}
    (1+\tau)^{-\frac{13}{12}}d\tau\\
    &\qquad
    +
     C
    \n{(v,\theta)}_{Y_{\varepsilon}(0,T)}^2
    \int_{\frac{t}{2}}^{t}
    (t-\tau)^{-\frac{|\alphah|}{2}}
    (1+\tau)^{-\frac{7}{4}}d\tau\\
    &\quad\leq
     C
    \n{(v,\theta)}_{Y_{\varepsilon}(0,T)}^2
    t^{-\frac{1}{2}-\frac{|\alphah|}{2}}.
\end{align}
Hence, we see that
\begin{align}
    &
    \n{\nablah^{\alphah}\partial_{x_3}^k\mathcal{D}_{\pm,2}^{\rm vel,h}[v,\theta](t)}_{L^2(\R^3)}
    \leq
      C
    \n{(v,\theta)}_{Y_{\varepsilon}(0,T)}^2
    t^{-\frac{1}{2}-\frac{|\alphah|}{2}}
\end{align}
for $2 \leq  t < T$.

\noindent
{\it Step 3. Estimates of $\{\mathcal{D}_{\pm,j}^{\rm vel,h}[v,\theta]\}_{j=3,4}$, $\{\mathcal{D}_{\pm,j}^{\rm temp}[v,\theta]\}_{j=3,4}$, and $\{\mathcal{D}_{\pm,j}^{\rm vel,v}[v,\theta]\}_{j=2,3,4,6}$.}

We only focus on the estimates of $\mathcal{D}_{\pm,3}^{\rm vel,h}[v,\theta]$ as the others are treated similarly.
We see that
\begin{align}
    &
    \n{\nablah^{\alphah}\partial_{x_3}^k\mathcal{D}_{\pm,3}^{\rm vel,h}[v,\theta](t)}_{L^2(\R^3)}
    \\
    &\quad\leq{}
    C
    \sum_{\ell=1}^3
    \int_0^t
    \n{e^{(t-\tau)\Deltah}|\nablah|\nablah^{\alphah}\partial_{x_3}^k(v_{\ell}v_3)(\tau)}_{L^2(\R^3)}d\tau\\
    &\quad\leq{}
    C
    \sum_{\ell=1}^3
    \int_0^{\frac{t}{2}}
    (t-\tau)^{-\frac{1}{2}-\frac{1+|\alphah|}{2}}
    \n{\partial_{x_3}^k(v_{\ell}v_3)(\tau)}_{L^2_{x_3}L^1_{\xh}(\R^3)}d\tau\\
    &\qquad
    +
    C
    \sum_{\ell=1}^3
    \int_{\frac{t}{2}}^{t}
    (t-\tau)^{-\frac{1}{2}}
    \n{\nablah^{\alphah}\partial_{x_3}^k(v_{\ell}v_3)(\tau)}_{L^2(\R^3)}d\tau.
\end{align}
Here, we see that 
\begin{align}
    &
    \sum_{\ell=1}^3
    \int_0^{\frac{t}{2}}
    (t-\tau)^{-\frac{1}{2}-\frac{1+|\alphah|}{2}}
    \n{\partial_{x_3}^k(v_{\ell}v_3)(\tau)}_{L^2_{x_3}L^1_{\xh}(\R^3)}d\tau\\
    &\quad
    \leq
    C
    \int_0^{\frac{t}{2}}
    (t-\tau)^{-\frac{1}{2}-\frac{1+|\alphah|}{2}}
    \sum_{m=0}^k
    \n{\partial_{x_3}^mv(\tau)}_{L^2(\R^3)}
    \n{v_3(\tau)}_{L^{\infty}_{x_3}L^2_{\xh}(\R^3)}d\tau\\
    &\qquad
    +
    C
    \int_0^{\frac{t}{2}}
    (t-\tau)^{-\frac{1}{2}-\frac{1+|\alphah|}{2}}
    \n{v(\tau)}_{L^{\infty}_{x_3}L^2_{\xh}(\R^3)}
    \sum_{m=0}^k
    \n{\partial_{x_3}^mv_3(\tau)}_{L^2(\R^3)}d\tau\\
    &\quad
    \leq
    C
    \int_0^{\frac{t}{2}}
    (t-\tau)^{-\frac{1}{2}-\frac{1+|\alphah|}{2}}
    \sum_{m=0}^k
    \n{\partial_{x_3}^mv(\tau)}_{L^2(\R^3)}
    \n{v_3(\tau)}_{L^2(\R^3)}^{\frac{1}{2}}
    \n{\partial_{x_3}v_3(\tau)}_{L^2(\R^3)}^{\frac{1}{2}}
    d\tau\\
    &\qquad
    +
    C
    \int_0^{\frac{t}{2}}
    (t-\tau)^{-\frac{1}{2}-\frac{1+|\alphah|}{2}}
    \n{v(\tau)}_{L^2(\R^3)}^{\frac{1}{2}}
    \n{\partial_{x_3}v(\tau)}_{L^2(\R^3)}^{\frac{1}{2}}
    \sum_{m=0}^k
    \n{\partial_{x_3}^mv_3(\tau)}_{L^2(\R^3)}d\tau\\
    &\quad
    \leq
     C  
    \n{(v,\theta)}_{Y_{\varepsilon}(0,T)}^2
    \int_0^{\frac{t}{2}}
    (t-\tau)^{-\frac{1}{2}-\frac{1+|\alphah|}{2}}
    (1+\tau)^{-\frac{5}{4}}
    d\tau \\
    &\quad
    \leq
      C
    \n{(v,\theta)}_{Y_{\varepsilon}(0,T)}^2
    t^{-\frac{1}{2}-\frac{1+|\alphah|}{2}}
\end{align}
and
\begin{align}
    &
    \sum_{\ell=1}^3
    \int_{\frac{t}{2}}^{t}
    (t-\tau)^{-\frac{1}{2}}
    \n{\nablah^{\alphah}\partial_{x_3}^k(v_{\ell}v_3)(\tau)}_{L^2(\R^3)}d\tau\\
    &\quad
    \leq 
    C
    \int_{\frac{t}{2}}^t
    (t-\tau)^{-\frac{1}{2}}
    \sum_{m=0}^k
    \n{\nablah^{\alphah}\partial_{x_3}^m v(\tau)}_{L^2(\R^3)}
    \n{ v_3(\tau) }_{L^{\infty}(\R^3)}d\tau\\
    &\qquad
    +
    C
    \int_{\frac{t}{2}}^t
    (t-\tau)^{-\frac{1}{2}}
    \sum_{m=0}^k
    \n{\partial_{x_3}^m v(\tau)}_{L^2(\R^3)}
    \n{ \nablah^{\alphah}v_3(\tau) }_{L^{\infty}(\R^3)}d\tau\\
    &\qquad
    +
    C
    \int_{\frac{t}{2}}^t
    (t-\tau)^{-\frac{1}{2}}
    \n{ v(\tau)}_{L^{\infty}(\R^3)}
    \sum_{m=0}^k
    \n{ {\nablah^{\alphah} } \partial_{x_3}^mv_3(\tau) }_{L^{2}(\R^3)}d\tau\\
    &\qquad
    +
    C
    \int_{\frac{t}{2}}^t
    (t-\tau)^{-\frac{1}{2}}
    \n{ \nablah^{\alphah} v(\tau)}_{L^{\infty}(\R^3)}
    \sum_{m=0}^k
    \n{\partial_{x_3}^m v_3(\tau)}_{L^2(\R^3)}d\tau\\
    &\quad
    \leq 
     C
    \n{(v,\theta)}_{Y_{\varepsilon}(0,T)}^2
    \int_{\frac{t}{2}}^t
    (t-\tau)^{-\frac{1}{2}}
    (1+\tau)^{-\frac{7}{4}-\frac{|\alphah|}{2}}d\tau\\
    &\quad
    \leq 
      C
    \n{(v,\theta)}_{Y_{\varepsilon}(0,T)}^2
    t^{-\frac{3}{4}-\frac{1+|\alphah|}{2}}.
\end{align}
Thus, we have 
\begin{align}
    &
    \n{\nablah^{\alphah}\partial_{x_3}^k\mathcal{D}_{\pm,3}^{\rm vel,h}[v,\theta](t)}_{L^2(\R^3)}
    \leq 
     C
    \n{(v,\theta)}_{Y_{\varepsilon}(0,T)}^2
    t^{-\frac{1}{2}-\frac{1+|\alphah|}{2}}
\end{align}
for $2 \leq  t < T$.
Thus, we complete the proof.
\end{proof}
{
Next, we consider the decay estimates of $\mathcal{D}_1[v]$ and $\mathcal{D}_2[v]$ in $L^p$ with $2\leq p \leq \infty$. }
\begin{lemm}\label{lemm:decay-Duh}
Let $(v,\theta) \in Y_{\varepsilon}(0,T)$ with some $0<\varepsilon<1/4$ and $2<T\leq \infty$. Then, there exists an absolute positive constant $C$ such that for any $2 \leq p \leq \infty$ and $\alpha=(\alphah,\alpha_3)\in (\mathbb{N} \cup \{ 0\} )^2\times (\mathbb{N} \cup \{ 0\} )$ with $|\alpha| \leq 1$, 
it holds 
\begin{align}
    & \n{ \Grad^{\alpha}  \mathcal{D}_1[v](t)}_{L^p(\R^3)} 
    \leq 
    C 
    \n{(v,\theta)}_{Y_{\varepsilon}(0,T)}^2 
    t^{ -(1-\frac{1}{p})-\f{1+|\alphah|}{2} } \log t,  \label{Duh-D1}    \\
    & \n{ \Grad^{\alpha} \mathcal{D}_2[v](t)}_{L^p(\R^3)} 
    \leq 
    C 
    \n{(v,\theta)}_{Y_{\varepsilon}(0,T)}^2
    t^{ -(1-\frac{1}{p})-\f{|\alphah|}{2} } \label{Duh-D2}  
\end{align} 
for $2 \leq t < T$.
\end{lemm}
\begin{proof}
We first consider \eqref{Duh-D1}. 
Based on Lemma \ref{lemm:L^2} and the interpolation inequality, it suffices to consider the $L^{\infty}$-estimate. Observe that
\begin{align}
  \n{ \nabla^{\alpha}
 \mathcal{D}_1[v](t)
 }_{ L^{\infty}(\R^3) } 
 & 
 \leq 
 \int_0^{\f{t}{2}}
 \n{
 \nablah^{\alphah} \nablah e^{(t-\tau)\Deltah}
 \p_{x_3}^{\alpha_3} \Ph(\vh \otimes \vh)(\tau)
 }_{ L^{\infty}( \R^3 ) } 
 d\tau \\
 & \quad 
 +
 \int_{\f{t}{2}}^{t}
 \n{
    \nablah^{\alphah} 
    e^{(t-\tau)\Deltah}
    \nablah
    \p_{x_3}^{\alpha_3}
    \Ph(\vh \otimes \vh)(\tau) 
 }_{ L^{\infty}(\R^3) } 
 d\tau \\
 &
 =:
 I_{1}(t)+ I_{2}(t). 
\end{align}  
For $I_{1}(t)$, we see  
\begin{align}
I_{1}(t) &  \leq 
C
\int_0^{\f{t}{2}}
\sum_{j\in \mathbb{Z}} 
e^{ -c 2^{2j} (t-\tau) }
2^{ j(1+|\alphah|)  }
\sum_{k \in \mathbb{Z}} 
\n{
 \Deltahh_j \Deltav_k  \p_{x_3}^{\alpha_3} \Ph(\vh \otimes \vh)(\tau)  
}_{ L^{\infty}(\R^3) } 
d\tau \\
& \leq 
C
\int_0^{\f{t}{2}}
\sum_{j\in \mathbb{Z}} 
e^{ -c 2^{2j} (t-\tau) }
2^{ j(1+|\alphah|)  } 2^{2j}
\sum_{k \in \mathbb{Z}} 
2^{k}
\n{
 \Deltahh_j \Deltav_k  \p_{x_3}^{\alpha_3} \Ph(\vh \otimes \vh)(\tau)  
}_{ L^{1}(\R^3) }  
d\tau \\
& \leq 
C
\int_0^{\f{t}{2}} 
(t-\tau)^{ -1-\f{1+|\alphah|}{2}  }
 \n{
 (\vh \otimes \vh)(\tau)
 }_{L^1_{\xh}W^{3,1}_{x_3}(\R^3)}
d\tau \\
& \leq 
C
\int_0^{\f{t}{2}} 
(t-\tau)^{ -1-\f{1+|\alphah|}{2}  }
\sum_{\ell=0}^{3}
\sum_{m=0}^{\ell}
\n{
\p_{x_3}^{m} \vh(\tau)
}_{ L^{2}(\R^3) }  
\n{
\p_{x_3}^{\ell-m} \vh(\tau)
}_{ L^{2}(\R^3) }  
d\tau  \\
&
\leq 
 C
    \n{(v,\theta)}_{Y_{\varepsilon}(0,T)}^2
\int_0^{\f{t}{2}} 
(t-\tau)^{ -1-\f{1+|\alphah|}{2}  }
(1+\tau)^{-1} 
d\tau \\
& 
\leq 
 C
    \n{(v,\theta)}_{Y_{\varepsilon}(0,T)}^2
    t^{ -1-\f{1+|\alphah|}{2}  }
    \log t. 
\end{align}  
For $I_{2}(t)$, we see  
\begin{align}
I_{2}(t) &  \leq 
C
\int_{\f{t}{2}}^{t} 
\sum_{j\in \mathbb{Z}} 
e^{ -c 2^{2j} (t-\tau) }
2^{ j  |\alphah| }
\sum_{k \in \mathbb{Z}} 
\n{
 \Deltahh_j \Deltav_k  \p_{x_3}^{\alpha_3}\nablah \Ph(\vh \otimes \vh)(\tau)  
}_{ L^{\infty}(\R^3) } 
d\tau  \\
& 
\leq 
 C
 \int_{\f{t}{2}}^{t} 
\sum_{j\in \mathbb{Z}} 
e^{ -c 2^{2j} (t-\tau) }
2^{ j  |\alphah| } 2^{\f{j}{2}}
\sum_{k \in \mathbb{Z}} 
2^{\f{k}{4}}
\n{
 \Deltahh_j \Deltav_k  \p_{x_3}^{\alpha_3}\nablah \Ph(\vh \otimes \vh)(\tau)  
}_{ L^{4}(\R^3) } 
d\tau \\
& 
\leq 
 C
 \int_{\f{t}{2}}^{t} 
 (t-\tau)^{ -\f{1}{4}-\f{|\alphah|}{2}    }
 \sup_{j\in \mathbb{Z}} 
 \sum_{k \in \mathbb{Z}} 
 2^{\f{k}{4}}
\n{
 \Deltahh_j \Deltav_k  \p_{x_3}^{\alpha_3}\nablah \Ph(\vh \otimes \vh)(\tau)  
}_{ L^{4}(\R^3) }  \\ 
&
\leq 
 C
 \int_{\f{t}{2}}^{t} 
 (t-\tau)^{ -\f{1}{4}-\f{|\alphah|}{2}    }
  \sup_{j,k \in \mathbb{Z}} 
  \n{
 \Deltahh_j \Deltav_k  \p_{x_3}^{\alpha_3}\nablah \Ph(\vh \otimes \vh)(\tau)  
}_{ L^{4}(\R^3) } ^{\f{3}{4}} \\
&
\qquad \qquad 
\times
\sup_{j,k \in \mathbb{Z}} 
  \n{
 \Deltahh_j \Deltav_k  \p_{x_3}^{\alpha_3+1}\nablah \Ph(\vh \otimes \vh)(\tau)  
}_{ L^{4}(\R^3) } ^{\f{1}{4}} 
d\tau \\
&
\leq 
C
    \n{(v,\theta)}_{Y_{\varepsilon}(0,T)}^2
 \int_{\f{t}{2}}^{t} 
 (t-\tau)^{ -\f{1}{4}-\f{|\alphah|}{2}    }
 (1+\tau)^{-\f{9}{4}}
 d\tau\\
 &
\leq 
C
    \n{(v,\theta)}_{Y_{\varepsilon}(0,T)}^2
 t^{ -1-\f{1+|\alphah|}{2}  }.
\end{align}  
Combining the two estimates above, we obtain
\begin{align}
  \n{ \nabla^{\alpha}
 \mathcal{D}_1[v](t)
 }_{ L^{\infty}(\R^3) } 
 \leq 
  C 
    \n{(v,\theta)}_{Y_{\varepsilon}(0,T)}^2 
    t^{ -1-\f{1+|\alphah|}{2} } \log t
\end{align}  
for all $\alpha=(\alphah,\alpha_3)$ and $|\alpha|\leq 1$.

Next, we handle \eqref{Duh-D2}. As before, it suffices to estimate in $L^{\infty}$. It follows that 
\begin{align}
  \n{ \nabla^{\alpha}
 \mathcal{D}_2[v](t)
 }_{ L^{\infty}(\R^3) } 
 & 
 \leq 
 \int_0^{\f{t}{2}}
 \n{
 \nablah^{\alphah}  e^{(t-\tau)\Deltah}
 \p_{x_3}^{\alpha_3} \p_{x_3} \Ph(v_3 \vh)(\tau) 
 }_{ L^{\infty}(\R^3) } 
 d\tau \\
 & \quad 
 +
 \int_{\f{t}{2}}^{t}
 \n{
   \nablah^{\alphah} 
   e^{(t-\tau)\Deltah}
    \p_{x_3}^{\alpha_3} 
    \p_{x_3}
    \Ph(v_3  \vh)(\tau)  
 }_{ L^{\infty}(\R^3) } 
 d\tau \\
 &
 =:
 J_{1}(t)+ J_{2}(t). 
\end{align}  
For $J_1(t)$, we estimate through
\begin{align}
    J_{1}(t) 
    &  \leq 
    C
    \int_0^{\f{t}{2}}
    \sum_{j\in \mathbb{Z}} 
    e^{ -c 2^{2j} (t-\tau) }
    2^{ j |\alphah|  }
    \sum_{k \in \mathbb{Z}} 
    \n{
    \Deltahh_j \Deltav_k 
    \p_{x_3}^{\alpha_3} \p_{x_3} \Ph(v_3  \vh)(\tau)  
    }_{ L^{\infty}(\R^3) } 
    d\tau \\
    & \leq 
    C
    \int_0^{\f{t}{2}}
    \sum_{j\in \mathbb{Z}} 
    e^{ -c 2^{2j} (t-\tau) }
    2^{ j |\alphah|  } 2^{2j}
    \sum_{k \in \mathbb{Z}} 
    2^{k}
    \n{
    \Deltahh_j \Deltav_k  \p_{x_3}^{\alpha_3} \p_{x_3} \Ph(v_3  \vh)(\tau)   
    }_{ L^{1}(\R^3) }  
    d\tau \\
    & \leq 
    C
    \int_0^{\f{t}{2}} 
    (t-\tau)^{ -1-\f{|\alphah|}{2}  }
    \n{
    (v_3  \vh)(\tau) 
    }_{L^1_{\xh}W^{4,1}_{x_3}(\R^3)}
    d\tau \\
    & \leq 
    C
    \int_0^{\f{t}{2}} 
    (t-\tau)^{ -1-\f{|\alphah|}{2}  }
    \sum_{\ell=0}^{4}
    \sum_{m=0}^{\ell}
    \n{
    \p_{x_3}^{m} v_3 (\tau)
    }_{ L^{2}(\R^3) }  
    \n{
    \p_{x_3}^{\ell-m} \vh(\tau)
    }_{ L^{2}(\R^3) }  
    d\tau  \\
    &
    \leq 
    C
    \n{(v,\theta)}_{Y_{\varepsilon}(0,T)}^2
    \int_0^{\f{t}{2}} 
    (t-\tau)^{ -1-\f{|\alphah|}{2}  }
    (1+\tau)^{   - \f{5}{4}  }   
    d\tau \\
    & 
    \leq 
    C
    \n{(v,\theta)}_{Y_{\varepsilon}(0,T)}^2
    t^{ -1-\f{|\alphah|}{2}  }. 
\end{align}  
To treat $J_2(t)$, we notice that
\begin{align}
    \n{\p_{x_3}^{2} v_3(\tau) }
    _{ L^{\infty}(\R^3)     }
    &
    \leq 
    C
    \n{\p_{x_3} v_3(\tau) }
    _{ L^{\infty}(\R^3)     }^{\f{3}{4}}
    \n{\p_{x_3}^{5} v_3(\tau) }^{\f{1}{4}}
    _{ L^{\infty}(\R^3)     } 
    \leq 
    C
    \n{(v,\theta)}_{Y_{\varepsilon}(0,T)}
    (1+\tau)^{-\f{3}{2}}, \\
     \n{\p_{x_3}^{3} v_3(\tau) }
    _{ L^{\infty}(\R^3)     }
    &
    \leq 
    C
    \n{\p_{x_3} v_3(\tau) }
    _{ L^{\infty}(\R^3)     }^{\f{1}{2}}
    \n{\p_{x_3}^{5} v_3(\tau) }^{\f{1}{2}}
    _{ L^{\infty}(\R^3)     } 
    \leq 
    C
    \n{(v,\theta)}_{Y_{\varepsilon}(0,T)}
    (1+\tau)^{-1},  \\
    \n{
    \p_{x_3}^{3}
    (v_3\vh)(\tau)  
    }_{ L^{\infty}(\R^3)     }
    &
    \leq 
    C \sum_{\ell=0}^{3} \sum_{m=0}^{\ell}
    \n{ \p_{x_3}^{m}v_3(\tau)   }_{ L^{\infty}(\R^3)     }
    \n{ \p_{x_3}^{\ell-m}v_3(\tau)   }_{ L^{\infty}(\R^3)     }\\
    & \leq
    C
    \n{(v,\theta)}_{Y_{\varepsilon}(0,T)}^2
    (1+\tau)^{ -2  }, \\
    \n{
    \p_{x_3}^{3}
    (v_3\vh)(\tau)  
    }_{ L^{4}(\R^3)     }
    &
    \leq 
    C
    \n{
    \p_{x_3}^{3}
    (v_3\vh)(\tau)  
    }_{ L^{2}(\R^3)     }^{\f{1}{2}}
    \n{
    \p_{x_3}^{3}
    (v_3\vh)(\tau)  
    }_{ L^{\infty}(\R^3)     }^{\f{1}{2}} \\
    &
    \leq 
    C
    \n{(v,\theta)}_{Y_{\varepsilon}(0,T)}^2
    (1+\tau)^{-\f{15}{8}}.
\end{align} 
Thus,
\begin{align}
J_{2}(t) &  \leq 
C
\int_{\f{t}{2}}^{t} 
\sum_{j\in \mathbb{Z}} 
e^{ -c 2^{2j} (t-\tau) }
2^{ j  |\alphah| }
\sum_{k \in \mathbb{Z}} 
\n{
 \Deltahh_j \Deltav_k  \p_{x_3}^{\alpha_3}
 \p_{x_3}
 \Ph(v_3  \vh)(\tau)   
}_{ L^{\infty}(\R^3) } 
d\tau  \\
& 
\leq 
 C
 \int_{\f{t}{2}}^{t} 
\sum_{j\in \mathbb{Z}} 
e^{ -c 2^{2j} (t-\tau) }
2^{ j  |\alphah| } 2^{\f{j}{2}}
\sum_{k \in \mathbb{Z}} 
2^{\f{k}{4}}
\n{
 \Deltahh_j \Deltav_k  \p_{x_3}^{\alpha_3} \p_{x_3}\Ph(v_3  \vh)(\tau)    
}_{ L^{4}(\R^3) } 
d\tau \\
& 
\leq 
 C
 \int_{\f{t}{2}}^{t} 
 (t-\tau)^{ -\f{1}{4}-\f{|\alphah|}{2}    }
 \sup_{j\in \mathbb{Z}} 
 \sum_{k \in \mathbb{Z}} 
 2^{\f{k}{4}}
\n{
 \Deltahh_j \Deltav_k  \p_{x_3}^{\alpha_3} \p_{x_3}\Ph(v_3 \vh)(\tau)   
}_{ L^{4}(\R^3) }  \\  
&
\leq 
 C
 \int_{\f{t}{2}}^{t} 
 (t-\tau)^{ -\f{1}{4}-\f{|\alphah|}{2}    }
  \sup_{j,k \in \mathbb{Z}} 
  \n{
 \Deltahh_j \Deltav_k  \p_{x_3}^{\alpha_3} \p_{x_3} \Ph(v_3  \vh)(\tau)  
}_{ L^{4}(\R^3) } ^{\f{3}{4}} \\
&
\qquad \qquad 
\times
\sup_{j,k \in \mathbb{Z}} 
  \n{
 \Deltahh_j \Deltav_k  \p_{x_3}^{\alpha_3+1} \p_{x_3} \Ph(v_3  \vh)(\tau)   
}_{ L^{4}(\R^3) } ^{\f{1}{4}} 
d\tau \\
& 
\leq 
C
    \n{(v,\theta)}_{Y_{\varepsilon}(0,T)}^2
 \int_{\f{t}{2}}^{t} 
 (t-\tau)^{ -\f{1}{4}-\f{|\alphah|}{2}    }
 (1+\tau)^{-\f{15}{8}} 
 d\tau\\
 &
\leq 
C
    \n{(v,\theta)}_{Y_{\varepsilon}(0,T)}^2
 t^{ -1-\f{|\alphah|}{2}  }.
\end{align}  
Hence, we obtain  
\begin{align}
  \n{ \nabla^{\alpha}
 \mathcal{D}_2[v](t)
 }_{ L^{\infty}(\R^3) } 
 \leq 
  C
    \n{(v,\theta)}_{Y_{\varepsilon}(0,T)}^2 
    t^{ -1-\f{|\alphah|}{2} }
\end{align}  
for all $\alpha=(\alphah,\alpha_3)$ and $|\alpha|\leq 1$. 
This completes the proof of Lemma \ref{lemm:decay-Duh}. 
\end{proof}

Finally, we consider the remaining Duhamel estimates in $L^p$ with $2\leq p \leq \infty$. 
\begin{lemm}\label{lemm:vel_h-temp}
Let $(v,\theta) \in Y_{\varepsilon}(0,T)$ with some $0<\varepsilon<1/4$ and $2<T\leq \infty$. Then, there exists a positive constant $C_{\ep}$ such that for any $2 \leq p \leq \infty$ and $\alpha=(\alphah,\alpha_3)\in (\mathbb{N} \cup \{ 0\} )^2\times (\mathbb{N} \cup \{ 0\} )$ with $|\alpha| \leq 1$, 
it holds 
\begin{align}
    &
    \begin{aligned}
    &
    \sum_{j=1,5}
    \sp{\n{\nabla^{\alpha}\mathcal{D}_{\pm,j}^{\rm vel,h}[v,\theta](t)}_{L^p(\R^3)}+\n{\nabla^{\alpha}\mathcal{D}_{\pm,j}^{\rm temp}[v,\theta](t)}_{L^p(\R^3)} 
    } \\
    &\quad
    \leq{}
    C
    \n{(v,\theta)}_{Y_{\varepsilon}(0,T)}^2
    t^{
    -(1-\frac{1}{p})
    - \frac{1+|\alphah|}{2}  
    } 
    \max
    \Mp{ t^{-\frac{3}{4}(1-\frac{2}{p})}\log t,1},
    \end{aligned}\\
    &
    \begin{aligned}
    &\sum_{j=2,6}
    \sp{\n{\nabla^{\alpha}\mathcal{D}_{\pm,j}^{\rm vel,h}[v,\theta](t)}_{L^p(\R^3)}+\n{\nabla^{\alpha}\mathcal{D}_{\pm,j}^{\rm temp}[v,\theta](t)}_{L^p(\R^3)}} \\ 
    &\quad
    \leq 
    C
    \n{(v,\theta)}_{Y_{\varepsilon}(0,T)}^2
  {  t^{-(1-\frac{1}{p})-{\frac{|\alphah|}{2}}-{A_0(p)}}   
       }   
     ,
    \end{aligned} \\
    &
    \begin{aligned}
    &
    \sum_{j=3,4}
    \sp{\n{\nabla^{\alpha}\mathcal{D}_{\pm,j}^{\rm vel,h}[v,\theta](t)}_{L^p(\R^3)}+\n{\nabla^{\alpha}\mathcal{D}_{\pm,j}^{\rm temp}[v,\theta](t)}_{L^p(\R^3)}}\\
    &\quad
    \leq 
    C
    \n{(v,\theta)}_{Y_{\varepsilon}(0,T)}^2 
  t^{
    -(1-\frac{1}{p})
    - \frac{1+|\alphah|}{2}  
   { -A_0(p)  } 
    } 
    \end{aligned} \\
    &
    \begin{aligned} 
   & \sum_{j=1,5}
    \n{\nabla^{\alpha}\mathcal{D}_{\pm,j}^{\rm vel,v}[v,\theta](t)}_{L^p(\R^3)}  \\  
    &\quad 
    \leq 
     C
    \n{(v,\theta)}_{Y_{\varepsilon}(0,T)}^2
    t^{
    -(1-\frac{1}{p})
    - \frac{1+|\alpha|}{2}  
    } 
 {   \max
    \Mp{ t^{-\frac{3}{4}(1-\frac{2}{p})}\log t,1} }
    \end{aligned}   \\
   &  \begin{aligned} 
   & \sum_{j=2,3,4,6}
    \n{\nabla^{\alpha}\mathcal{D}_{\pm,j}^{\rm vel,v}[v,\theta](t)}_{L^p(\R^3)}  \\  
    &\quad
    \leq 
     C
    \n{(v,\theta)}_{Y_{\varepsilon}(0,T)}^2
    t^{
    -(1-\frac{1}{p})
    - \frac{1+|\alpha|}{2}  
 {   -A_0(p)  }
     } 
    \end{aligned}  
\end{align}
for $2 \leq  t < T$.
\end{lemm}

\begin{proof} 

\noindent
{\it Step 1. Estimates of $\{\mathcal{D}_{\pm,j}^{\rm vel,h}[v,\theta]\}_{j=1,5}$, $\{\mathcal{D}_{\pm,j}^{\rm temp}[v,\theta]\}_{j=1,5}$.}

We only focus on the estimates of $\mathcal{D}_{\pm,1}^{\rm vel,h}[v,\theta]$ as the others are treated similarly.
It follows from Theorem \ref{thm:lin} that  
\begin{align}
    &
    \n{  \nabla^{\alpha}  
    \mathcal{D}_{\pm,1}^{\rm vel,h}[v,\theta](t)}_{L^p(\R^3)}  \\ 
    & \quad 
    \leq
    C
    \int_0^{\frac{t}{2}}
    (t-\tau)^{-(1-\frac{1}{p})-\frac{3}{4}(1-\frac{2}{p})-\frac{1+|\alphah|}{2}}
    \n{(\vh \otimes \vh)(\tau)}_{L^1_{\xh} W^{3+\alpha_3,1}_{x_3}(\R^3) } 
    d\tau \\  
    &
    \qquad  
    +
    C
    \n{
    \int_{\frac{t}{2}}^t
    e^{(t-\tau)\Deltah}
    e^{\pm i (t-\tau)\frac{|\nablah|}{|\nabla|}}
    \Grad^{\alpha}   
    \frac{\nablah}{|\nablah|}
    \frac{\partial_{x_3}^2}{|\nabla|^2}
    \sp{\frac{\nablah}{|\nablah|} \cdot \sp{\nablah \cdot (\vh \otimes \vh)(\tau)}}
    d\tau
    }_{L^p(\R^3)}\\
    &\quad
    =: I_{1,p}(t) + I_{2,p}(t).
\end{align}
For the estimate of $I_{1,p}(t)$, it holds by Corollary \ref{cor:disp} that
\begin{align}
    I_{1,p}(t)
    \leq{}&
    C
    t^{-(1-\frac{1}{p})-\frac{3}{4}(1-\frac{2}{p})-\frac{1+|\alphah|}{2}}
    \int_0^{\frac{t}{2}}  
    \sum_{k=0}^{3+\alpha_3} \sum_{\ell=0}^{k}
    \n{\partial_{x_3}^{k-\ell}\vh(\tau)}_{L^2(\R^3)}
    \n{\partial_{x_3}^{\ell}\vh(\tau)}_{L^2(\R^3)} 
    d\tau \\
    \leq{}&
       C
    \n{(v,\theta)}_{Y_{\varepsilon}(0,T)}^2
    t^{-(1-\frac{1}{p})-\frac{3}{4}(1-\frac{2}{p})-\frac{1+|\alphah|}{2}}
    \int_0^{\frac{t}{2}}
    (1+\tau)^{-1}d\tau\\
    \leq{}&
    C
    \n{(v,\theta)}_{Y_{\varepsilon}(0,T)}^2
    t^{-(1-\frac{1}{p})-\frac{3}{4}(1-\frac{2}{p})-\frac{1+|\alphah|}{2}}
    \log t.
\end{align}
Next, we consider the estimate of $I_{2,p}(t)$, we have for the case of $p=2$ that
\begin{align}
    I_{2,2}(t)
    \leq{}&
    C
    \int_{\frac{t}{2}}^t
    (t-\tau)^{-\frac{1}{2}}
    \n{ 
    \nabla^{\alpha}
    (\vh \otimes \vh)(\tau)}_{L^2(\R^3)}d\tau\\
    \leq{}&
    C
    \int_{\frac{t}{2}}^t
    (t-\tau)^{-\frac{1}{2}}
    \n{\nabla^{\alpha}\vh(\tau)}_{L^2(\R^3)}  \n{\vh(\tau)}_{L^{\infty}(\R^3)}d\tau
    \label{Duh-I_2}    \\
    \leq{}&
    C
    \n{(v,\theta)}_{Y_{\varepsilon}(0,T)}^2
    \int_{\frac{t}{2}}^t
    (t-\tau)^{-\frac{1}{2}}(1+\tau)^{-\frac{3}{2}-\frac{|\alphah|}{2}}d\tau\\
    \leq{}&
    C
    \n{(v,\theta)}_{Y_{\varepsilon}(0,T)}^2
    t^{-\frac{1}{2}-\frac{1+|\alphah|}{2}}.
\end{align}
For the case of $p=\infty$, {we deduce with the help of Lemma \ref{lemm:para} that} 
\begin{align}
    I_{2,\infty}(t)
    \leq{}&
    C
    \sum_{j,k\in\mathbb{Z}}
    { 2^j}
    2^{\frac{1}{2}k}
    \int_{\frac{t}{2}}^t
    \n{\Deltahh_j\Deltav_k e^{(t-\tau)\Deltah}
    {   \nabla^{\alpha} } 
     \nablah 
    \cdot(\vh \otimes \vh)(\tau)
    }_{L^2(\R^3)}d\tau\\
    \leq{}&
    C
    \sum_{j,k\in\mathbb{Z}}
    { 2^{(1+|\alphah|)j}}
    2^{\frac{1}{2}k}
    \int_{\frac{t}{2}}^t
    e^{-c2^{2j}(t-\tau)}
    \n{\Deltahh_j\Deltav_k 
    {   
    \p_{x_3}^{\alpha_3}
    {\nablah \cdot(\vh \otimes \vh)(\tau)}}}_{L^2(\R^3)}d\tau\\
    \leq{}&
    \begin{dcases}
        C t^{  { \frac{1}{2} } }
        \sup_{\frac{t}{2}\leq \tau \leq t}
        \sup_{j \in \mathbb{Z}}
        \sum_{k \in \mathbb{Z}}
        2^{\frac{1}{2}k}
        \n{\Deltahh_j\Deltav_k 
        {   
        \p_{x_3}^{\alpha_3}
        {\nablah \cdot(\vh \otimes \vh)(\tau)}}}_{L^2(\R^3)} 
        & (|\alphah| = 0)\\
        C
        \sup_{\frac{t}{2}\leq \tau \leq t}
        \sum_{j,k\in \mathbb{Z}}
        2^{\frac{1}{2}k}
        \n{\Deltahh_j\Deltav_k 
        {   
        {\nablah \cdot(\vh \otimes \vh)(\tau)}}}_{L^2(\R^3)} 
        & (|\alphah| = 1)
    \end{dcases}
    \\
    \leq{}&
    \begin{dcases}
        C t^{  { \frac{1}{2} } }
        \sup_{\frac{t}{2}\leq \tau \leq t}
        \sum_{m=0}^1
        \n{{   
        \p_{x_3}^{\alpha_3+m}
        {\nablah \cdot(\vh \otimes \vh)(\tau)}}}_{L^2(\R^3)}
        & (|\alphah| = 0)\\
        C
        \sup_{\frac{t}{2}\leq \tau \leq t}
        \n{{(\vh \otimes \vh)(\tau)}}_{\dot{\mathcal{B}}_{2,1}^{1,\frac{1}{2}} (\R^3)}  
        & (|\alphah| = 1)
    \end{dcases}
    \\
    \leq{}&
    \begin{dcases}
        \begin{aligned}
        C 
        t^{  { \frac{1}{2} } }
        \sup_{\frac{t}{2}\leq \tau \leq t}
        \sum_{m=0}^{\alpha_3+1}
        &
        \left(\n{\nablah \p_{x_3}^{m} \vh(\tau)}_{L^2(\R^3)} \n{ \vh(\tau)}_{L^{\infty}(\R^3)}\right.
        \\
        &+
        \left.
        \n{\nablah  \vh(\tau)}_{L^{\infty}(\R^3)} \n{ \p_{x_3}^{m}\vh(\tau)}_{L^2(\R^3)}
        \right)
        \end{aligned}
        & (|\alphah| = 0)\\
        C
        \sp{\sup_{\frac{t}{2}\leq \tau \leq t}\n{{\vh(\tau)}}_{\dot{\mathcal{B}}_{2,1}^{1,\frac{1}{2}}(\R^3)  } }^2 
        & (|\alphah| = 1) 
    \end{dcases}
    \\
    \leq{}&
    C\n{(v,\theta)}^{2}_{Y_{\varepsilon}(0,T)}t^{-1-\frac{1+|\alphah|}{2}}.
\end{align}
Hence, we have
\begin{align}
    I_{2,p}(t)
    \leq
    I_{2,2}(t)^{\frac{2}{p}}
    I_{2,\infty}(t)^{1-\frac{2}{p}}
    \leq
    C\n{(v,\theta)}_{Y_{\varepsilon}(0,T)}^2t^{-(1-\frac{1}{p})-\frac{1+|\alphah|}{2}}.
\end{align}

\noindent
{\it Step 2. Estimates of $\{\mathcal{D}_{\pm,j}^{\rm vel,h}[v,\theta]\}_{j=2,6}$, $\{\mathcal{D}_{\pm,j}^{\rm temp}[v,\theta]\}_{j=2,6}$.} 

We only focus on the estimates of $\mathcal{D}_{\pm,2}^{\rm vel,h}[v,\theta]$ as the others are treated similarly.
It follows that  
\begin{align}
    &
    \n{  \nabla^{\alpha}  
    \mathcal{D}_{\pm,2}^{\rm vel,h}[v,\theta](t)}_{L^p(\R^3)}  \\ 
    & \quad 
    \leq
    C
    \int_0^{\frac{t}{2}}
    (t-\tau)^{-(1-\frac{1}{p})-\frac{3}{4}(1-\frac{2}{p})-\frac{|\alphah|}{2}}
    \n{ \p_{x_3} (v_3 \vh)(\tau) }
    _{L^1_{\xh} W^{3+\alpha_3,1}_{x_3}(\R^3) } 
    d\tau \\  
    &
    \qquad  
    +
    C
    \n{
    \int_{\frac{t}{2}}^t
    e^{(t-\tau)\Deltah}
    e^{\pm i (t-\tau)\frac{|\nablah|}{|\nabla|}}
   \Grad^{\alpha}  
    \frac{\nablah}{|\nablah|}
    \frac{\partial_{x_3}^2}{|\nabla|^2}
    \sp{\frac{\nablah}{|\nablah|} \cdot \p_{x_3} (v_3 \vh) }  (\tau) 
    d\tau
    }_{L^p(\R^3)}\\
    & \quad
    =: J_{1,p}(t) + J_{2,p}(t).
\end{align}
For the estimate of $J_{1,p}(t)$, it holds from Corollary \ref{cor:disp} that
\begin{align}
    J_{1,p}(t)
    \leq{}&
    C
    t^{-(1-\frac{1}{p})-\frac{3}{4}(1-\frac{2}{p})-\frac{|\alphah|}{2}}
    \int_0^{\frac{t}{2}} 
    \sum_{k=1}^{4+\alpha_3} \sum_{\ell=0}^{k}
    \n{\partial_{x_3}^{k-\ell}v_3(\tau)}_{L^2(\R^3)}
    \n{\partial_{x_3}^{\ell}\vh(\tau)}_{L^2(\R^3)} 
    d\tau \\
    \leq{}&
       C
    \n{(v,\theta)}_{Y_{\varepsilon}(0,T)}^2
    t^{-(1-\frac{1}{p})-\frac{3}{4}(1-\frac{2}{p})-\frac{|\alphah|}{2}}
    \int_0^{\frac{t}{2}}
    (1+\tau)^{-\f{13}{12}  }d\tau\\
    \leq{}&
    C
    \n{(v,\theta)}_{Y_{\varepsilon}(0,T)}^2
    t^{-(1-\frac{1}{p})-\frac{3}{4}(1-\frac{2}{p})-\frac{|\alphah|}{2}} .
\end{align}
Next, we consider the estimate of $J_{2,p}(t)$. In case of $p=2$, we see
\begin{align}
    J_{2,2}(t) 
    \leq{}&
    C
    \int_{\frac{t}{2}}^t
    (t-\tau)^{-\frac{|\alphah|}{2}}
    \n{ 
    \p_{x_3}^{\alpha_3+1}
    (v_3 \vh)(\tau)}_{L^2(\R^3)}d\tau\\ 
    \leq{}&
    C
    \n{(v,\theta)}_{Y_{\varepsilon}(0,T)}^2
    \int_{\frac{t}{2}}^t 
    (t-\tau)^{-\frac{|\alphah|}{2}}(1+\tau)^{-\frac{7}{4}}d\tau\\
    \leq{}&
    C
    \n{(v,\theta)}_{Y_{\varepsilon}(0,T)}^2
    t^{-\frac{3}{4}-\frac{|\alphah|}{2}}.  
\end{align}
In case of $p=\infty$, we see that
\begin{align}
    J_{2,\infty}(t) 
    \leq{}&
    C
    \sum_{j,k\in\mathbb{Z}}
    2^j
    2^{\frac{1}{2}k}
    \int_{\frac{t}{2}}^t
    e^{-c2^{2j}(t-\tau)}
    \n{\Deltahh_j\Deltav_k 
    \p_{x_3}^{\alpha_3+1}
    \nablah^{\alphah}  
    (v_3 \vh)(\tau)
    }_{L^2(\R^3)}  d\tau  \\
    \leq{}&
    C
    \int_{\frac{t}{2}}^t
    (t-\tau)^{-\f{1}{2}}
    \sum_{k \in \Z} 
    2^{\frac{1}{2}k} 
    \sup_{j\in \Z}
    \n{\Deltahh_j\Deltav_k 
    \p_{x_3}^{\alpha_3+1}
    \nablah^{\alphah}  
    (v_3 \vh)(\tau)
    }_{L^2(\R^3)}   d\tau  \\
    \leq{}&
    C
    \int_{\frac{t}{2}}^t
    (t-\tau)^{-\f{1}{2}}  
    \sup_{j,k\in \Z} 
    \n{\Deltahh_j\Deltav_k 
    \p_{x_3}^{\alpha_3+1}
    \nablah^{\alphah}  
    (v_3 \vh)(\tau)
    }_{L^2(\R^3)} ^{\f{1}{2}}  \\
    & \qquad 
    \times  
    \sup_{j,k\in \Z} 
    \n{\Deltahh_j\Deltav_k 
    \p_{x_3}^{\alpha_3+2}
    \nablah^{\alphah}  
    (v_3 \vh)(\tau)
    }_{L^2(\R^3)} ^{\f{1}{2}} d\tau   \\
    \leq{}&
    C\n{(v,\theta)}_{Y_{\varepsilon}(0,T)}^2
    \int_{\frac{t}{2}}^t
    (t-\tau)^{-\f{1}{2}}  
    (1+\tau)^{ -\f{7}{4}-\f{|\alphah|}{2}    } 
    d\tau \\
    \leq{}&
    C
    \n{(v,\theta)}_{Y_{\varepsilon}(0,T)}^2
    t^{-\frac{5}{4}-\frac{|\alphah|}{2}}.  
\end{align}
Interpolating between the above two estimates gives
\begin{align}
    J_{2,p}(t)
    \leq
    J_{2,2}(t)^{\frac{2}{p}}
    J_{2,\infty}(t)^{1-\frac{2}{p}}
    \leq
    C\n{(v,\theta)}_{Y_{\varepsilon}(0,T)}^2
    t^{-(1-\frac{1}{p})-\frac{|\alphah|}{2} -\f{1}{4}  }
\end{align}
for all $2\leq p \leq \infty$. {Here, we remark that the decay estimate of $J_{2,p}(t)$ should be responsible for the constraint of $A_0(p)$ defined in Theorem \ref{main-thm}. }

\noindent
{\it Step 3. Estimates of $\{\mathcal{D}_{\pm,j}^{\rm vel,h}[v,\theta]\}_{j=3,4}$, $\{\mathcal{D}_{\pm,j}^{\rm temp}[v,\theta]\}_{j=3,4}$, and $\{\mathcal{D}_{\pm,j}^{\rm vel,v}[v,\theta]\}_{j=2,3,4,6}$.}

We only focus on the estimates of $\mathcal{D}_{\pm,3}^{\rm vel,h}[v,\theta]$ as the others are treated similarly. 
It follows that 
\begin{align}
    & 
    \n{  \nabla^{\alpha}  
    \mathcal{D}_{\pm,3}^{\rm vel,h}[v,\theta](t)}_{L^p(\R^3)}  \\ 
    & \quad 
    \leq
    C
    \int_0^{\frac{t}{2}}
    (t-\tau)^{-(1-\frac{1}{p})-\frac{3}{4}(1-\frac{2}{p})-\frac{1+|\alphah|}{2}}
    \n{ (v_3 \vh)(\tau) } 
    _{L^1_{\xh} W^{3+\alpha_3,1}_{x_3}(\R^3) } 
    d\tau \\  
    &
    \qquad  
    +
    C
    \n{
    \int_{\frac{t}{2}}^t
    e^{(t-\tau)\Deltah}
    e^{\pm i (t-\tau)\frac{|\nablah|}{|\nabla|}}
   \Grad^{\alpha}  
    \frac{\nablah}{|\nablah|}
    \frac{\partial_{x_3}}{|\nabla|}
   \nablah\cdot (v_3 \vh)
      (\tau) 
    d\tau
    }_{L^p(\R^3)}\\
    &\quad 
    =: K_{1,p}(t) + K_{2,p}(t).
\end{align}
For the estimate of $K_{1,p}(t)$, it holds by Corollary \ref{cor:disp} that
\begin{align}
    K_{1,p}(t)
    \leq{}&
    C
    t^{-(1-\frac{1}{p})-\frac{3}{4}(1-\frac{2}{p})-\frac{1+|\alphah|}{2}}
    \int_0^{\frac{t}{2}} 
    \sum_{k=0}^{3+\alpha_3} \sum_{\ell=0}^{k}
    \n{\partial_{x_3}^{k-\ell}v_3(\tau)}_{L^2(\R^3)}
    \n{\partial_{x_3}^{\ell}\vh(\tau)}_{L^2(\R^3)} 
    d\tau \\
    \leq{}&
       C
    \n{(v,\theta)}_{Y_{\varepsilon}(0,T)}^2
    t^{-(1-\frac{1}{p})-\frac{3}{4}(1-\frac{2}{p})-\frac{1+|\alphah|}{2}}
    \int_0^{\frac{t}{2}}
    (1+\tau)^{-\f{13}{12}  }d\tau\\
    \leq{}&
    C
    \n{(v,\theta)}_{Y_{\varepsilon}(0,T)}^2
    t^{-(1-\frac{1}{p})-\frac{3}{4}(1-\frac{2}{p})-\frac{1+|\alphah|}{2}} .
\end{align}
Next, we consider the estimate of $K_{2,p}(t)$. In case of $p=2$, we see 
\begin{align}
    K_{2,2}(t) 
    \leq{}&
    C
    \int_{\frac{t}{2}}^t
    (t-\tau)^{-\frac{|\alphah|}{2}}
    \n{ 
    \Grad^{\alpha}
    (v_3 \vh)(\tau)}_{L^2(\R^3)}d\tau\\ 
    \leq{}&
    C
    \n{(v,\theta)}_{Y_{\varepsilon}(0,T)}^2
    \int_{\frac{t}{2}}^t 
    (t-\tau)^{-\frac{|\alphah|}{2}}
    (1+\tau)^{-\frac{7}{4} -\f{|\alphah|}{2}  }   d\tau  \\
    \leq{}&
    C
    \n{(v,\theta)}_{Y_{\varepsilon}(0,T)}^2
    t^{-\frac{3}{4}-\frac{1+|\alphah|}{2}}.   
\end{align}
In case of $p=\infty$, we invoke Lemma \ref{lemm:para} again and calculate as 
\begin{align}
    K_{2,\infty}(t)
    \leq{}&
    C
    \sum_{j,k\in\mathbb{Z}}
    { 2^j}
    2^{\frac{1}{2}k}
    \int_{\frac{t}{2}}^t
    \n{\Deltahh_j\Deltav_k e^{(t-\tau)\Deltah}
    {   \nabla^{\alpha} } 
     \nablah 
    \cdot(v_3 \vh)(\tau)   
    }_{L^2(\R^3)}d\tau\\
    \leq{}&
    C
    \sum_{j,k\in\mathbb{Z}}
    { 2^{(1+|\alphah|)j}}
    2^{\frac{1}{2}k}
    \int_{\frac{t}{2}}^t
    e^{-c2^{2j}(t-\tau)}
    \n{\Deltahh_j\Deltav_k 
    {   
    \p_{x_3}^{\alpha_3}
    {\nablah \cdot(v_3 \vh)(\tau)}}}_{L^2(\R^3)}d\tau\\
    \leq{}& 
    \begin{dcases}
        C t^{  { \frac{1}{2} } }
        \sup_{\frac{t}{2}\leq \tau \leq t}
        \sup_{j \in \mathbb{Z}}
        \sum_{k \in \mathbb{Z}}
        2^{\frac{1}{2}k}
        \n{\Deltahh_j\Deltav_k 
        {   
        \p_{x_3}^{\alpha_3}
        {\nablah \cdot(v_3 \vh)(\tau)}}}_{L^2(\R^3)} 
        & (|\alphah| = 0)\\
        C
        \sup_{\frac{t}{2}\leq \tau \leq t}
        \sum_{j,k\in \mathbb{Z}}
        2^{\frac{1}{2}k}
        \n{\Deltahh_j\Deltav_k 
        {   
        {\nablah \cdot(v_3 \vh)(\tau)}}}_{L^2(\R^3)} 
        & (|\alphah| = 1)
    \end{dcases}
    \\
    \leq{}&
    \begin{dcases}
        C t^{  { \frac{1}{2} } }
        \sup_{\frac{t}{2}\leq \tau \leq t}
        \sum_{m=0}^1
        \n{{   
        \p_{x_3}^{\alpha_3+m}
        {\nablah \cdot(v_3 \vh)(\tau)}}}_{L^2(\R^3)}
        & (|\alphah| = 0)\\
        C
        \sup_{\frac{t}{2}\leq \tau \leq t}
        \n{{(v_3 \vh) (\tau)}}_{\dot{\mathcal{B}}_{2,1}^{1,\frac{1}{2}} (\R^3)}  
        & (|\alphah| = 1)
    \end{dcases}
    \\
    \leq{}&
    \begin{dcases}
        \begin{aligned}
        C 
        t^{  { \frac{1}{2} } }
        \sup_{\frac{t}{2}\leq \tau \leq t}
        \sum_{m=0}^{\alpha_3+1}
        &
        \left(\n{\nablah \p_{x_3}^{m} v_3(\tau)}_{L^2(\R^3)} \n{ \vh(\tau)}_{L^{\infty}(\R^3)}\right. 
        \\
        &+
        \left.
        \n{\nablah  v_3(\tau)}_{L^{\infty}(\R^3)} \n{ \p_{x_3}^{m}\vh(\tau)}_{L^2(\R^3)}
        \right)
        \end{aligned}
        & (|\alphah| = 0)\\
        C
        \sp{\sup_{\frac{t}{2}\leq \tau \leq t}\n{{v_3(\tau)}}_{\dot{\mathcal{B}}_{2,1}^{1,\frac{1}{2}} (\R^3)} } 
         \sp{\sup_{\frac{t}{2}\leq \tau \leq t}\n{{\vh(\tau)}}_{\dot{\mathcal{B}}_{2,1}^{1,\frac{1}{2}} (\R^3)} }
        & (|\alphah| = 1) 
    \end{dcases} \\
    \leq{}&
    \begin{dcases}
         C\n{(v,\theta)}^{2}_{Y_{\varepsilon}(0,T)}
         t^{-\f{7}{4}}
        & (|\alphah| = 0)\\
        C\n{(v,\theta)}^{2}_{Y_{\varepsilon}(0,T)}
         t^{-\f{9}{4}}  
        & (|\alphah| = 1)
    \end{dcases}
    \\
    ={}&
    C\n{(v,\theta)}^{2}_{Y_{\varepsilon}(0,T)}t^{-\f{5}{4}-\frac{1+|\alphah|}{2}}. 
\end{align}
Application of interpolation inequality yields 
\begin{align}
    K_{2,p}(t)
    \leq
    K_{2,2}(t)^{\frac{2}{p}}
    K_{2,\infty}(t)^{1-\frac{2}{p}}
    \leq
    C\n{(v,\theta)}_{Y_{\varepsilon}(0,T)}^2
    t^{-(1-\frac{1}{p})-\frac{1+|\alphah|}{2} -\f{1}{4}  }
\end{align}
for all $2\leq p \leq \infty$. 
\end{proof}

\subsection{Proof of Theorem \ref{main-thm}}
Let $(v,\theta)$ be the global solution to \eqref{eq:B2} for given initial datum $(v_0,\theta_0) \in X^{8,4}(\R^3)$ satisfying the smallness condition on $H^8(\R^3)$ that is required in Lemma \ref{lemm:GWP-Sob} with $m=8$.
In this subsection, we complete the proof of Theorem \ref{main-thm} by the bootstrapping argument.
\begin{proof}[Proof of Theorem \ref{main-thm}]
It follows from Sobolev embeddings and Lemma \ref{lemm:GWP-Sob} that 
\begin{align}\label{Y01}
    \n{(v,\theta)}_{Y_{\varepsilon}(0,2)}
    \leq 
    C
    \sup_{0 \leq t \leq 2}
    \n{(v,\theta)(t)}_{H^8(\R^3)}
    \leq 
    C_1
    \n{(v_0,\theta_0)}_{X^{8,4}(\R^3)}
\end{align}
for some positive constant $C_1$.
It follows from Theorem \ref{thm:lin} that there exists a positive constant $C_{\ep,2}$ such that 
\begin{align}\label{Y-lin}
  \n{ 
  (v^{\rm lin},  \theta^{\rm lin} )
  }_{ Y_{\ep}(0,\infty) }
  \leq 
  C_{\ep,2}
  \n{
  (v_0,\theta_0)}_{X^{8,4}(\R^3)}. 
\end{align}
Let $T^*$ be the suprimum of $t > 0$ satisfying 
\begin{align}\label{pro:hyp}
    \n{ (v,\theta) }_{ Y_{\varepsilon}(0,t) }
    \leq 
    2 ( C_1 + C_{\ep,2} )
    \n{(v_0,\theta_0)}_{X^{8,4}(\R^3)}. 
\end{align}
Suppose by contradiction that $T^* < \infty$.
From the definition of $C_1$, we see that $T^*>2$.
Let $2<T<T^*$.
Then, we see by all lemmas in the previous subsection, \eqref{Y01}, and \eqref{Y-lin} that 
\begin{align}
    \n{ (v,\theta) }_{ Y_{\varepsilon}(0,T) }
    \leq{}&
    \n{ (v,\theta) }_{ Y_{\varepsilon}(0,2) }
    +
    \n{ (v,\theta) }_{ Y_{\varepsilon}(2,T) }\\
    \leq{}&
    \n{ (v,\theta) }_{ Y_{\varepsilon}(0,2) }
    +
    \n{ (v^{\rm lin},\theta^{\rm lin}) }_{ Y_{\varepsilon}(0,\infty) }\\
    &
    +
    \sum_{j=1}^6
    \n{
    \sp{
    \mathcal{D}_{\pm,j}^{\rm vel,h}[v,\theta],
    \mathcal{D}_{\pm,j}^{\rm vel,v}[v,\theta],
    \mathcal{D}_{\pm,j}^{\rm temp}[v,\theta]
    }
    }_{Y_{\varepsilon}(2,T)}\\
    \leq{}&
    ( C_1 + C_{\ep,2} )
    \n{(v_0,\theta_0)}_{X^{8,4}(\R^3)}
    +
    C_2
    \n{ (v,\theta) }_{ Y_{\varepsilon}(0,T) }^2\\
    \leq{}&
    ( C_1 + C_{\ep,2} )
    \n{(v_0,\theta_0)}_{X^{8,4}(\R^3)}
    +
    4
    C_2( C_1 + C_{\ep,2} )^2\n{(v_0,\theta_0)}_{X^{8,4}(\R^3)}^2
\end{align}
for some positive constant $C_2$. 
Choosing the initial datum so small that 
\begin{align}
    \n{(v_0,\theta_0)}_{X^{8,4}(\R^3)}
    \leq 
    \frac{1}{8C_2( C_1 + C_{\ep,2} )},
\end{align}
we have 
\begin{align}
    \n{ (v,\theta) }_{ Y_{\varepsilon}(0,T) }
    \leq 
    \frac{3}{2} 
    ( C_1 + C_{\ep,2} )
    \n{(v_0,\theta_0)}_{X^{8,4}(\R^3)},
\end{align}
which leads a contradiction to the definition of $T^*$.
Thus, we have $T^*=\infty$ and obtain 
\begin{align}
    \n{ (v,\theta) }_{ Y_{\varepsilon}(0,\infty) }
    \leq 
    2 ( C_1 + C_{\ep,2} )
    \n{(v_0,\theta_0)}_{X^{8,4}(\R^3)}.
\end{align}
From the definition of $Y_{\varepsilon}(0,\infty)$-norm, we obtain all desired decay estimates and complete the proof.
\end{proof}

\vspace{10mm}
{\bf{Acknowledgements}} 
M. Fujii was supported by Grant-in-Aid for Research Activity Start-up, Grant Number JP23K19011.
Y. Li was supported by Natural Science Foundation of Anhui Province under grant number 2408085MA018, Natural Science Research Project in Universities of Anhui Province under grant number 2024AH050055, National Natural Science Foundation of China under grant number 12001003; he sincerely thanks Professor Yongzhong Sun for patient guidance and encouragement.

\vspace{4mm}
{\bf{Data Availability}} Data sharing is not applicable to this article as no datasets were generated or analyzed
during the current study.
\vspace{4mm}

{\bf{Conflicts of interest}} All authors certify that there are no conflicts of interest for this work.


\begin{thebibliography}{100}

\bib{Bah-Che-Dan-11}{book}{
   author={Bahouri, Hajer},
   author={Chemin, Jean-Yves},
   author={Danchin, Rapha\"{e}l},
   title={Fourier analysis and nonlinear partial differential equations},
   volume={343},
   publisher={Springer, Heidelberg},
   date={2011},
   pages={xvi+523},
}




\bib{Ben-Pan-Wu-22}{article}{
   author={Ben Said, Oussama},
   author={Pandey, Uddhaba Raj},
   author={Wu, J.}, 
   title={The stabilizing effect of the temperature on buoyancy-driven fluids},
   journal={Indiana Univ. Math. J.},
   volume={71},
   date={2022},
  pages={2605-2645},
}


\bib{Cha-06}{article}{
   author={Chae, D.}, 
   title={Global regularity for the 2D Boussinesq equations with partial viscosity terms},
   journal={Adv. Math.},
   volume={203},
   date={2006},
  pages={497-513},
}


\bib{CDGG01}{article}{
   author={Chemin, Jean-Yves},
   author={Desjardins, B.},
   author={Gallagher, Isabelle},
   author={Grenier, Emmanuel},
   title={Fluids with anisotropic viscosity},
   journal={M2AN Math. Model. Numer. Anal.},
   volume={34},
   date={2000},
  pages={315-335},
}



\bib{CZ07}{article}{
   author={Chemin, Jean-Yves},
   author={Zhang, Ping},
   title={On the global wellposedness to the $3$-D incompressible anisotropic Navier--Stokes equations},
   journal={Comm. Math. Phys.},
   volume={272},
   date={2007},
  pages={529-566},
}











\bib{Den-Wu-Zha-21}{article}{
   author={Deng, W.},
   author={Wu, J.}, 
   author={Zhang, P.},
   title={Stability of Couette flow for $2$D Boussinesq system with vertical dissipation},
   journal={J. Funct. Anal.},
   volume={281},
   date={2021},
  pages={Paper No. 109255, 40 pp.},
}



\bib{Don-Wu-Xu-Zhu-21}{article}{
   author={Dong, B.},
   author={Wu, J.}, 
   author={Xu, X.},
   author={Zhu, N.},
   title={Stability and exponential decay for the $2$D anisotropic Boussinesq equations with horizontal dissipation},
   journal={Calc. Var. Partial Differential Equations.},
   volume={60},
   date={2021},
  pages={Paper No. 116, 21 pp.},
}



\bib{F21}{article}{
   author={Fujii, Mikihiro}, 
   title={Large time behavior of solutions to the $3$D anisotropic Navier--Stokes equation},
   journal={Nonlinear Anal. Real World Appl.},
   volume={71},
   date={2023},
  pages={Paper No. 103821, 33 pp.},
}



\bib{FL-24}{article}{
   author={Fujii, Mikihiro}, 
   author={Li, Yang},
   title={Large time behavior for solutions to the anisotropic Navier--Stokes equations in a $3$D half-space},
   journal={Preprint},
}



\bib{FW}{article}{
 author={Fujii, Mikihiro},
 author={Watanabe, Keiichi},
 title={Compressible Navier--Stokes--Coriolis system in critical Besov spaces},
 journal={Preprint},
}


\bib{Hou-Li-05}{article}{
   author={Hou, T.Y.}, 
   author={Li, C.},
   title={Global well-posedness of the viscous Boussinesq equations},
   journal={Discrete Contin. Dyn. Syst.},
   volume={12},
   date={2005},
 pages={1-12},
}






\bib{Ift-99}{article}{
   author={Iftimie, D.},
   title={The resolution of the Navier-Stokes equations in anisotropic spaces},
   journal={Rev. Mat. Iberoamericana.},
   volume={15},
   date={1999},
  pages={1-36},
}





\bib{Lee-Tak-17}{article}{
   author={Lee, S.},
   author={Takada, R.},
   title={Dispersive estimates for the stably stratified Boussinesq equations},
   journal={Indiana Univ. Math. J.},
   volume={66},
   date={2017},
  pages={2037-2070},
}





\bib{Li-Pai-Zha-20}{article}{
   author={Liu, Yanlin},
   author={Paicu, Marius},
   author={Zhang, Ping},
   title={Global well-posedness of 3-D anisotropic Navier-Stokes system with
   small unidirectional derivative},
   journal={Arch. Ration. Mech. Anal.},
   volume={238},
   date={2020},
   pages={805--843},
}




\bib{JTW23}{article}{
   author={Ji, Ruihong},
   author={Tian, Ling},
   author={Wu, Jiahong}, 
   title={$3$D anisotropic Navier--Stokes equations in
$\mathbb{T}^2 \times \R$: stability and large-time behaviour},
   journal={Nonlinearity.},
   volume={36},
   date={2023},
  pages={3219-3237},
}



\bib{JWY21}{article}{
   author={Ji, Ruihong},
   author={Wu, Jiahong}, 
   author={Yang, Wanrong},
   title={Stability and optimal decay for the $3$D Navier--Stokes equations with horizontal dissipation},
   journal={J. Differential Equations.},
   volume={290},
   date={2021},
  pages={57-77},
}




\bib{Ji-Yan-Wu-22}{article}{
   author={Ji, R.},
   author={Yan, L.},
   author={Wu, J.}, 
   title={Optimal decay for the $3$D anisotropic Boussinesq equations near the hydrostatic balance},
   journal={Calc. Var. Partial Differential Equations.},
   volume={61},
   date={2022},
  pages={Paper No. 136, 34 pp.},
}










\bib{Kan-Lee-Ngu-24}{article}{
   author={Kang, Kyungkeun},
   author={Lee, Jihoon},
   author={Nguyen, Dinh Duong}, 
   title={Global well-posedness and stability of the $2$D Boussinesq equations with partial dissipation near a hydrostatic equilibrium},
   journal={J. Differential Equations.},
   volume={393},
   date={2024},
  pages={1-57},
}



\bib{Koh-Lee-Tak-14}{article}{
   author={Koh, Y.},
   author={Lee, S.},
   author={Takada, R.}, 
   title={Strichartz estimates for the Euler equations in the rotational framework},
   journal={J. Differential Equations.},
   volume={256},
   date={2014},
  pages={707-744},
}














\bib{Lai-Wu-Zho-21}{article}{
   author={Lai, S.},
   author={Wu, J.},
   author={Zhong, Y.}, 
   title={Stability and large-time behavior of the $2$D Boussinesq equations with partial dissipation},
   journal={J. Differential Equations.},
   volume={271},
   date={2021},
  pages={764-796},
}




\bib{Lar-Lun-Titi-13}{article}{
   author={Larios, A.},
   author={Lunasin, E.},
   author={Titi, Edriss S.}, 
   title={Global well-posedness for the $2$D Boussinesq system with anisotropic 
viscosity and without heat diffusion},
   journal={J. Differential Equations.},
   volume={255},
   date={2013},
  pages={2636-2654},
}











\bib{Li-Titi-16}{article}{
   author={Li, J.},
   author={Titi, Edriss S.},
   title={Global well-posedness of the $2$D Boussinesq equations with vertical dissipation},
   journal={Arch. 
Ration. Mech. Anal.},
   volume={220},
   date={2016},
  pages={983-1001},
}









\bib{L22}{article}{
   author={Li, Yang}, 
   title={Large time behavior of the solutions to $3$D incompressible MHD system with horizontal dissipation or horizontal magnetic diffusion},
   journal={Calc. Var. Partial Differential Equations.},
    volume={63},
   date={2024},
  pages={Paper No. 43, 65 pp},
}



\bib{Pai05}{article}{
   author={Paicu, Marius},
   title={\'{E}quation anisotrope de Navier--Stokes dans des espaces critiques},
   journal={Rev. Mat. Iber.},
   volume={21},
   date={2005},
  pages={179-235},
}








\bib{Tao-Wu-Zha0-Zhen-20}{article}{
author={Tao, L.},
   author={Wu, J.}, 
   author={Zhao, K.},
   author={Zheng, X.},
   title={Stability near hydrostatic equilibrium to the $2$D Boussinesq equations without thermal diffusion},
   journal={Arch. Ration. Mech. Anal.},
   volume={237},
   date={2020},
  pages={585-630},
}





\bib{Tak-19}{article}{
   author={Takada, Ryo},
   title={Strongly stratified limit for the 3D inviscid Boussinesq
   equations},
   journal={Arch. Ration. Mech. Anal.},
   volume={232},
   date={2019},
   pages={1475--1503},
}





\bib{Wu-Zha-21}{article}{
   author={Wu, J.}, 
   author={Zhang, Q.},
   title={Stability and optimal decay for a system of $3$D anisotropic Boussinesq equations},
   journal={Nonlinearity.},
   volume={34},
   date={2021},
  pages={5456-5484},
}



\bib{XZ22}{article}{
   author={Xu, Li},
   author={Zhang, Ping}, 
   title={Enhanced dissipation for the third component of $3$D anisotropic Navier--Stokes equations},
   journal={J. Differential Equations.},
   volume={335},
   date={2022},
  pages={464–496},
}



\bib{ZF1}{article}{
   author={Zhang, Ting},
   author={Fang, Daoyuan},
   title={Global wellposed problem for the 3-D incompressible anisotropic Navier-Stokes equations},
   journal={J. Math. Pures Appl.},
   volume={90},
   date={2008},
  pages={413-449},
}






\bib{Zil-21}{article}{
   author={Zillinger, Christian}, 
   title={On enhanced dissipation for the Boussinesq equations},
   journal={J. Differential Equations.},
   volume={282},
   date={2021},
  pages={407-445},
}

\end{thebibliography}
\end{document}